\documentclass[12pt]{article}

\usepackage{a4wide}
\usepackage[french,english]{babel}
\usepackage[utf8]{inputenc}
\usepackage[T1]{fontenc}
\usepackage{amsthm}
\usepackage{amssymb}
\usepackage{amsmath}
\usepackage{stmaryrd}
\usepackage{imakeidx}
\usepackage{graphicx}
\usepackage{caption}
\usepackage{mathtools}
\theoremstyle{plain}
\usepackage[colorlinks=true,breaklinks=true,linkcolor=black]{hyperref} 
\usepackage{amsfonts}
\usepackage{appendix}
\usepackage[normalem]{ulem}
\usepackage{cancel}
\usepackage{comment}
\usepackage{enumerate}
\usepackage{tikz-cd}
\usepackage[
ordering=Kac,
edge/.style=blue,
indefinite-edge={draw=green,fill=white,densely dashed},
indefinite-edge-ratio=5,
mark=o,
root-radius=.06cm]
{dynkin-diagrams}

\newtheorem{Theorem}{Theorem}[section] 
\newtheorem{Proposition}[Theorem]{Proposition}
\newtheorem{Corollary}[Theorem]{Corollary}
\newtheorem{Definition}[Theorem]{Definition}

\newtheorem{Lemma}[Theorem]{Lemma}

\newtheorem{Remark}[Theorem]{Remark}

\newtheorem{Definition/Proposition}[Theorem]{Definition/Proposition}
\newtheorem{Theorem*}{Theorem}
\newtheorem{Corollary*}[Theorem*]{Corollary}
\usepackage{dsfont}
\date{}
\makeindex[columns=3]
\usepackage{mathrsfs}
\usepackage[all]{xy}

\usepackage{float}

\makeatletter
\renewcommand\theequation%
{\thesection.\arabic{equation}}
\@addtoreset{equation}{section}
\makeatother

\theoremstyle{Definition}
\newtheorem{proposition*}[thmbis]{Proposition}

\theoremstyle{Definition}
\newtheorem{proposition**}[thmter]{Proposition}

\makeatletter \@addtoreset{figure}{section}\makeatother

\newcommand{\R}{\mathbb{R}}
\newcommand{\A}{\mathbb{A}}

\newcommand{\N}{\mathbb{N}}
\newcommand{\Z}{\mathbb{Z}}

\newcommand{\I}{\mathcal{I}}
\newcommand{\G}{\mathbb{G}}
\newcommand{\TT}{\mathbb{T}}
\newcommand{\T}{\mathcal{T}}
\newcommand{\Id}{\mathrm{Id}}

\newcommand{\s}{\mathfrak{s}}

\newcommand{\VC}{\mathcal{V}}
\newcommand{\BC}{\mathcal{B}}

\newcommand{\qp}{\varpi}

\newcommand{\efface}[1]{}

\newcommand{\sph}{\mathrm{sph}}

\newcommand{\pr}{\mathrm{proj}}
\newcommand{\ve}{\mathrm{vert}}

\newcommand{\cC}{\mathcal{C}}

\newcommand{\cE}{\mathcal{E}}

\newcommand{\cH}{\mathcal{H}}

\newcommand{\cK}{\mathcal{K}}

\newcommand{\cM}{\mathcal{M}}

\newcommand{\cO}{\mathcal{O}}

\newcommand{\cP}{\mathcal{P}}

\newcommand{\cT}{\mathcal{T}}

\newcommand{\sT}{\mathscr{T}}

\newcommand{\cU}{\mathcal{U}}

\newcommand{\fc}{\mathfrak{c}}

\newcommand{\fd}{\mathfrak{d}}

\newcommand{\bk}{\mathds{k}}

\newcommand{\fs}{\mathfrak{s}}

\newcommand{\bx}{\mathbf{x}}

\newcommand{\by}{\mathbf{y}}

\newcommand{\bz}{\mathbf{z}}

\renewcommand{\phi}{\varphi}
\renewcommand{\emptyset}{\varnothing}

\renewcommand{\tilde}[1]{\widetilde{#1}}

\makeatletter
\def\Ddots{\mathinner{\mkern1mu\raise\p@
\vbox{\kern7\p@\hbox{.}}\mkern2mu
\raise4\p@\hbox{.}\mkern2mu\raise7\p@\hbox{.}\mkern1mu}}
\makeatother

\DeclareMathOperator{\dom}{Dom}
\DeclareMathOperator{\im}{Im}

\renewcommand{\hom}{\operatorname{Hom}}

\DeclareMathOperator{\sgn}{sgn}

\newcommand{\Inv}{\mathrm{Inv}}

\usepackage[left,pagewise]{lineno}

\hypersetup{
pdfauthor={Hébert, Philippe},
pdftitle={Paths and Kazhdan-Lusztig polynomials},
}

\title{On affine Kazhdan-Lusztig R-polynomials for Kac-Moody groups}
\author{Auguste \textsc{Hébert} \\Université de Lorraine, Institut Élie Cartan de Lorraine, F-54000 Nancy, France\\ UMR 7502,
auguste.hebert@univ-lorraine.fr \\  \and Paul \textsc{Philippe} \\ Université Jean Monnet, Institut Camille Jordan, F-42023 Saint-Étienne, France\\ UMR 5208, paul.philippe@univ-st-etienne.fr}

\begin{document}
\maketitle

\begin{abstract}

In \cite{muthiah2019double}, D. Muthiah proposed a strategy to define affine Kazhdan-Lusztig $R$-polynomials for Kac-Moody groups. Since then, Bardy-Panse, the first author and Rousseau  have introduced in \cite{bardy2025twin} the formalism of twin masures   and we have extended combinatorial results from affine root systems to general Kac-Moody root systems in \cite{hebert2024quantum}. In this paper, we use these results to explicitly define affine $R$-Kazhdan-Lusztig polynomials for Kac-Moody groups. The construction is based on a path model lifting to twin masures.  Conjecturally, these polynomials count the cardinality of  intersections of opposite affine Schubert cells, as in the case of reductive groups.
\end{abstract}

\section{Introduction}

\subsection{The affine reductive setup}

\paragraph{Reductive groups over Laurent polynomial rings}
Let $\G$ be a split reductive group scheme with the data of a Borel subgroup $\mathbb{B}$ containing a maximal torus $\mathbb{T}$. Let $W^v=N_{\G}(\mathbb{T})/\mathbb{T}$ be its vectorial Weyl group and $Y$ be its coweight lattice: $Y=\hom(\mathbb G_m,\mathbb{T})$.
The action of $W^v$ on $\mathbb{T}$ induces an action of $W^v$ on $Y$ and allows to form the semidirect product $\tilde{W}=Y\rtimes W^v$. This group called the extended affine Weyl group of $\G$, appears naturally in the geometry and the representation theory of $\G$ over discretely valued fields  or Laurent polynomial rings. A foundational work in this regard was done by N. Iwahori and H. Matsumoto in 1965 (\cite{iwahori1965bruhat}), when they exhibited a Bruhat decomposition of $\G(\mathbb Q_p)$ indexed by $\tilde{W}$. 

Let $\mathds k$ be a finite field and $\qp$ be a transcendental element over $\mathds k$. Let $G=\G(\mathds k(\!(\qp)\!))$, its integral points form the maximal compact subgroup $\G(\mathds k[\![\qp]\!])$  which contains the Iwahori subgroup, defined as $I_0=\{g\in K\mid g \mod \qp \in \mathbb B(\mathds k)\}$. The affine Weyl group is identified with the quotient $N_{G}(\mathbb T(\mathds k(\!(\qp)\!)))/\mathbb T(\mathds k[\![\qp]\!])$, and admits lift in $G$. Then, $G$ admits a decomposition in $I_0$ double cosets indexed by $\tilde{W}$, the Iwahori-Matsumoto Bruhat decomposition: $$G=\bigsqcup\limits_{\bx \in \tilde{W}}I_0 \bx I_0.$$
This decomposition also holds for $\G(\mathcal K)$ when $\mathcal K$ is any discretely valued field with residue field $\mathds k$. Actually it was first proven by Iwahori and Matsumoto for $\G(\mathbb Q_p)$. The group $\tilde{W}$ is a finite extension of a Coxeter group, it is  equipped with a Bruhat order which enables to describe the homogeneous space $G/I_0$, an infinite-dimensional variety over $\mathds k$ referred as the affine flag variety. We are interested by $R$-Kazhdan-Lusztig polynomials, which is an instance of the relation between the structure of $G/I_0$ and the combinatorics of $\tilde W$.

\paragraph{Reductive affine $R$-Kazhdan-Lusztig polynomials}
In $G=\mathbb G(\mathds k(\!(\qp)\!))$, the Iwahori group $I_0$ admits an opposite Iwahori subgroup $I_\infty=\{g\in \G(\mathds k[\qp^{-1}])\mid  g \mod \qp^{-1} \in \mathbb{B}^-(\mathds k) \}$ where $\mathbb{B}^-$ is the opposite Borel subgroup to $\mathbb{B}$. Then we have an Iwahori-Birkhoff decomposition: $$G=\bigsqcup\limits_{\by \in \tilde{W}}I_\infty \by I_0.$$ The affine $R$-Kazhdan-Lusztig polynomials $(R_{\bx,\by})_{\bx,\by\in \tilde{W}}$ are combinatorially defined polynomials (see \cite{kazhdan1979representations} or \cite[\S 5.3]{bjorner2005combinatorics} for a more recent expository) which verify , for all finite field $\mathds k$:
\begin{equation}\label{eq: RPolreductive_intro}|I_\infty \by I_0 \cap I_0 \bx I_0 /I_0|=R_{\bx,\by}(|\mathds k|).\end{equation}

This formula, in its affine form, is implicitly used by D. Kazhdan and G. Lusztig in \cite[\S 5]{kazhdan1980schubert}, and was proven by Z. Haddad (\cite{haddad1985coxeter}). A method to describe the combinatorics of these polynomials, and their appearance in Formula \eqref{eq: RPolreductive_intro}, is through a gallery model in the standard apartment of the Bruhat-Tits building associated to $G$. It is a particular instance of the techniques used by E. Milićević, Y. Naqvi, P. Schwer and A. Thomas (see \cite{milicevic2022gallery}, \cite{schwer2022shadows} where the authors describe similar subsets of $G/I_0$).

The affine $R$-polynomials are auxiliary polynomials in order to define the more famous $P$-Kazhdan-Lusztig polynomials, which are the cornerstone of Kazhdan-Lusztig theory, but they also seem to have a representation theoretic interpretation on their own (see \cite{eberhardt2023standard}).

\paragraph{Extension to Kac-Moody groups}
Replace $(\G,\TT)$ by a general split Kac-Moody group $\G$ with maximal split torus $\TT$, as defined by Tits in \cite{tits1987uniqueness}; its vectorial Weyl group may now be infinite. Let $Y=\hom(\G_m,\TT)$ be the cocharacter lattice of $(\G,\TT)$ and let $Y^+$ is the set of $W^v$-translates of the dominant elements of $Y$. When $\G$ is  reductive $Y^+=Y$ but otherwise $Y^+$ is a proper subcone of $Y$.  Set $W^+=Y^+\rtimes W^v$.  The Iwahori-Matsumoto decomposition no longer holds on $G=\G(\mathds k(\!(\qp)\!))$ (and neither on $\G(\mathcal K)$ for more general discretely valued fields $\mathcal K$) but it can be replaced by:
$$G^+=\bigsqcup\limits_{\qp^\lambda w\in W^+}I_0 \qp^\lambda wI_0, $$ where $G^+$ is some sub-semi-group of $G$ (see \cite[\S 3.4.2]{braverman2016iwahori} or \cite[\S 1.9]{bardy2016iwahori}). This weakened Iwahori-Matsumoto decomposition appears in related contexts, for example in the definition of the Iwahori-Hecke algebra of $\mathbb{G}(\cK)$, when $\cK$ is a non-Archimedean local field (see \cite{braverman2014affine} and \cite{bardy2016iwahori}). In this paper, we construct affine $R$-polynomials in the Kac-Moody setting, using the theory of masures.

\subsection{Affine $R$-Kazhdan-Lusztig polynomials in the Kac-Moody case}

\paragraph{Masures and twin masures}

Masures are generalizations of Bruhat-Tits buildings adapted to the study of Kac-Moody groups, introduced by S. Gaussent and G. Rousseau originally to compute MV-cycles through a path model \cite{gaussent2008kac}. Let $\omega$ be a non-trivial discrete  valuation on a field $\cK$. In order to study $\mathbb{G}(\cK)$, Rousseau  (generalizing the construction of \cite{gaussent2008kac}) defined a masure $\I_\omega=\I(\G,\cK,\omega)$ on which $G$ acts (see \cite{rousseau2016groupes}).  When $\G$ is reductive, $\I_\omega$ is the usual Bruhat-Tits building. Let $\A_\omega=Y\otimes \R$ be the standard apartment. When $\G$ is a reductive group then $W^+$ is a Coxeter group and $\A_\omega$ is a realization of the Coxeter complex of $W^+$. Therefore it has a simplicial structure such that $W^+$ acts transitively on the simplices of maximal dimension, which are called alcoves. When $\G$ is general Kac-Moody, there is no such simplicial structure but one can still define analogs of alcoves; no longer as simplices but as filters (sets of sets), and $W^+$ acts freely transitively on the set of (special) alcoves. We have $\I_\omega=\bigsqcup_{g\in G} g.\A_\omega$ (the dependency on $\omega$ is hidden in the fact that the action of $G$ on $\I_\omega$ depends on $\omega$). 

Let us now specify to the case of global function fields: we suppose that $\mathcal K$ is of the form $\mathds k(\qp)$, where $\qp$ is transcendental over $\mathds k$. Let $\omega_{\oplus}$ (resp. $\omega_{\ominus}$) be the valuation on $\cK$ such that $\omega_{\oplus}(\bk^* \qp)=\{1\}$ (resp. $\omega_{\ominus}(\bk^* \qp^{-1})=\{1\}$). Let $\I_{\oplus}=\I_{\omega_{\oplus}}$ and $\I_{\ominus}=\I_{\omega_{\ominus}}$.   Then one defines the Iwahori subgroups in terms of masures: $I_0$ (resp. $I_\infty$\index{i@$I_\infty$}) is the fixator of a particular alcove, named the fundamental alcove $C_0$ (resp. $C_\infty$) of $\A:=\A_{\omega_{\oplus}}$ (resp. $\A_\ominus=\A_{\omega_{\ominus}}$) in $G$. We then can see $G/I_0$ as a set of alcoves of $\I_{\oplus}$. When $\G$ is reductive,  the pair $(\I_{\oplus},\I_{\ominus})$ forms a twin building: there is a  codistance $d^*$ relating $\I_{\oplus}$ and $\I_{\ominus}$ (see \cite[5.8]{abramenko2008buldings}). One can also define a retraction $\rho_{C_\infty}:\I_{\oplus}\rightarrow \A$ centered at $C_\infty$: this is the unique $I_\infty$-invariant map whose restriction to $\A$ is the identity. We can then recover $d^*$ from $\rho_{C_\infty}$.  In the Kac-Moody case, no equivalent of the codistance is known for the moment. However, one can define $\rho_{C_\infty}$ on a subset of $\I$. A difficulty is that the set of definition of $\rho_{C_\infty}$ is hard to determine and very few is currently known on this set. 

Conjecturally (see \cite[4.4.1]{bardy2025twin}), this set   is large enough and we can use $\rho_{C_\infty}$ to  compute $|I_\infty \by I_{0}\cap I_{0} \bx I_{0}/I_{0}|$.  Elements of $I_\infty \by I_{0}\cap I_{0} \bx I_{0}/I_{0}$ are in bijection with the alcoves in $\rho_{C_0}^{-1}(C_\bx)\cap \rho_{C_\infty}^{-1}(C_\by)$. Each of these alcoves defines an open segment (the data of a segment and an alcove based at its end point) in $\I_\oplus$ and, when $\bx$ has finite stabiliser in $W^v$ (we say that it is spherical),  we can count them using the path model in $\A$ described below.

\paragraph{Affine Bruhat order on $W^+$}
When $G$ is reductive, the extended affine Weyl group $\tilde{W}$, which is a finite extension of the affine Weyl group $W^a$, naturally inherits a Bruhat order $\leq$ from the Coxeter group structure of $W^a$. This order plays a crucial roles in the study of  Kazhdan-Lusztig polynomials. For example, if $\bx,\by\in W^a$, then  $R_{\bx,\by}\neq 0$ if and only if $\by \leq \bx$ and it is conjectured that $R_{\bx,\by}$ only depends on the structure of the interval $[\by,\bx]=\{\bz\in \tilde{W}\mid \by\leq \bz\leq \bx\}$.   In the Kac-Moody setting, A. Braverman, D. Kazhdan and M. Patnaik propose in \cite[Appendix B2]{braverman2016iwahori} a definition of a preorder on $W^+$ which would replace the Bruhat order of $\tilde{W}$ and they conjecture that it is a partial order. In \cite{muthiah2018iwahori}, D. Muthiah extends the definition of this preorder to any Kac-Moody group $G$ and proves that it is actually an order. A. Welch proved in \cite[Theorem 3]{welch2022classification} that the intervals for this order are finite, when $G$ if non-twisted affine of ADE type and we generalized this result to any Kac-Moody groups in \cite[Corollary 1.2]{hebert2024quantum}.

\paragraph{$C_\infty$-Hecke open paths and $R$-Kazhdan-Lusztig polynomials}

In this paper we introduce $C_\infty$-Hecke open paths, they are the piece-wise affine paths which appear as images of open segments by $\rho_{C_\infty}$ in the masure $\I_\oplus$, but they can be defined combinatorially in $\A$ without reference to the masure. Non-open $C_\infty$-Hecke paths have already been introduced by Muthiah (\cite{muthiah2019double} where they are called $I_\infty$-Hecke paths) and developed by N. Bardy-Panse, the first-named author and Rousseau in \cite{bardy2025twin}. An open path is a non-open path with the additional data of a decoration, and an ending alcove (see Definition~\ref{Def: open paths}). These open paths permit to study $G^+/I_0$, whereas non-open paths are suited to the study of $G^+/K$, where $K$ is fixator of the vertex $0_\A$ of the fundamental alcove. In particular any $C_\infty$-Hecke open path has a type $\bx\in W^+$ (which we suppose spherical, that is to say that $\bx$ has finite stabilizer in $W^v$) and an end alcove $C_\by$ with $\by\in W^+$; we denote by $\cC^\infty_\bx(\by)$ the set of $C_\infty$-Hecke open paths of type $\bx$ and with end alcove $C_\by$. We start by looking at the relation between $C_\infty$-Hecke open paths and the affine Bruhat order. In particular, using the finiteness of the intervals, we obtain:

\begin{Proposition}(see Corollary~\ref{Corollary: finiteness open hecke paths})
Let $\bx,\by\in W^+$ with $\bx$ spherical. Then $\mathcal C^{\infty}_{\bx}(\by)$ is finite, and it is empty if $\by \not\leq \bx$.
\end{Proposition}

\paragraph{$R$-Kazhdan-Lusztig polynomials}
A given $C_\infty$-Hecke open path $\underline{\fc}$ of spherical type $\bx\in W^+$ lifts to open segments in $\I_\oplus$, and we show that the number of these lifts are obtained by the evaluation at $|\mathds k|$ of an integral polynomial $R_{\bx,\underline{\fc}}$, independent of the masure, for which we give an explicit formula in Theorem~\ref{Theorem : number of lifts spherical}. Summing up $R_{\bx,\underline{\fc}}$ upon all the paths of $C^\infty_\bx(\by)$ we are able to define an $R$-Kazhdan-Lusztig polynomial $R_{\bx,\by}$ associated to any pair $(\bx,\by)$ of $W^+$ such that $\bx$ is spherical. They therefore appear in the following result:
\begin{Theorem}
Let $\bx,\by\in W^+$ be such that $\bx$ is spherical. Let $A_{\bx,\by}(\I_\oplus)$ denote the set of alcoves $C\in \rho_{C_0}^{-1}(C_\bx)\cap \rho_{C_\infty}^{-1}(C_\by)$ such that the retraction $\rho_{C_\infty}$ is defined on the whole open segment $[0_\A,C]\subset \I_\oplus$. This is a finite set and:
\begin{equation}
    |A_{\bx,\by}(\I_\oplus)|=R_{\bx,\by}(|\mathds k|).
\end{equation}  
\end{Theorem}

By construction $|A_{\bx,\by}(\I_\oplus)|\leq|I_\infty \by I_0 \cap I_0\bx I_0/I_0|$ and upon Conjecture \cite[4.4.1]{bardy2025twin} this is an equality. Therefore Formula \eqref{eq: RPolreductive_intro} remains conjectural in this context, however these polynomials already have interesting combinatorial properties related to the affine Bruhat order on $W^+$. For example, we show that, if $\bx$ covers $\by$, then $R_{\bx,\by}(X)=X-1$ which is a standard property in the reductive setting.

Note that in the recent preprint \cite{patnaik2024local},  M. Patnaik obtains results which might be related to ours. He proves that when $\mathbb{G}$ is affine, in some completion $\hat{G}$ of $G$, a variant of the left hand side of \eqref{eq: RPolreductive_intro} is finite (see \cite[5.2]{patnaik2024local}). He proves an analogue (in his frameworks, which is slighlty different from ours) of a conjecture in \cite[4.4.1]{bardy2025twin}. However, he does not prove polynomiality results.  His techniques are very different from ours and do not use masures. It would be interesting to relate these two works.

\paragraph{Going further}
In classical Kazhdan-Lusztig theory, $R$-polynomials are auxiliary tools to compute the more famous $P$-Kazhdan-Lusztig polynomials, which play a significant role in geometric representation theory and representation theory of Lie algebras (see \cite{kazhdan1980schubert}). We would like to define the affine version of these $P$-polynomials in the general Kac-Moody setting, this requires that the $R$-polynomials satisfy an appropriate involutive formula in which the affine Bruhat length appears (see \cite[\S 1.2.4]{muthiah2019double}).

\paragraph{Organization of the paper}
In section~\ref{Section: recollection affine Bruhat order} we introduce the affine Weyl semi-group associated to a Kac-Moody root datum, and the affine Bruhat order. We then give a geometric interpretation of these objects.
In section~\ref{s_Cinfty_open_paths}, we introduce the notion of $C_\infty$-Hecke open path in the standard apartment. We establish properties of these paths with respect to the Bruhat order (Proposition~\ref{Proposition: open paths give chains}) and use it and the finiteness results on the affine Bruhat order to prove finiteness results on the set of $C_\infty$-open paths (see Corollary~\ref{Corollary: finiteness open hecke paths}).

We then turn, in section~\ref{s_lifts_paths}, to liftings of $C_\infty$-Hecke open paths in the masure, in order to construct affine $R$-polynomials. We start by introducing twin masures, and in subsection~\ref{subsection: lifts in masure} we explain our strategy to construct affine $R$-polynomials. We prove in subsection~\ref{ss_retraction_open_segments} that $C_\infty$-friendly open segments retract to $C_\infty$-Hecke open paths, and we give combinatorial formulas to count, for each $C_\infty$-Hecke open path of a given type, its number of lifts in subsection~\ref{ss_counting_lifts}. We conclude, in subsection~\ref{subsection: Definition Kazhdan-Lusztig Polynomials}, by defining the affine Kazhdan-Lusztig $R$-polynomials, and we compute $R$-polynomials of covers.

There is an index of notations at the end of the paper.

\paragraph{Acknowledgement} We thank Nicole Bardy-Panse, Stéphane Gaussent, Dinakar Muthiah and Guy Rousseau for insightful discussions and comments on the subject. We also thank Manish Patnaik for sharing his interesting paper \cite{patnaik2024local} and for  exchanges on the topic. 

\tableofcontents

\section{The standard apartment and the affine Bruhat order}\label{Section: recollection affine Bruhat order}

\subsection{Definitions and notations}

\paragraph{Kac-Moody root systems}
Let $\mathcal D = (A,X,Y,(\alpha_i)_{i \in I},(\alpha_i^\vee)_{i \in I})$ be a Kac-Moody root datum in the sense of \cite[\S 8]{remy2002groupes}. It is a quintuplet such that:
\begin{itemize}
    \item $I$ is a finite indexing set and $A=(a_{ij})_{(i,j)\in I\times I}$ is a generalized Cartan matrix.
    \item $X$ and $Y$ are two dual free $\mathbb Z$-modules of finite rank, we write $\langle , \rangle$ the duality bracket.  
    \item $(\alpha_i)_{i\in I}$ (resp. $(\alpha_i^\vee)_{i \in I}$) is a family of linearly independent elements of $X$ (resp. $Y$), the simple roots (resp. simple coroots).
    
    \item For all $(i,j) \in I^2$ we have $\langle \alpha_i^\vee,\alpha_j \rangle = a_{ij}$.
\end{itemize}

\paragraph{Vectorial Weyl group} For every $i \in I$ set $s_i \in \operatorname{Aut}_\mathbb Z(X):  x \mapsto x- \langle \alpha_i^\vee, x \rangle \alpha_i$. The generated group \index{w@$W^v$} $W^v=\langle s_i \mid i \in I \rangle$ is the \textbf{vectorial Weyl group} of the Kac-Moody root datum. The duality bracket $\langle Y,X\rangle$ induces a contragredient action of $W^v$ on $Y$, explicitly $s_i(y)=y-\langle y,\alpha_i\rangle \alpha_i^\vee$. By construction the duality bracket is then $W^v$-invariant.

The pair $(W^v,\{s_i\mid i \in I\})$ is a Coxeter system, in particular $W^v$ has a Bruhat order $<$ and a length function $\ell$ compatible with the Bruhat order.  The Weyl group $W^v$ is infinite if and only if the matrix $A$ is not of finite type, we will usually work in this setting.

\paragraph{Real roots} Let $\Phi=W^v.\{\alpha_i \mid i \in I \}$ be the set of real roots of $\mathcal D$, it is a, possibly infinite, root system (see \cite[1.2.2 Definition]{kumar2002kac}). 
In particular let $\Phi_+=\Phi\cap \oplus_{i\in I} \mathbb N \alpha_i$ be the set of positive real roots, then $\Phi=\Phi_+ \sqcup -\Phi_+$, we write $\Phi_-=-\Phi_+$ the set of negative roots.

The set $\Phi^\vee=W^v.\{\alpha_i^\vee \mid i \in I\}$ is the set of \textbf{coroots}, and its subset $\Phi^\vee_+=\Phi^\vee \cap \oplus_{i\in I} \mathbb N \alpha_i^\vee$ is the set of \textbf{positive coroots}.

To each root $\beta$ corresponds a unique coroot $\beta^\vee$: if $\beta=w(\alpha_i)$ then $\beta^\vee=w(\alpha_i^\vee)$. This map $\beta\mapsto\beta^\vee$ is well defined, bijective between $\Phi$ and $\Phi^\vee$ and sends positive roots to positive coroots. Note that $\langle \beta^\vee,\beta\rangle = 2$ for all $\beta \in \Phi$.

Moreover to each root $\beta$ one associates a reflection $s_\beta \in W^v$: if $\beta=w(\pm\alpha_i)$ then $s_\beta:= ws_iw^{-1}$\index{s@$s_\beta$}. It is well-defined, independently of  the choices of $w$ and $i$. Explicitly it is the map $x \mapsto x-\langle \beta^\vee , x \rangle \beta$. We have $s_\beta=s_{-\beta}$ and the map $\beta\mapsto s_\beta$ forms a bijection between  the set of positive roots and the set $\{ws_iw^{-1}\mid (w,i)\in W^v\times I\}$ of reflections of $W^v$.

Let $w\in W^v$. The \textbf{inversion set} of $w$ is the set \index{i@$\Inv(w)$} $$\Inv(w):=\{\alpha \in \Phi_+ \mid w(\alpha)\in \Phi_-\}=\{\alpha \in \Phi_+ \mid ws_\alpha <w\}.$$ An element of the vectorial Weyl group is completely determined by its inversion set.

\paragraph{Dominant coweights and Tits cone}

Elements of $Y$ are called coweights. We say that a coweight $\lambda\in Y$ is \textbf{dominant} if $\langle \lambda,\alpha_i\rangle\geq 0$ for all $i\in I$, or equivalently if $\langle \lambda,\beta\rangle \geq 0$ for all $\beta\in\Phi_+$. The set of dominant coweights is denoted as \index{y@$Y^{++}$, $Y^+$} $Y^{++}=\{\lambda \in Y\mid \langle \lambda,\alpha_i\rangle \geq 0\; \forall i\in I\}$. 

The orbit of $Y^{++}$ by $W^v$ is the \textbf{integral Tits cone}, $Y^+=\bigcup\limits_{w\in W^v} w.Y^{++}$; it is a sub-cone of $Y$ and it is a proper sub-cone if and only if $W^v$ is infinite, if and only if $A$ is not of finite type (see \cite[1.4.2]{kumar2002kac}).

The set $Y^{++}$ is a fundamental domain for the action of $W^v$ on $Y^+$, and for any coweight $\lambda \in Y^+$ (resp. $\lambda\in -Y^{+}$), we denote $\lambda^{++}$ to be the unique element of $Y^{++}$ (resp. $-Y^{++}$) in its $W^v$-orbit.

\paragraph{Parabolic subgroups and minimal coset representatives}
Let $\lambda\in Y^{++}$ be a dominant coweight. Then the fixator of $\lambda$ in $W^v$ is a standard parabolic subgroup, which we denote by \index{w@$W_\lambda$, $W_J$, $W^\lambda$} $W_\lambda$. Explicitly, if $J=\{i\in I \mid \langle \lambda,\alpha_i\rangle=0\}$, then $W_\lambda =W_J:=\langle r_j\mid j\in J\rangle$.

Any class $[w]\in W^v/W_\lambda$ admits a unique element of minimal length, and we denote by $W^\lambda$ the set of coset representatives of minimal length. Therefore $w\in W^\lambda \iff \ell(w)\leq \ell(\tilde w)\,\forall \tilde w \in wW_\lambda$.

Suppose now that $\lambda\in Y^+$ (not necessarily dominant), we denote by \index{v@$v^\lambda$}$v^\lambda$ the unique element of minimal length such that $\lambda=v^\lambda \lambda^{++}$, it is an element of $W^{\lambda^{++}}$. The element $v^\lambda$ is determined by: $\{\alpha\in \Phi_+ \mid \langle \lambda,\alpha\rangle <0\}=\Inv((v^\lambda)^{-1})$.

The coweight $\lambda$ also admits a parabolic fixator in $W^v$, which is $W_\lambda = v^\lambda W_{\lambda^{++}} (v^\lambda)^{-1}$.

A coweight $\lambda$ is \textbf{spherical} if its fixator $W_\lambda$ is finite, and it is \textbf{regular} if $W_\lambda$ is trivial. These notions only depend on the dominant part $\lambda^{++}$ of $\lambda$. We denote by \index{Y@$Y^+_{sph}$} $Y^+_{sph}$ the set of spherical coweights. Note that a coweight $\lambda$ is regular and dominant if and only if $\langle \lambda,\alpha_i\rangle > 0$ for all  $i\in I$.

For any $\lambda\in Y^+_{sph}$, we denote by $w_{\lambda,0}$\index{w@$w_{\lambda,0}$} the maximal element of the standard parabolic subgroup $W_{\lambda^{++}}$.

\paragraph{Affine Weyl semi-group}
Through the left action $W^v \curvearrowright Y$ we can form the semidirect product $Y\rtimes W^v$, for $(\lambda,w)\in Y\rtimes W^v$ we denote by $\qp^\lambda w$ the corresponding element of $Y\rtimes W^v$. If $W^v$ is finite, this semi-direct product is (a finite extension of) a Coxeter group, hence admits a Bruhat order, a Bruhat length and a set of simple reflections. However when $W^v$ is infinite, it does not have a Coxeter structure, and there is no known analog of the Bruhat order on it.

By definition, $Y^+\subset Y$ is stable by the action of $W^v$, therefore we can form \index{w@$W^+$}$W^+=Y^+ \rtimes W^v$ which is a sub-semi-group of $Y\rtimes W^v$. This semi-group is called \textbf{the affine Weyl semi-group}, it is sometimes referred as "the affine Weyl group" in the literature, even though it is not a group.  

Braverman, Kazhdan and Patnaik have defined a preorder on $W^+$ in \cite[B.2]{braverman2016iwahori} using an affine root system on which $Y\rtimes W^v$ acts. We now introduce this preorder (which is, in fact, an order), and in Section~\ref{subsection: affine bruhat order apartment} we give a geometric interpretation, already detailed in~\cite{philippe2023grading}.

We denote by \index{p@$\pr^{Y^+}$, $\pr^{Y^{++}}$} $\pr^{Y^+}$ the projection function $W^+\rightarrow Y^+$ and $\pr^{Y^{++}}$ the projection on the dominant coweight, so if $\bx=\qp^\lambda w \in W^+$, then $\pr^{Y^+}(\bx)=\lambda$ and $\pr^{Y^{++}}(\bx)=\lambda^{++}$. We say that an element $\bx\in W^+$ is \textbf{spherical} if its coweight $\pr^{Y^+}(\bx)$ is spherical, we then denote $W^+_{sph}=Y^+_{sph}\rtimes W^v$ the set of spherical elements in $W^+$.

\paragraph{Affine roots}
Let \index{p@$\Phi^a$}$\Phi^a=\Phi \times \mathbb Z$, it is \textbf{the set of affine roots} of $\mathcal D$. Let $(\beta,n)\in \Phi^a$, we say that $(\beta,n)$ is positive if $n>0$ or ($n=0$ and $\beta \in \Phi_+$); we write $\Phi^a_+$ for the set of positive affine roots. We also have $\Phi^a = \Phi^a_+ \sqcup -\Phi^a_+$. We can see $\underline\beta=(\beta,n)\in \Phi^a$ as the affine form on $Y$ defined by $\underline{\beta}(\lambda)=\langle \lambda,\beta\rangle +n$. Since $Y\rtimes W^v$ acts by affine transformations on $Y$, it also acts on its space of affine forms (by $(\underline{w}\cdot \varphi)(\lambda)=\varphi(\underline{w}^{-1}\lambda)$ for $\underline{w}\in Y\rtimes W^v$ and $\varphi\in \operatorname{Hom_{Aff}}(Y,\mathbb Z)$), and this action preserves $\Phi^a$. We therefore have a natural action of $Y\rtimes W^v$ on $\Phi^a$, explicitely given by:

\begin{equation}\label{eq : W^+ action}
    \qp^\lambda w.(\beta,n)=(w\beta,n+\langle \lambda,w\beta\rangle).
\end{equation}

For any $n \in \mathbb Z$, its sign is denoted $\sgn(n)\in \{-1,+1\}$, with the convention that $\sgn(0)=+1$.

For $n\in \Z$ and $\beta\in \Phi_+$, we set:\begin{align}\label{eq : affine_roots} \index{b@$\beta[n]$}
&\beta[n]=(\sgn(n) \beta,|n|\qp)\in \Phi_+^a \\
\index{s@$s_{\beta[n]}$}&s_{\beta[n]}=\qp^{n\beta^\vee} s_\beta \in  Y\rtimes W^v.\end{align}
If $n\neq0$ we also define $\beta[n]\in \Phi^a_+$ for $\beta\in\Phi_-$, by $\beta[n]=(-\beta)[-n]$.
Note that the element $s_{\beta[n]}$ does not belong to $W^+$ in general, but it is an element of order $2$ in $Y\rtimes W^v$ which fixes $\beta[n]^{-1}(\{0\})$, the \textbf{affine reflection associated to $\beta[n]$}.
\paragraph{Affine Bruhat order} The \textbf{affine Bruhat order} on $W^+$ is the transitive closure of the relation $<$ defined by, for $\bx\in W^+$ and $\beta[n]\in \Phi^a_+$: $s_{\beta[n]}\bx<\bx \iff \bx^{-1}(\beta[n]) \notin \Phi^a_+$. In some sense this means that $\beta[n]$ belongs to the inversion set of $\bx^{-1}$. Note in particular that if $\bx^{-1}(\beta[n]) \notin \Phi_+^a$ and $\bx$ lies in $W^+$ then $s_{\beta[n]}\bx\in W^+$. 

Explicitly, if $\bx=\qp^\lambda w \in W^+$, then $\bx^{-1}(\beta[n])\notin \Phi_+^a$ if and only if $|n|<\sgn(n)\langle \lambda,\beta\rangle$, or $|n|=\sgn(n)\langle \lambda,\beta\rangle$ and $\sgn(n)w^{-1}\beta \in\Phi_-$.

D. Muthiah (see \cite{muthiah2018iwahori}) has proved that the transitive closure of $<$ is anti-symmetric, hence it is an order. Moreover, together with D. Orr, they define in \cite{muthiah2019bruhat} a $\mathbb Z$-valued length function on $W^+$ strictly compatible with the affine Bruhat order.

\subsection{The affine Bruhat order in the standard apartment}\label{subsection: affine bruhat order apartment}

There is a strong geometric intuition behind root systems, vectorial Weyl groups and the vectorial Bruhat order, developed for instance in the context of buildings in \cite{ronan1989lectures}. There is also a geometrical interpretation of the affine Bruhat order on $W^+$ which we develop in this section, it takes place in the standard apartment $\A$,  which is the affine space underlying $Y\otimes \R$. 
More precisely, elements of $W^+$ can be identified with certain local germs in $\A$, called alcoves, and each positive affine root corresponds to a hyperplane $M_{\beta[n]}$ separating the set of alcoves in two. For $\bx\in W^+,\beta[n]\in \Phi_+^a$, we have that $\bx<s_{\beta[n]}\bx$ if and only if the alcove corresponding to $\bx$ is on the same side of the wall $M_{\beta[n]}$ as the fundamental alcove (the alcove corresponding to $1_{W^+}$). In the reductive setting, we recover the geometric interpretation of the affine Bruhat order in term of tiling of the standard apartment.
We start by recalling the geometrical interpretation of the standard Bruhat order of a (vectorial) Weyl group.
\vspace{0.5 cm}

\subsubsection{The vectorial apartment}

We fix a Kac-Moody root datum $\mathcal D=(A,X,Y,(\alpha_i^\vee)_{i\in I},(\alpha_i)_{i\in I})$.
Let $V=Y\otimes_\mathbb Z \mathbb R$, $X$ embeds in its dual $V^\vee$ and the vectorial Weyl group $W^v$ acts naturally on it.  For each root $\beta \in \Phi_+$ let $M_\beta=\{x \in V \mid \langle x,\beta \rangle = 0 \}$, it is a hyperplane of $V$ called \textbf{the vectorial wall associated to $\beta$}. We denote $\mathcal M_0=\{M_\beta\mid \beta\in \Phi_+\}$. The pair $(V,\mathcal M_0)$ is the \textbf{vectorial apartment} of $\mathcal D$.

Inside $V$ we have \textbf{the fundamental chamber} $C^v_f=\{v \in V \mid \langle v,\alpha_i\rangle>0, \forall i\in I\}$ and the \textbf{Tits cone} $\sT=W^v.\overline{C^v_f}$\index{t@$\sT$, $\mathring{\sT}$}, it is a convex cone. We denote its interior by $\mathring{\sT}$, it is the set of elements with finite stabilizer in $W^v$. Note that $Y^+=\sT\cap Y$, $Y^+_{sph}=\mathring{\sT}\cap Y$ and $Y^{++}=\overline{C^v_f}\cap Y$. We extend the definition of $^{++}$\index{z@$^{++}$} to $\sT\cup -\sT$: for $x\in \sT$ (resp. $x\in -\sT$), we denote by $x^{++}$ the unique element of $\overline{C^v_f}\cap W^v.x$ (resp. $-\overline{C^v_f}\cap W^v.x$).

A (positive) \textbf{vectorial chamber} is a set of the form $w.C^v_f$ for $w \in W^v$. Since $C^v_f$ has trivial stabiliser in $W^v$, the set of chambers is in natural bijection with $W^v$ by \index{c@$C^v_f$, $C^v_w$} $w\mapsto C^v_w:=wC^v_f$. Note that the vectorial chambers are the connected components of $\sT\setminus (\bigcup\limits_{\beta \in \Phi_+}M_\beta)$. Moreover, if $\beta=w(\alpha_i)$ with $\alpha_i$ a simple root, then $M_\beta \cap \sT$ contains $\overline{C^v_w}\cap \overline{C^v_{ws_i}}$, which we call the \textbf{panel} of type $i$ of the chamber $C_w$. We define negative chambers as the opposite of positive chambers. If $\Phi$ is infinite, then positive and negative chambers form two distinct $W^v$-orbit.
\vspace{0.5 cm}

We can put a structure of simplicial complex on $\sT$, for which the chambers are the cells of maximal rank and the panels are the cells of maximal rank within non-chambers. The walls split the Tits cone in two parts, and separate the set of vectorial chambers in two: say that $C^v_w$ is on the positive side of $M_\beta$ if $w^{-1}(\beta)>0$. In particular since $\beta$ is a positive root, the positive side is always the one which contains the fundamental chamber.
\begin{Remark}
    The condition is on $w^{-1}(\alpha)$ and not $w(\alpha)$ because $W^v$ acts on the left on $V$ and the contragredient action on its dual, which contain the roots, is more naturally a right action: $\langle w.y,\alpha\rangle=\langle y,w^{-1}(\alpha)\rangle$
\end{Remark}
Then the vectorial Bruhat order can be interpreted by:
    $s_\beta w>w$ if and only if, when we split $\sT$ along $M_\beta$ the chambers $C^v_w$ and $C^v_f$ are in the same connected component of $\sT$, that is to say $C^v_w$ is on the positive side of $M_\beta$. 
    
Moreover $\Inv(w^{-1})$ the inversion set of $w^{-1}$ can be interpreted as the set of walls separating the chamber $C^v_w=w.C^v_f$ from the fundamental chamber $C^v_f$.

\begin{figure}[h]

    \centering
    \includegraphics[width=\textwidth]{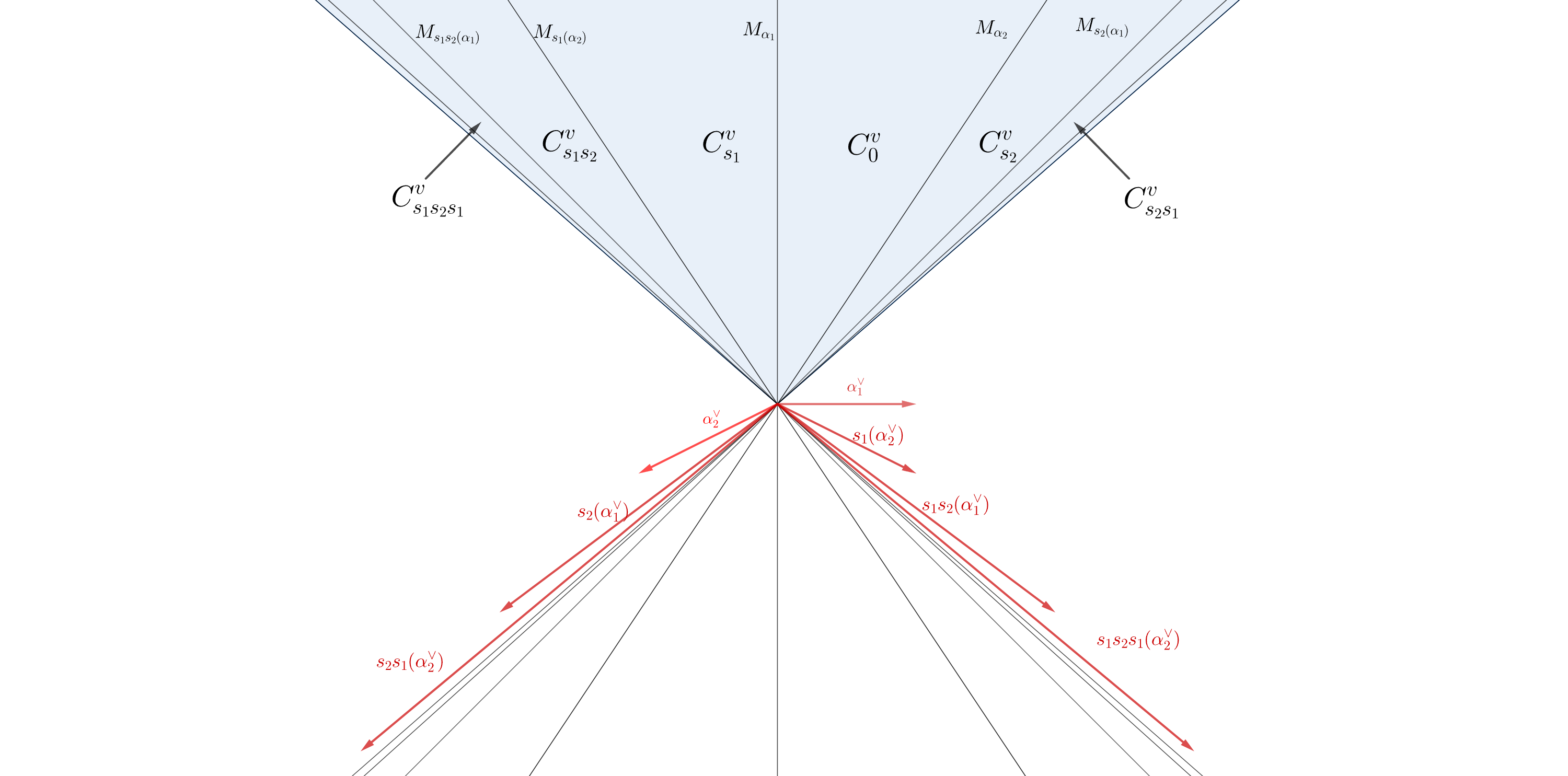}
    \caption{The vectorial apartment for a root datum of Kac-Moody Matrix $\begin{pmatrix}
        2 & -3 \\ -2 & 2
    \end{pmatrix}$.}
    \label{fig: figure2}
\end{figure}


\subsubsection{The affine apartment}\label{subsection: affine apartment}
We now define the affine apartment associated to a Kac-Moody root datum, as it is defined in \cite{gaussent2014spherical}. Let $\mathcal D=(A,X,Y,(\alpha_i^\vee)_{i\in I},(\alpha_i)_{i\in I})$ be a Kac-Moody root datum and let $\A$ be the real affine space associated to $V=Y\otimes\mathbb R$.

To any positive affine root $\beta[n] \in \Phi_+^a$ corresponds an affine hyperplane $M_{\beta[n]}$, the \textbf{affine wall} associated to the affine root $\beta[n]$. This hyperplane is equal to $(\beta[n])^{-1}(\{0\})$ if we consider $\beta[n]$ as an affine form on $\A$. Explicitly, we therefore have \index{m@$M_{\beta[n]}$}\begin{equation}
    M_{\beta[n]}=\{x \in \A \mid \langle x,\beta\rangle +n=0 \}.\end{equation} 
    Let $\mathcal M^a=\{M_{\beta[n]}\mid \beta[n]\in\Phi_+^a\}$\index{m@$\cM^a$} be the set of affine walls. Then the \textbf{affine apartment} is the pair $(\A,\mathcal M^a)$. If $n\in \mathbb R\setminus\mathbb Z$ and $\beta\in \Phi$, the affine hyperplane $\{x\in \A \mid \langle x,\beta\rangle+n=0\}$ is referred as a \textbf{ghost wall}, and is still denoted $M_{\beta[n]}$. A \textbf{half-apartment} is any region of $\A$ delimited by an affine wall, hence any region of the form $\{x \in \A \mid \langle x,\beta\rangle +n \mathcal R 0\}$, for $\mathcal R \in \{<,>,\leq,\geq\}$. 

The group $W^v\ltimes Y$ acts by affine transformations on $\A$, by the formula $\qp^\lambda w(x)=-\lambda+w(x)$. This action preserves $\mathcal M^a$.

To give an interpretation of the affine Bruhat order, we need to define the affine analog of vectorial chambers, which are called alcoves. However, when the root system is infinite, the set of affine walls is not locally discrete, and therefore the connected components of $\A \setminus \cup_{M\in \mathcal M^a} M$ have empty interior and are not adapted. There are two ways of dealing with this issue, either by working in the tangent space of $\A$, as it is done in \cite{philippe2023grading}, or by defining affine chambers no longer as sets but as filters of sets, which is better suited for masure theoretic purposes and therefore is more often done in the literature (see \cite{gaussent2014spherical}).
\paragraph{Filters}A \textbf{filter} on a set $E$ is a nonempty set $\VC$ of nonempty subsets of $E$ such that, for all subsets $S$, $S'$ of $E$,  if $S$, $S'\in \VC$ then $S\cap S'\in \VC$ and, if $S'\subset S$, with $S'\in \VC$ then $S\in \VC$.

Let $E,E'$ be  sets, $E'\subset E$ and $\VC$ be a filter on $E'$. One says that a set $\Omega\subset E$ contains $\VC$ if there exists $\Omega'\in \VC$ such that $\Omega'\subset \Omega$ (or equivalently if $\Omega\in \VC$ if $E=E'$). Let $f:E\rightarrow E$. One says that $f$ fixes $\VC$ if there exists $\Omega'\in \VC$ such that $f$ fixes $\Omega'$. 

A \textbf{basis} for a filter $\VC$ is a subset $\BC\subset \VC$ such that $\VC=\{S\subset E \mid \exists B\in \BC,\; B\subset S\}$. If $\BC$ is a set of sets closed by finite intersection and not containing the empty set, then $\{S\subset E \mid \exists B\in \BC,\; B\subset S\}$ is a filter with basis $\BC$.  

\paragraph{Alcoves}
Let $\|.\|$ be an arbitrary norm on $V$, and for any $x \in \A$, $r\in \mathbb R_+$ let $B(x,r)$ denote the open ball of radius $r$ centered at $x$ for $\|.\|$.
Let $x\in \A$, $w\in W^v$ and $\varepsilon\in \{+,-\}$. The (closed) \textbf{alcove based at $x$ and of vectorial direction $\varepsilon C^v_w$} is the filter with basis $\{(x+\varepsilon \overline{C^v_w}) \cap B(x,r)\mid r>0\}$, it is denoted \index{a@$a(x,\varepsilon w)$} $a(x,\varepsilon w)$, $\varepsilon$ is its \textbf{sign}. 
The \textbf{sector} generated by $a(x,\varepsilon w)$ is the set $x+\varepsilon \overline{C^v_w} \subset \A$. When $(\A,\mathcal M^a)$ is discrete (when the Kac-Moody root system is of finite type), we recover the definition of pointed alcoves: alcoves with a choice of base-point.

An alcove is well-defined by its base-point and its vectorial direction: if $a$ is an alcove, then its base-point $x$ is the unique point of $\A$ such that $\forall U \in a,\; x\in U$, and its vectorial direction is the unique vectorial cone of the form $\varepsilon C^v_w$ such that $(x+\varepsilon C^v_w) \cap U$ has non-empty interior for all $U\in a$. We denote by $a^v$ the direction of an alcove $a$.

The semi-direct product $W^v\ltimes Y$ acts freely on the set of alcoves, through the following action (induced by its action on $\A$):
$$\qp^\lambda w.a(x,\varepsilon v)=a(w.x-\lambda,\varepsilon wv).$$ 
We can therefore identify $W^+=Y^+\rtimes W^v$, with a subset of alcoves: to any $\qp^\lambda w \in W^+$, the associated alcove is $a(-\lambda,+w)=\qp^\lambda w.a(0_\A,1_{W^v})$, which we denote \index{c@$C_{\qp^\lambda w}$} $C_{\qp^\lambda w}$.

The map $(x,\varepsilon,w)\mapsto a(x,\varepsilon w)$ induces a bijection between the set of alcoves and $\A\times \{-,+\}\times W^v$. Using this bijection, we equip the set of alcoves with the product topology, where $\A$ is equipped with the usual topology of finite dimensional affine space and $\{-,+\}\times W^v$ is equipped with the discrete topology. Given any point $x\in \A$, the set of alcoves based at $x$ forms the tangent fundamental apartment at $x$, which has a particular combinatorial structure described below.

\paragraph{Tangent apartment at a point} Let $x\in \A$ and $\varepsilon\in \{+,-\}$. Denote by $\cT^\varepsilon_x(\A)$  the set of alcoves of sign $\varepsilon$ based at $x$. We define a $W^v$-valued distance function $d^\varepsilon: \cT^\varepsilon_x(\A_\oplus)\rightarrow W^v$\index{d@$d^\varepsilon,d^+,d^-$}: if $a_1,a_2\in \cT^\varepsilon_x(\A)$, there exist unique elements $w_1,w_2\in W^v$ such that $a_1=a(x,\varepsilon w_1)$ and $a_2=a(x,\varepsilon w_2)$. Then we set: \begin{equation}\label{e_def_d_epsilon}
    d^\varepsilon(a_1,a_2)=w_1^{-1} w_2.
\end{equation}

We also define a codistance function: $d^\ast: \cT^+_x(\A_\oplus)\times \cT^-_x(\A_\oplus) \cup \cT^-_x(\A_\oplus)\times\cT^+_x(\A_\oplus) \rightarrow W^v$\index{d@$d^\ast$}:

If $(a_+,a_-)\in \cT^+_x(\A_\oplus)\times \cT^-_x(\A_\oplus)$, similarly there exist  $w_+,w_-\in W^v$ such that $a_+=a(x,+ w_+)$ and  $a_-=a(x,- w_-)$. Then set $d^\ast(a_+,a_-)=w_+^{-1}w_-$ and \begin{equation}\label{e_def_d_ast}
    d^\ast(a_-,a_+)=w_-^{-1}w_+=d^\ast(a_+,a_-)^{-1}.
\end{equation} We say that $a_-,a_+$ are \textbf{opposite} if $d^\ast(a_-,a_+)=1_{W^v}$.
We call $(\cT^\pm_x(\A),d^\pm,d^\ast)$ the \textbf{standard tangent apartment} at $x$. In building theoretic terms, it is a twin apartment (or twin thin building) of type $(W^v,S)$.

\paragraph{Geometric interpretation of the affine Bruhat order}

Mirroring the classical situation, \index{c@$C_0$} $C_0=a(0_\A,+1_{W^v})$ is called the \textbf{fundamental alcove} and $W^+$ acts on $\{C_{\bx}\mid \bx \in W^+\}$ simply transitively. Affine walls (and ghost walls) separate naturally the set of alcoves in two and we call the side containing $C_0$ the positive side. Note that an affine root $(\beta,n)$ is positive if and only if $\beta(C_0)+n\subset \mathbb R_+$, therefore the positive side of an affine wall corresponds to the associated positive root. Let $\beta\in \Phi$, $n\in \Z$ and $\qp^\lambda w\in W^+$.
\begin{itemize}
\item If $n\neq 0$ and $-\lambda \notin M_{\beta[n]}$, then $C_{\qp^\lambda w}$ is on the positive side of $M_{\beta[n]}$ if and only if $-\lambda$ is in the same connected component of $\A \setminus M_{\beta[n]}$ than $0_\A$, that is to say $|n|-\sgn(n)\langle \lambda,\beta\rangle>0$.
\item If $n>0$ (resp. $n<0$) and $-\lambda \in M_{\beta[n]}$, it is on the positive side if and only if $C^v_w$ is on the positive side (resp. negative side) of $M_\beta$, that is to say $\sgn(n)w^{-1}(\beta)>0$.
\item If $n=0$ and $-\lambda\notin M_{\beta[0]}$ then $C_{\qp^\lambda w}$ is on the positive side of the wall if and only if $\langle \lambda,\beta\rangle <0$ hence on the side of $C^v_f$.
\item If $n=0$ and $-\lambda \in M_{\beta[0]}$, then $C_{\qp^\lambda w}$ is on the positive side of the wall if and only if $C^v_w$ is on the positive side of $M_\beta$, that is to say $w^{-1}(\beta)>0$, hence on the side of $C_0$. 
\end{itemize}

Then the $W^+$-Bruhat order can be interpreted by: $s_{\beta[n]}\qp^\lambda w >\qp^\lambda w$ if and only if $C_{\qp^{\lambda}w}$ is on the positive side of $M_{\beta[n]}$.
\begin{Remark}
    Another possible choice of embedding would be to identify $\qp^\lambda w$ to the alcove $a(\lambda,-w)$ and define $M_{\beta[n]}$ as $\{x\mid \langle x,\beta\rangle =n\}$ but either way, the vectorial direction of the alcove needs to be of opposite sign of its base-point. Our choice is motivated by the choices made in \cite{bardy2025twin}.
\end{Remark}
\begin{figure}[h]

    \centering
    \includegraphics[width=\textwidth]{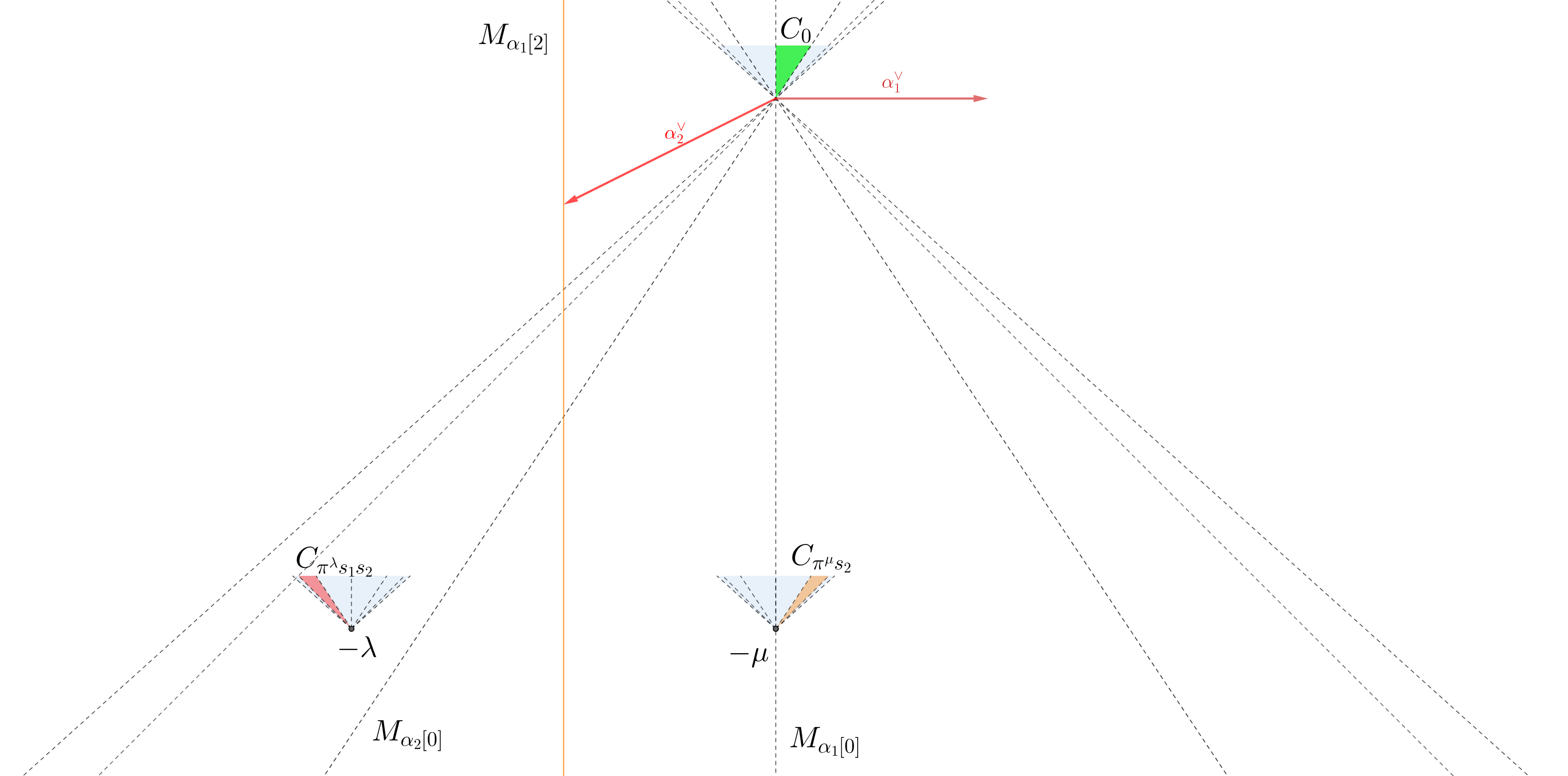}
    \caption{The affine apartment for a root datum with Kac-Moody Matrix $\begin{pmatrix}
        2 & -3 \\ -2 & 2
    \end{pmatrix}$.}
    \label{fig: figure3}
\end{figure}

\subsubsection{Segments, paths, segment germs}

We end this preliminary section by formally introducing segments, paths and segment germs in the standard apartment. We also define projections of alcoves on points and on segment germs, the projections will play an important computational role.

By \textbf{segment} in $\A$, we mean an affine map $\fs: [t_{in},t_f]\mapsto \A$ where $t_{in}<t_f$ are real numbers. 

A \textbf{path} in $\A$ is a piece-wise affine map $[t_{in},t_f]\mapsto \A$, with $t_{in}<t_f\in \R$ . That is to say, it is a continuous map $\fc:[t_{in},t_f]\mapsto\A$ such that there exists a finite number of times $t_0=t_{in}<t_1<\dots<t_n=t_f$, and segments $(\fs_k)_{k\in \llbracket1,n\rrbracket}$ such that, for each $k\in \llbracket1,n\rrbracket$, $\dom(\fs_k)=[t_{k-1},t_k]$ (where $\dom$ denotes the domain of a map) and $\fc(t)=\fs_k(t)$ for all $t\in [t_{k-1},t_k]$. We identify two paths $\fc$, $\fc'$ if $\im(\fc)=\im(\fc')$ and $\fc^{-1}\circ \fc'$ has positive derivative. Therefore we may always assume that the domain of a path (and in particular of a segment) is $[0,1]$.

\begin{Definition}[Segment germs]\label{Definition_segment_germ}
    For $t\in[0,1)$, (resp. $t\in[0,1)$, resp.  $t\in(0,1)$) let $[t,t^+)$ (resp. $(t^-,t]$, resp. $(t^-,t^+)$) be the filter of $[0,1]$ with basis $\{[t,t+\varepsilon)\mid \varepsilon>0\}$ (resp. $\{(t-\varepsilon,t]\mid \varepsilon>0\}$, resp.  $\{(t-\varepsilon,t+\varepsilon)\mid \varepsilon>0\}$). Filters of this form and their images in $\A$ by affine maps from $\R$ to $\A$ are called \textbf{segment germs}.
    
    In particular, for any path $\fc: [0,1]\mapsto \A$, we denote by \index{c@$\fc_+(t)$, $\fc_-(t)$} $\fc_+(t)$ the segment germ $\fc_+(t):=\fc([t,t^+))$ and by $\fc_-(t)$ the segment germ $\fc_-(t):=\fc((t^-,t])$.  So $\fc_+(t)$ is the filter of $\A$ with basis $\{\fc([t,t+\varepsilon))\mid \varepsilon>0\}$. Note that, by definition of the derivative and since $\fc$ is piece-wise affine, $\fc_+(t)=\fc(t)+[0,0^+)\fc'_+(t)$  and  $\fc_-(t)=\fc(t)+(0^-,0]\fc'_-(t)$ where $\fc'_+(t)$ (resp. $\fc_-'(t)$) denotes the right (resp. left) derivative of $\fc$ at $t$. The derivatives $\fc'_+(t)$ and $\fc_-'(t)$ depend on the choice of parametrization, but the filters $\fc_+(t)$, $\fc_-(t)$ do not.
\end{Definition}

\begin{Definition}[$\lambda$-paths]

Let $\lambda\in \A$. A $\lambda$-path is a piecewise affine map $\fc:[0,1]\rightarrow \A$ such that for all $t\in (0,1)$, we have $\fc'_{\pm}(t)\in W^v.\lambda$. In particular, the path $\fs_\lambda:t\mapsto -t\lambda$\index{s@$\fs_\lambda$} is a $-\lambda$-path which we call the \textbf{standard segment of type $\lambda$}. 
    
\end{Definition}

\paragraph{Projection on points and on segment germs}
 For $x,y\in \A$, we write $x\leq y$\index{z@$\leq$} if $y-x\in \sT$. Since $\sT$ is a convex cone, it is  a preorder on $\A$, called the \textbf{Tits preorder}.
 
Let $x,y\in \A$ be such that  $x\leq y$ (resp. $y\leq x$) and $a\in \cT^\pm_y(\A)$ be an alcove based at $y$. Then the \textbf{projection of $a$ at $x$}, denoted $\pr_x(a)\in \cT^+_x(\A)$ (resp. $\pr_x(a)\in \cT^-_x(\A)$), \index{p@$\pr_x(a)$} is the unique alcove based at $x$ which contains $a$ in its sector, it should be thought of as "the closest alcove based at $x$ from $a$" (see \cite[\S 2.1]{bardy2021structure} for a proof of uniqueness). 

Suppose that $\fs$ is a path in $\A$ and $t\in [0,1]$ is such that $\fs(t)=:x$ and $\fs'_-(t)\in \mp\mathring{\sT}$. Then the segment germ $\fs_-(t)=x+(0^-,0]\fs'_-(t)$ is contained in finitely many alcoves of $\cT^\pm_x(\A)$, and the projection of $a$ on $\fs_-(t)$ is the unique alcove $\pr_{\fs_-(t)}(a)$ containing $\fs_-(t)$ which minimizes the distance to the point projection $\pr_x(a)$. Similarly, if $\fs'_+(t)\in \mp\mathring{\sT}$ then $\fs_+(t)$ is contained in a finite number of alcoves of $\cT^\mp_x(\I_\oplus)$, and we define $\pr_{\fs_+(t)}(a)$ as the unique alcove containing $\fs_+(t)$ which maximizes the codistance to $\pr_x(a)$. Again $\pr_{\fs_\varepsilon(t)}(a)$ should be thought of as "the closest alcove from $a$ containing $\fs_\varepsilon(t)$".

\begin{Lemma}\label{Lemma: projection walls}
    Let $\fs:[t_0,t_1]\rightarrow  \A$ be a segment, where $t_0,t_1\in \R$. Let $C$ be an alcove based at $\fs(t_1)$. Let $D$ be a half-apartment containing $\fs(t_0)$ and $C$. Then $D$ contains $\pr_{\fs_{+}(t_0)}(C)$.
\end{Lemma}

\begin{proof}
Let $\underline{\beta}=(\beta,n)\in \Phi^a$ be such that $D=\{x\in \A\mid \underline{\beta}(x)\geq 0\}$. Set $M=\underline{\beta}^{-1}(\{0\})$. If $M\supset [\fs(t_0),\fs(t_1)]$, then $D\supset \mathrm{prism}_{\fs_+(t_0)}(C)$, where $\mathrm{prism}$ is defined in \cite[2.1.2]{bardy2021structure}. Then by \cite[Lemma 2.1]{bardy2021structure}, $D$ contains $\pr_{\fs_+(t_0)}(C)$.
    
    Assume now that $M\not\supset [\fs(t_0),\fs(t_1)]$. If$\underline{\beta}(\fs(t_0))>0$, then $\underline{\beta}(C')>0$ for any alcove $C'$ based at $\fs(t_0)$ and in particular $\pr_{\fs_+(t_0)}(C)\subset D$. Assume $\underline{\beta}(\fs(t_0))=0$. Then $\beta(\fs(t_1)-\fs(t_0))>0$.  Write $\pr_{\fs_+(t_0)}(C)=a(\fs(t_0),\varepsilon v)$, with $\varepsilon\in \{-,+\}$ and $v\in W^v$. We have $\beta(\varepsilon v.C^v_f)\in \{\R_{>0},\R_{<0}\}$ and as $\overline{\varepsilon v.C^v_f}\supset \fs(t_1)-\fs(t_0)$, we have $\beta(\varepsilon v.C^v_f)=\R_{>0}$. Therefore $\underline{\beta}(\fs(t_0)+\varepsilon v.C^v_f)=\R_{>0}$, hence $\underline{\beta}(\pr_{\fs_+(t_0)}(C))\subset \underline{\beta}(\fs(t_0)+\overline{\varepsilon v.C^v_f})=\R_{\geq 0}$ and the lemma follows.
\end{proof}

\begin{Lemma}\label{Lemma: projection codistance}
Let $\lambda \in \mathring{\sT}$, let $t_1<t_2\in \R$ and let $\fs: [t_1,t_2]\mapsto \A$ be a segment such that $\fs'_+(t_1)=\pm \lambda$. Let $u\in W^v$ and suppose that $a=a(\fs(t_1),\pm u)$ dominates $\fs_+(t_1)$. Then $\pr_{\fs_-(t_2)}(a)=a(\fs(t_2),\mp u w_{\lambda,0})$, where $w_{\lambda,0}$ is the maximal element of the finite standard parabolic subgroup $W_{\lambda^{++}}$.
\end{Lemma}

\begin{proof}
     By definition of the projection, $\pr_{\fs_-(t_2)}(a)$ is the unique alcove $\tilde a$ dominating $\fs_-(t_2)$ and such that there is no wall containing $\fs_-(t_2)$ and separating $\tilde a$ from $a$.
     
     With the notation of the lemma, let $\varepsilon \in \{+,-\}$ be such that $\fs'_\eta(t)=\varepsilon \lambda$ for all $(t,\eta)\in[t_1,t_2]^\pm$. Since $a=a(\fs(t_1),\varepsilon u)$ dominates $\fs_+(t)$, $\lambda=u\lambda^{++}$ and $W_\lambda = uW_{\lambda^{++}} u^{-1}$. Therefore alcoves dominating $\fs_-(t_2)$ are the alcoves in the set $$\{a(\fs(t_2),-\varepsilon u w) \mid w\in W_{\lambda^{++}}\}=\{a(\fs(t_2),-\varepsilon  \tilde w u \mid \tilde w \in W_{\lambda}\}.$$ Any wall $M$ separating two such alcoves needs to contain the segment germ $\fs_-(t_2)$ and since $\s$ is affine on $[t_1,t_2]$, it also contains $\fs_+(t_1)$. Therefore the half-apartment delimited by $M$ and containing $a=a(\fs(t_1),\varepsilon u)$ also contains $a(\fs(t_2),\varepsilon u)$. Suppose that $w\in W_{\lambda^{++}}$ is such that there exists a root $\beta\in \Phi_+$ with $\langle\lambda^{++},\beta\rangle =0$ and $s_\beta w>w$, and let $M$ denote the wall of direction $u(\beta)$ going through $\fs(t_2)$. Since $\langle \lambda,u(\beta)\rangle = \langle \lambda^{++},\beta\rangle=0$, $M$ contains $\fs_-(t_2)$. Let $x$ denote any point of the vectorial chamber $C^v_w$. Then $s_\beta w> w \iff \langle x,\beta\rangle >0\iff \langle u(x),u(\beta)\rangle>0$. Moreover, by definition $-C^v_u=-u.C^v_f=\{y\in V \mid \langle y,u(\alpha)\rangle <0 \forall \alpha \in \Phi_+\}$. Therefore the vectorial wall of direction $u(\beta)$ separates $-C^v_u$ from $uC^v_w=C^v_{uw}$ if and only if $s_\beta w > w$. By the definition of the projection, we deduce that, if we write $\pr_{\fs_-(t_2)}(a)=a(\fs(t_2),-\varepsilon u w)$ then $s_\beta w < w$ for any root $\beta\in \Phi_+$ such that $\langle \lambda^{++},\beta\rangle =0$. Hence $w$ needs to be maximal in $W_{\lambda^{++}}$, that is to say $w=w_{\lambda,0}$.
\end{proof}

\begin{Lemma}\label{Lemma: projection segments}
   Let $\lambda \in \mathring{\sT}$, let $t_{in},t_f\in \R$ and let $\fs:[t_{in},t_f]\mapsto \A$ be a segment with direction $\fs'_+(t_{in})=\pm \lambda$. Then for any $t_{in}\leq t_1<t_2<t_3\leq t_f$, and any two alcoves $a_1, a_2$ respectively dominating $\fs_+(t_1)$ and $\fs_+(t_2)$ we have $\pr_{\fs_-(t_3)}(a_1)=\pr_{\fs_-(t_2)}(a_2)$ if and only if $a_1$ and $a_2$ have the same direction.
\end{Lemma}
\begin{proof}
  Let $\varepsilon\in \{+,-\}$ be such that $\fs'_+(t_{in})=\varepsilon\lambda$. For $i\in \{1,2\}$, let $u_i\in W^v$ be such that $a_i=a(\fs(t_i),\varepsilon u_i)$. The alcoves $a_1$ and $a_2$ have same direction if and only if $u_1=u_2$. Moreover, by Lemma~\ref{Lemma: projection codistance}, $\pr_{\fs_-(t_3)}(a_1)=\pr_{\fs_-(t_3)}(a_2)\iff u_1 w_{\lambda,0}=u_2w_{\lambda,0}\iff u_1=u_2$ which proves the result.
\end{proof}

\paragraph{Local fundamental chamber}\label{subsubsection: Local fundamental chamber}For $x\in -\sT$, let \index{c@$C^{++}_x$, $C^{\infty}_x$}$C^{++}_x:=\pr_x(C_0)$ be the projection of the fundamental alcove $C_0$ at $x$, the unique alcove which contains $C_0$ in its sector. By \cite[\S 5.1]{bardy2025twin}, it is also the unique alcove which is on the positive side of every wall (including ghost walls) which go through $x$. We also define $C^\infty_{x}$ as the opposite alcove of $C^{++}_{x}$ based at $x$.

\begin{Lemma}\label{l_explicit_C++}
Let $x\in -\sT$, write $x=-v^x x^{++}$ with $v^x\in W^v$ of minimal length and $x^{++}\in \overline{C^v_f}$. Then $C^{++}_x$ is the alcove $a(x,+v^x)$, and $C^\infty_x=a(x,-v^x)$. 
    
\end{Lemma}

\begin{proof}
Note that $0_\A = x+v^xx^{++}\in x+\overline{C^v_{v^x}}$, so it lies in the sector of $a(x,+v^x)$. Assume by contradiction that there is an affine or ghost wall $M_{\beta[n]}$ containing $x$ and such that $a(x,+v^x)$ lies on the negative side of $M_{\beta[n]}$. Since $0_\A \in C_0\cap (x+\overline{C^v_{v^x}})$, the wall $M_{\beta[n]}$ necessarily goes through $0_\A$, so $n=0$ and $s_\beta x=x$. Then $s_{\beta[0]}.a(x,+v^x)=a(x,+s_\beta v^x)$ lies on the positive side of the wall $M_{\beta[0]}$. By definition of the affine Bruhat order, this implies that $s_\beta.v^x<v^x$, which, since $s_\beta v^x x^{++}=-x$, contradicts minimality of $v^x$. By uniqueness of the projection we conclude that $C^{++}_x=a(x,+v^x)$, and since $C^\infty_x$ is its opposite, $C^\infty_x=a(x,-v^x)$.
\end{proof}

The alcove $C^{++}_x$ can be considered as the local fundamental chamber at $x$: by definition an alcove based at $x$ and of direction $C^v$, is on the positive side of $M_{\beta[n]}$ if and only if $C^v$ is on the same side of the (affine or ghost) wall $M_\beta$ as $C^v_{v^x}$ (which is the direction of $C^{++}_x$).

\section{$C_\infty$-Hecke open paths}\label{s_Cinfty_open_paths}
The goal of this section is to introduce $C_\infty$-Hecke open paths, which are generalized piece-wise affine segments in an affine space satisfying combinatorial properties related to the affine Bruhat order; they are slight variants of the $C_\infty$-Hecke paths introduced by Muthiah (\cite{muthiah2019double}) and further studied by Bardy-Panse, the first-named author and Rousseau   (\cite{bardy2025twin}).

\subsection{Definition and relation with the affine Bruhat order}\label{ss_Cinfty_hecke_paths}

For any $t,t'\in\R$ such that $t<t'$, we denote by $[t,t']^\pm$ the set $[t,t']^\pm:=\big([t,t')\times\{+\}\big)\sqcup \big((t,t']\times\{-\}\big)$. \index{t@$[t,t']^\pm$}

\begin{Definition}[Open paths]\label{Def: open paths}
    An \textbf{open path} is a tuple $\underline{\fc}:=(\fc,D,\underline{\fc}(1))$ where:
    \begin{enumerate}
        \item $\fc:[0,1]\rightarrow \A$ is a $-\lambda$-path, for some spherical $\lambda\in \mathring{\sT}$, 
        \item $D$ is the data, for each $(t,\varepsilon)\in [0,1]^\pm$, of an alcove $D^\varepsilon_t$ based at $\fc(t)$ and dominating $\fc_\varepsilon(t)$,
        \item $\underline{\fc}(1)$ is an alcove based at the endpoint $\fc(1)$ of $\fc$.
    \end{enumerate}
    We respectively call $\fc$, $D$ and $\underline{\fc}(1)$ the \textbf{underlying path}, the \textbf{decoration} and the \textbf{end alcove} of $\underline{\fc}$.

    Moreover, we say that it is \textbf{admissible} if, for any $t_0<t_1\in [0,1]$ such that $\fc$ is affine on $[t_0,t_1]$, we have that, for all $t\in(t_0,t_1)$, 
    \begin{equation}\label{e_admissibility}
   D^-_t=\pr_{\fc_-(t)}(D^+_{t_0}) \, \text{and} \, D^+_t=\pr_{\fc_+(t)}(D^-_{t_1}).     
    \end{equation} 
    This is equivalent to requiring that for all $(t,\varepsilon),(t',\varepsilon')\in [t_0,t_1]^{\pm}$, we have \begin{equation}\label{e_decoration}
        D_{t}^\varepsilon=\pr_{\fc_{\varepsilon}(t)}(D_{t'}^{\varepsilon'}).
    \end{equation} 

\end{Definition}

\begin{Remark} Suppose that $\underline{\fc}$ is an   open path whose non-open component is a $-\lambda$-path for $\lambda$ regular, that is to say $W_\lambda=1$. Then there is a unique alcove dominating $\fc_\varepsilon(t)$ for $(t,\varepsilon)\in[0,1]^\pm$, so the decoration is uniquely determined from the data of the non-open component $\fc$.

\end{Remark}

\begin{Lemma}\label{Lemma: admissibility condition}Let $\lambda\in \mathring{\sT}$, let $w_{\lambda,0}$ denote the maximal element of the finite standard parabolic subgroup $W_{\lambda^{++}}$. Let $\underline{\fc}=(\fc,D,\underline{\fc}(1))$ be an open path such that $\fc$ is a $-\lambda$-path. Let $0\leq t_1<t_2<t_3\leq1$ and suppose that the admissibility condition is satisfied on $[t_1,t_2]$ and $[t_2,t_3]$. Then the admissibility condition is satisfied on $[t_1,t_3]$ if and only if:
   \begin{equation}\label{eq: admissibility condition}  \fc'_+(t_2)\neq \fc'_-(t_2)
    \text{ or } d^\ast(D^-_{t_2},D^+_{t_2})=w_{\lambda,0}.
\end{equation}
\end{Lemma}

\begin{proof}Let $\underline\fc$ be an open path as in the lemma.

    If $\fc'_+(t_2)\neq \fc'_-(t_2)$ then, for any $\tilde t_1<\tilde t_2\in [t_1,t_3]$, if $\fc$ is affine on $[\tilde t_1,\tilde t_2]$ then either $\tilde t_1 \geq t_2$, either $\tilde t_2 \leq t_2$, so it is clear that the admissibility condition is satisfied on $[t_1,t_3]$. Suppose that $\fc'_+(t_2)= \fc'_-(t_2)$, then let $[\tilde t_1,\tilde t_2]\subset [t_1,t_3]$ be the maximal interval containing $t_2$ on which $\fc$ is affine. Since the admissibility condition is satisfied on $[t_1,t_2]$ and $[t_2,t_3]$ it is also satisfied on $[t_1,\tilde t_1]$ and $[\tilde t_2,t_3]$. Therefore as $\fc'$ varies on $\tilde t_1$ and on $\tilde t_2$ it is satisfied on $[t_1,t_3]$ if and only if it is satisfied on $[\tilde t_1,\tilde t_2]$. By Lemma~\ref{Lemma: projection segments}, the admissibility condition for $[\tilde t_1,t_2]$ and $[t_2,\tilde t_2]$ implies that $D^+$ (resp. $D^-$) has constant direction on $[\tilde t_1,t_2)$ and on $[t_2,\tilde t_2)$ (resp. on $(\tilde t_1,t_2]$ and on $(t_2,\tilde t_2]$), therefore the decoration is admissible on $[\tilde t_1,\tilde t_2]$ if and only if $D^+_{t_2}$ and $D^+_{\tilde t_1}$ have same direction (if and only if $D^-_{t_2}=\pr_{\fc_-(t_2)}(D^+_{\tilde t_1})$ and $D^-_{\tilde t_2}=\pr_{\fc_-(\tilde t_2)}(D^+_{t_2})$ have same direction). Let $u\in W^v$ be such that $D^+_{\tilde t_1}=a(\fc(\tilde t_1),- u)$ then $D^+_{t_2}$ and $D^+_{\tilde t_1}$ have same direction if and only if $D^+_{t_2}=a(\fc(t_2),- u)$. Moreover by Lemma~\ref{Lemma: projection codistance} $D^-_{t_2}=a(\fc(t_2), u w_{\lambda,0})$, we deduce that $D^+_{t_2}$ and $D^+_{\tilde t_1}$ have same direction if and only if $d^\ast(D^-_{t_2},D^+_{t_2})=w_{\lambda,0}$, which concludes the proof.

\end{proof}

\begin{Remark}\label{r_characterization_decorations}
Let $\underline{\fc}=(\fc,D,\underline{\fc}(1))$ be an admissible open path, let $\lambda$ be such that $\fc$ is a $-\lambda$-path and recall that $w_{\lambda,0}$ denotes the maximal element of $W_{\lambda^{++}}$. Set $E=\{t\in (0,1)\mid \fc'_+(t)\neq \fc'_-(t)\}$. Write $E=\{t_1,\ldots,t_n\}$, with $0<t_1<\ldots <t_n<t_{n+1}=1$. From Lemma~\ref{Lemma: admissibility condition}, we deduce that the decoration $D$ is completely determined by the choice, for each $i\in \llbracket1,n+1\rrbracket$ of a vectorial chamber $C^v_i$ dominating $\fc'_-(t_i)$, as follows:

For each $i\in \llbracket 1,n+1\rrbracket$ let $u_i\in W^v$ be the element such that $C^v_i=u_i.C^v_f$. The decoration $D$ associated to $(C_1^v,\ldots,C_{n+1}^v)$ is then given by $D_t^-=	a(\fc(t),u_i)=germ_{\fc(t)}(\fc(t)+C^v_i)$ (resp. $D_t^+=a(\fc(t),-u_iw_{\lambda,0})$) for $t\in (t_{i-1},t_i]$ (resp. $t\in [t_{i-1},t_i)$)  and $i\in \llbracket 1,n+1\rrbracket$.

 In particular, if the non-open component $\fc$ is a segment then it suffices to choose $D^-_1$, and there is a canonical choice given by the end alcove $\underline{\fc}(1)$, which motivates the following definition.
\end{Remark}

\begin{Definition}[Open segments]\label{Def: open segments}
    Let $\bx=\qp^\lambda w \in W^+_{sph}$ be any spherical element. The \textbf{open segment} $\underline\fs_\bx=(\fs_\lambda,D_\bx,C_\bx)$ is the unique admissible open path with underlying path $\fs_\lambda: t\mapsto -t\lambda$, with end alcove $\underline\fs_\bx(1)=C_\bx$ and such that $D^-_{\bx,1}=\pr_{\fs_{\lambda,-}(1)}(C_\bx)$.  

Explicitly, the decoration $D_\bx$ is given by $D^\varepsilon_{\bx,t}=\pr_{\fs_{\lambda,\varepsilon}(t)}(C_\bx)$ for any $(t,\varepsilon)\in[0,1]^\pm$.

An open segment is an element of the form $\underline{\fs}_\bx$, for some $\bx\in W^+_{sph}$.
\end{Definition}
Open paths can be folded along affine walls, through the following process. $C_\infty$-Hecke open paths are then obtained from open segments through successive foldings satisfying a positivity condition (see \ref{Definiton: Hecke open path}).
\begin{Definition}[Folding of an open path]
Let $\underline{\fc}=(\fc,D,\underline{\fc}(1))$ be an open path, $t_0\in [0,1]$ be a time and $\beta[n]\in \Phi_+^a$ be such that $\fc(t_0)\in M_{\beta[n]}$. The \textbf{folding} of $\underline{\fc}$ along $M_{\beta[n]}$ is the open path $\varphi_{t_0,\beta[n]}(\underline{\fc})=(\varphi_{t_0,\beta[n]}(\fc),\varphi_{t_0,\beta[n]}(D),\varphi_{t_0,\beta[n]}(\underline{\fc})(1))$ defined by:
\begin{enumerate}
    \item $\varphi_{t_0,\beta[n]}(\fc)(t)=\begin{cases} \fc(t) \ &\text{if}\ t<t_0 \\ s_{\beta[n]}\fc(t) \ &\text{if}\ t\geq t_0   
    \end{cases}$,
    
    \item $\varphi_{t_0,\beta[n]}(D)^\varepsilon_t=\begin{cases} D^\varepsilon_t \ &\text{if}\ t<t_0 \ \text{or} \ (t,\varepsilon)=(t_0,-) \\ s_{\beta[n]}D^\varepsilon_t \ &\text{if}\ t>t_0 \ \text{or} \ (t,\varepsilon)=(t_0,+)
        
    \end{cases}$,
    \item $\varphi_{t_0,\beta[n]}(\underline{\fc})(1)=s_{\beta[n]}\underline{\fc}(1)$.
\end{enumerate}

  Note in particular that a folding at time $t_0=1$ only affects the end alcove. We may also fold decorated paths, in a similar fashion.

For $t_0<1$ (resp. $t_0=1$), we say that the folding is \textbf{positive} (or $C_\infty$-centrifuged) if $D^+_{t_0}$ (resp. $\underline{\fc}(1)$) is on the negative side of the wall $M_{\beta[n]}$. That is to say, if it is on the side not containing $C_0$ or equivalently if it is on the side containing $C^\infty_{\fc(t_0)}$.

\end{Definition}

\begin{Remark}
    Note that the folding of an admissible open path need not be admissible (if the folding is done at a time $t_0$ such that $\fc_+(t_0)\subset M_{\beta[n]}$).
\end{Remark}

We may now define $C_\infty$-Hecke open paths, which are particular open paths obtained trough successive positive foldings.
\begin{Definition}[$C_\infty$-Hecke open paths]\label{Definiton: Hecke open path}
    Let $\bx=\qp^\lambda w\in W^+$ be any spherical element, a \textbf{$C_\infty$-Hecke open path of type $\bx$} is an admissible open path $\underline{\fc}$ for which there exists a sequence $((t_i,\beta_i[n_i]))_{i\in \llbracket 0,r\rrbracket}$ such that:
    \begin{enumerate}
        \item The sequence $(t_i)_{i\in \llbracket 0, r\rrbracket}$ is a non-decreasing sequence in $[0,1]$.
        \item For any $i\in \llbracket 0,r\rrbracket$, we have $s_{\beta_{i-1}[n_{i-1}]}\dots s_{\beta_0[n_0]}(-t_i\lambda)\in M_{\beta_i[n_i]}$. 
        
        \item The open path $\underline{\fc}$ is obtained through successive positive foldings from the open segment $\underline\fs_\bx$:
        \begin{equation}
            \underline{\fc}=\varphi_{t_r,\beta_r[n_r]}\circ \dots \circ \varphi_{t_0,\beta_0[n_0]}(\underline\fs_\bx).
        \end{equation}
       
    \end{enumerate}
    If $\underline{\fc}$ is a $C_\infty$-Hecke open path, any such sequence $((t_i,\beta_i[n_i]))_{i\in \llbracket0,r\rrbracket}$ is called a \textbf{folding data} for $\underline{\fc}$. Note that, by condition $2$, the integers $(n_i)_{i\in \llbracket0,r\rrbracket}$ are completely determined from $(t_i,\beta_i)_{i\in \llbracket 0,r\rrbracket}$, hence we might drop them from our notation. 
    
    Given any spherical element $\bx\in W^+$ and any element $\by\in  W^+$, we let $\cC^\infty_\bx(\by)$\index{c@$\cC^\infty_\bx(\by)$} denote the set of $C_\infty$-Hecke open paths of type $\bx$ and with end alcove $C_\by$.

\end{Definition}

\begin{figure}[hhhh]

    \centering
    \includegraphics[width=\textwidth]{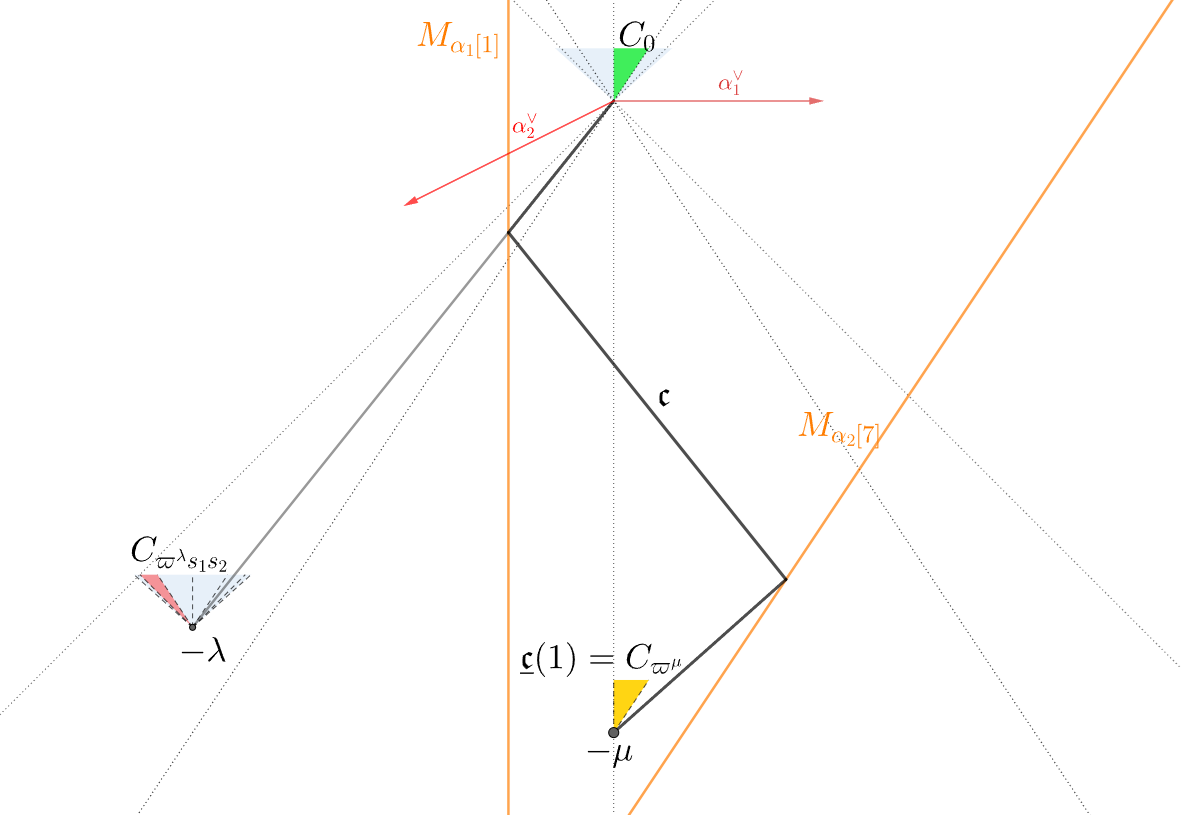}
    \caption{A $C_\infty$-Hecke open path of type $\qp^\lambda s_1 s_2$ and end alcove $C_{\qp^\nu}$.}
    \label{fig: chemindehecke}
\end{figure}

On figure \ref{fig: chemindehecke} we give an example of a $C_\infty$-Hecke open path $\underline{\fc}$ of type $\qp^\lambda s_1 s_2$, indicating its non-open component $\fc$ and its end alcove $\underline{\fc}(1)$. Note that $\lambda$ is regular, so the decoration needs not be precised. We have highlighted in orange two walls which correspond to the folding data $(\alpha_1[1],\alpha_2[7])$ for $\underline{\fc}$.

\begin{Remark}
    To a given type $\bx \in W^+$ and a given folding data $((t_i,\beta_i[n_i]))_{i \in \llbracket0,r\rrbracket}$, there is at most one associated $C_\infty$-Hecke open path. However, the converse is not true: a given admissible open path $\underline{\fc}$ may simultaneously be a $C_\infty$-Hecke open path of different types. For instance, if $\lambda \in Y^{++}$ and $\beta\in \Phi_+$ are such that $\langle \lambda,\beta\rangle \neq 0$, then $\underline\fs_{s_\beta \qp^\lambda}$ is a $C_\infty$-Hecke open path of type $s_\beta \qp^\lambda=\qp^{s_\beta\lambda}s_\beta$ (with trivial folding data) and of type $\qp^\lambda$  (with folding data $((0,\beta))$).
\end{Remark}

The positivity condition satisfied by $C_\infty$-Hecke open paths relates them with the affine Bruhat order, as detailed in Proposition \ref{Proposition: open paths give chains} below.
\begin{Proposition}\label{Proposition: open paths give chains}
    Let $\underline{\fc}=(\fc,D,\underline{\fc}(1))$ be a $C_\infty$-Hecke open path of type $\bx=\qp^\lambda w$ with folding data $((t_i,\beta_i[n_i]))_{i\in \llbracket0,r\rrbracket}$. Then, the following properties hold:
    \begin{enumerate}
        \item The end alcove $\underline{\fc}(1)$ is a special alcove: there exists $\by\in W^+$ such that $\underline{\fc}(1)=C_\by$.
        \item We have a chain for the affine Bruhat order:
        \begin{equation}\label{eq: openchain1}\by=s_{\beta_r[n_r]}\dots s_{\beta_0[n_0]}\bx  <\dots<s_{\beta_0[n_0]}\bx <\bx. \end{equation}
        \item The restriction $\fc|_{[t_r,1]}$ is a segment, and for all $(t,\varepsilon)\in [t_r,1]^{\pm}$, we have $D^\varepsilon_t=\pr_{\fc_\varepsilon(t)}(\underline{\fc}(1))$.
    \end{enumerate}
\end{Proposition}
\begin{proof}
For this proof only, we call $C_\infty$-Hecke open paths the open paths which are not necessarily admissible, but which are obtained through positive folding from open segments as in Definition~\ref{Definiton: Hecke open path}; the difference is that the direction of the decoration can vary at a point where the path does not fold (but it still only varies at finitely many points).
    We prove these three points by induction on $r$. If $r=-1$ then $\underline{\fc}$ is the open segment $\underline\fs_\bx$, and $\underline{\fc}(1)=C_\bx$ by definition, in particular it is a special alcove.
    
     Suppose that points 1-2-3 hold for $C_\infty$-Hecke open paths with folding data of length $r-1\geq-1$ and let $\underline{\fc}$ be a $C_\infty$-Hecke open path of type $\bx$ with a folding data $\big((t_k,\beta_k[n_k])\big)_{k\in\llbracket 0,r\rrbracket}$. 
    
    Then $\underline{\fc}_0:=\phi_{t_{r-1},\beta_{r-1}[n_{r-1}]}\circ \dots\circ \phi_{t_0,\beta_{0}[n_0]} (\underline\fs_\bx)$  is a $C_\infty$-Hecke open path of type $\bx$ and with a folding data of length $r-1$, so for which the result stands. Moreover $\underline{\fc}=\varphi_{t_r,\beta_r[n_r]}(\underline{\fc}_0)$ and this folding needs to be positive. Note that $\underline{\fc}(1)=s_{\beta_r[n_r]}\underline{\fc}_0(1)$, therefore since $\underline{\fc}_0(1)$ is a special alcove, so is $\underline{\fc}(1)$, which proves $1$.
    
    Moreover setting $\underline{\fc}_0(1)=C_{\by_0}$ and $\underline{\fc}(1)=C_\by$, we have $\by=s_{\beta_r[n_r]}\by_0$ and by iteration, to prove $2.$ it suffices to check that $s_{\beta_r[n_r]}\by_0<\by_0$. By the geometric interpretation of the affine Bruhat order, this is verified if and only if $\by_0$ is on the negative side of $M_{\beta_r[n_r]}$. Suppose first that $t_r=1$. Then since the folding is positive this is verified, so $\by< \by_0$, and point $3.$ is trivially verified in this case. Suppose now that $t_r<1$. Then since the folding is positive, we know that $D^+_{\fc_0,t_r}$ is on the negative side of $M_{\beta_r[n_r]}$. By point $3.$ and Lemma~\ref{Lemma: projection walls}, $D^+_{\fc_0,t_r}$ and $\underline{\fc}_0(1)$ are on the same side of this wall, hence $\by<\by_0$. Finally, point $3.$ for $\underline{\fc}$ is a consequence of point $3.$ for $\underline{\fc}_0$: for $(t,\varepsilon)\in [t_r,1]^{\pm}$, $D^\varepsilon_t=s_{\beta_r[n_r]}D^\varepsilon_{\fc_0,t}=s_{\beta_r[n_r]} \pr_{\fc_{0,\varepsilon(t)}}(\underline{\fc}_0(1))=\pr_{\fc_\varepsilon(t)}(\underline{\fc}(1))$.

\end{proof}

Using \cite[Corollary 3.4]{hebert2024quantum}, we deduce:

\begin{Corollary}\label{Corollary: finiteness open hecke paths}
Let $\bx$ and $\by$ be two elements of $W^+$, with $\bx$ spherical. The set $\cC^\infty_{\bx}(\by)$ is finite, and if it is non-empty, then $\by\leq\bx$. 
 \end{Corollary}
This is a variant of \cite[Theorem 5.54]{muthiah2019double} which extends, by \cite[Corollary 3.4]{hebert2024quantum}, from the untwisted affine ADE case to the general Kac-Moody case.

\subsection{A local characterization of $C_\infty$-Hecke open paths}
In this subsection, we give a local characterization of $C_\infty$-Hecke open paths. That is to say, in order to check whether an open path is a $C_\infty$-Hecke open path, we may only look locally at the points on which the path folds, see Lemma \ref{l_characterization_decorated_Hecke_paths}.
\begin{Definition}
    Let $a\in  \A$ and $C,\tilde{C}$ be two alcoves based at $a$. A  $(C_a^\infty,W^v_{a})$-chain from $C$ to $\tilde{C}$ is a  finite sequence $C^{(0)}=C,\ldots, C^{(r)}=\tilde{C}$ of alcoves of $\A$ based at $a$ such that there exist $\beta_1,\ldots, \beta_r\in \Phi$ such that for all $i\in \llbracket 1,r\rrbracket$, if $n_i=-\langle a,\beta_i\rangle$, we have: \begin{enumerate}
    \item[(i)]  $s_{\beta_i[n_i]}(C^{(i-1)})=C^{(i)}$,

    \item[(ii)]  $\beta_i[n_i](C^{(i-1)})$ and $\beta_i[n_i](C_{a})$  have opposite signs,

    \item [(iii)] $n_i\in \Z$.
\end{enumerate}

Then $(\beta_1[n_1],\ldots,\beta_r[n_r])$ (or simply $(\beta_1,\ldots,\beta_r)$) is called a \textbf{chain data} from $C$ to $\tilde{C}$.

\begin{Remark}
In \cite[5]{bardy2025twin}, $(C^\infty_a,W^v_a)$-chains are  sequences of vectors instead of sequences of alcoves. It is easy to pass from chains of alcoves to chains of vectors. Indeed, let  $C^{(0)},\ldots, C^{(r)}$ be a $(C^\infty_a,W^v_a)$-chain. Let $C_0^v,\ldots,C^v_r$ be the directions of these alcoves, let $\xi\in  \overline{C^v_f}$ and for $i\in \llbracket 1,r\rrbracket$, denote by $\xi_i$ the element of $W^v.\xi\cap \overline{C^v_i}$. 
Then if $(\xi_i')_{i\in \llbracket 0,r'\rrbracket}$ is obtained from $(\xi_i)_{i\in \llbracket 0,r\rrbracket}$ by keeping only one vector from each sequence $\xi_i,\xi_{i+1},\ldots$ of constant vectors, $(\xi_i')$ is a $(C^\infty_a,W^v_a)$-chain. It is also easy to pass from a sequence of vectors to a sequence of alcoves.
\end{Remark}

\end{Definition}

\begin{Lemma}\label{l_characterization_decorated_Hecke_paths}
Let $\underline{\fc}=(\fc,D,\underline{\fc}(1))$ be an admissible open path. Let $\bx\in W^+_{sph}$, we assume that $\fc$ is a $(-\lambda)$-path for $\lambda=\pr^{Y^+}(\bx)$. 

Let $\tilde{\underline{\fc}}(1)$ denote the unique alcove in $\cT^+_{\fc(1)}(\A)$ such that $d^+(D^-_1,\tilde{\underline{\fc}}(1))=d^+(D^-_{\bx,1},C_\bx)$ (recall that $D_\bx$ denotes the decoration on the open segment $\underline\fs_\bx$). Let $t_0,\ldots,t_{n+1}\in [0,1]$ be such that  $0=t_0<t_1<\ldots<t_{n+1}=1$ and such that $\{t_i\mid i\in \llbracket 1,n\rrbracket\}=\{t\in (0,1)\mid \fc'_-(t)\neq \fc'_+(t)\}$. Then $\underline{\fc}$ is a $C_\infty$-Hecke open path if and only if: \begin{enumerate}
    \item for every $i\in \llbracket 1,n\rrbracket$, there exists a $(C^\infty_{\fc(t_i)},W^v_{\fc(t_i)})$-chain from $D_{t_i}^+$ to ${\lim}_{t\to t_i^-} D_{t}^+$. 

    \item there exists a $(C^\infty_{\fc(1)},W^v_{\fc(1)})$-chain from $\underline{\fc}(1)$ to $\tilde{\underline{\fc}}(1)$. 
\end{enumerate} 
\end{Lemma}
\begin{proof} 
Let $\underline{\fc}$ be an admissible open path satisfying 1. and 2.  We prove that $\underline{\fc}$ is a $C_\infty$-Hecke open path by induction on $n(\underline{\fc})=n:=|\{t\in (0,1)\mid \fc'_-(t)\neq \fc'_+(t)\}|$. Let $(\gamma_1[p_1],\ldots,\gamma_u[p_u])$ be a chain data from $\underline{\fc}(1)$ to $\tilde{\underline{\fc}(1)}$. Let $\underline{\tilde{\fc}}=(\fc,D,\tilde{\underline{\fc}}(1))$. Then we have \begin{equation}\label{e_description_ufc}
    \underline{\fc}=\phi_{1,\gamma_1[p_1]}\circ \ldots \circ \phi_{1,\gamma_u[p_u]}(\underline{\tilde{\fc}}),
    \end{equation} and $\underline{\tilde{\fc}}$ satisfies 1. and \begin{equation}\label{e_continuity_condition}
        d^+(D^-_1,\tilde{\underline{\fc}}(1))=d^+(D^-_{\bx,1},C_\bx)
    \end{equation}
    Note that equality \eqref{e_continuity_condition} is invariant under applying foldings at times $t<1$ to $\tilde{\underline{\fc}}$.
    Moreover the foldings involved in \eqref{e_description_ufc} are positive, by definition of a chain data.

 Assume that $n>0$. Let $t_*=\max\{t\in (0,1)\mid \fc'_+(t)\neq \fc'_-(t)\}$. Let $(\beta_1[m_1],\ldots,\beta_r[m_r])$ be a chain data from  $D_{t_*}^+$ to $\lim_{t\to t_*^-} D_t^+$. Then we have $\lim_{t\rightarrow t_*^-}D_t^+=s_{\beta_r[m_r]}\circ \ldots \circ s_{\beta_1[m_1]}(D_{t_*}^+)$. 
 Let $\underline{\hat{\fc}}=\phi_{t_*,\beta_r[m_r]}\circ \ldots \circ \phi_{t_*,\beta_1[m_1]}(\underline{\tilde{\fc}})$. Write $\underline{\hat{\fc}}=(\hat{\fc},\hat{D},\underline{\hat{\fc}}(1))$. Then we have: \begin{align*}\hat{D}_{t_*}^+&=s_{\beta_r[m_r]}\circ \ldots \circ s_{\beta_1[m_1]}( D_{t_*}^+)\\
 &=s_{\beta_r[m_r]}\circ \ldots \circ s_{\beta_1[m_1]}\circ s_{\beta_1[m_1]}\circ \ldots \circ s_{\beta_r[m_r]}(\lim_{t\to t_*^-} D_t^+)=\lim_{t\to t_*^-} D_t^+=\lim_{t\to t_*^-} \hat{D}_t^+. \end{align*} 
 
  Let $\varepsilon>0$ be such that $\hat{\fc}$ is affine on $[t_*-\varepsilon,t_*]$.  By the admissibility condition, the direction of $D_t^+$ does not depend on $t\in [t_*-\varepsilon,t_*)$. Denote by $C^v$ this direction. Then for $t\in [t_*-\varepsilon,t_*)$, we have $(\hat{\fc})'_+(t)=\mu$, where $\mu$ is the unique element of $W^v.(-\lambda)\cap \overline{C^v}$. But $(\hat{c})'_+(t_*)$ is the unique element of $\overline{C^v}\cap W^v.(-\lambda)$ (since $C^v$ is also the direction of $\hat{D}_{t_*}^+$). Therefore $\hat{\fc}$ does not fold at $t_*$ and $n(\underline{\hat{\fc}})=n(\underline{\fc})-1$. As $\max\{t\in (0,1)\mid (\hat{\fc})'_+(t)\neq (\hat{\fc})'_-(t)\}<t_*$ (if the set of foldings of $\hat{\fc}$ is non-empty) and as $\underline{\hat{\fc}}$ satisfies \eqref{e_continuity_condition} (with $\tilde{\ }$ replaced by $\hat{\ }$), we obtain $\underline{\fc}$ inductively by successive positive foldings of a path   $\underline{\mathfrak{d}}$ satisfying $n(\underline{\mathfrak{d}})=0$  and \eqref{e_continuity_condition} with $\underline{\tilde{\fc}}$ replaced by $\underline{\fd}$. Since $\underline\fd$ is admissible, has underlying non-open path $\fs_\lambda$ and satisfies \eqref{e_continuity_condition}, we have $\underline\fd=\underline\fs_\bx$ and therefore $\underline{\fc}$ is a $C_\infty$-Hecke open path of type $\bx$.

 Conversely, let $\underline{\fc}$ be a $C_\infty$-Hecke open path of type $\bx\in W^+_{sph}$, and let $\tilde{\underline{\fc}}(1)$ be as defined in the statement.  Write $\underline{\fc}=\phi_{t_r,\beta_r[m_r]}\circ \ldots \circ \phi_{t_0,\beta_0[m_0]}(\underline{\fs}_{\bx}),$ where $((t_i,\beta_i[m_i]))_{i\in \llbracket 0,r\rrbracket}$ is a folding data for $\underline{\fc}$ and assume that $r>0$. Let $E=\{t\in [0,1]\mid \exists i\in \llbracket 0,r\rrbracket, t_i=t\}$. Let $t_*\in E$.  Set $J=\{i\in \llbracket 0,r\rrbracket\mid t_i=t_*\}$. Then $J$ is an integers interval. Let $r'=|J|-1$ and $r_1=\min(J)$. If $t_*<1$ (resp. $t_*=1$), set $C^{(0)}=D_{t_*}^+$ (resp. $C^{(0)}=\tilde{\underline{\fc}}(1)$) and define inductively $(C^{(i)})_{i\in \llbracket 1,r'\rrbracket}$ by  $C^{(i)}=s_{\beta_{r'+r_1-1-i}[m_{r'+r_1-1-i}]}(C^{(i-1)})$, for $i\in \llbracket 1,r'\rrbracket$. Then $(C^{(i)})_{i \in \llbracket 1,r'\rrbracket}$ is a $(C^\infty_{\fc(t_*)},W^v_{\fc(t_*)})$-chain from $D_{t_*}^+$ to $\lim_{t\to t_*^-}D_t^+$  (resp. $\underline{\fc}(1)$ to $\tilde{\underline{\fc}}(1)$), which proves the lemma. 
\end{proof}

\section{Lifts of paths in the masure and application to Kazhdan-Lusztig polynomials}\label{s_lifts_paths}

The motivation for the introduction of $C_\infty$-Hecke open paths is that they appear as images by certain retractions of open segments in a masure, an analog of the Bruhat-Tits building for Kac-Moody groups. We introduce Kac-Moody groups, masures and twin masures in the first subsection. Each $C_\infty$-Hecke open path can thus be lifted, in any masure built upon $\A$, to numerous open segments, and we are able to compute these lifts, it is the content of subsection~\ref{ss_counting_lifts}. From this computation and finiteness of the number of $C_\infty$-Hecke open paths, we can define Kazhdan-Lusztig $R$-polynomials on $W^+$, which is the content of subsection \ref{subsection: Definition Kazhdan-Lusztig Polynomials}. Upon certain masure-theoretic conjectures, this gives an expression for the cardinal of the equivalent of Formula \eqref{eq: RPolreductive_intro} for Kac-Moody groups.

\subsection{Kac-Moody groups, masures and twin masures}

\subsubsection{Split Kac-Moody groups over valued fields and masures}

 Let  $\G=\G_{\mathcal{D}}$\index{g@$G$} be the group functor associated in \cite{tits1987uniqueness} with  the  generating root datum $\mathcal{D}$, see also \cite[Chapter 7]{marquis2018introduction} or \cite[8]{remy2002groupes}. Let $\cK$\index{k@$\cK$}\index{o@$\omega$} be a   field equipped with a valuation $\omega:\cK\twoheadrightarrow \Z\cup\{+\infty\}$. Let $G= \G(\cK)$  be the \textbf{split Kac-Moody group over $\cK$ associated with $\mathcal{D}$}.  The group $G$ is generated by the following subgroups:\begin{itemize}
\item the fundamental torus $T=\TT(\cK)$, where $\TT=\mathrm{Spec}(\Z[X])$,

\item the root subgroups $U_\alpha=\mathbb U_\alpha(\cK)$\index{u@$U_\alpha$}, isomorphic to $(\cK,+)$ by an isomorphism $x_\alpha$\index{x@$x_\alpha$}.
\end{itemize}

In \cite{gaussent2008kac} and \cite{rousseau2016groupes} (see also \cite{rousseau2017almost}) the authors associate a masure $\I=\I(\G,\TT,\cK,\omega)$\index{i@$\I$} on which the group $G$ acts. We recall briefly the construction of this masure. Let $N$\index{n@$N$} be the normalizer of $T$ in $G$.

A first step is to  define an action $\nu=\nu_\omega$ of $N$ on $\A$, see \cite[3.1]{gaussent2008kac}.
For $t\in T$, denote by $a_t$ the unique element of $\A$ satisfying $\chi(a_t)=-\omega(\chi(t))$, for every $\chi\in X$. If $t\in T$, $\nu(t)$ is the translation on $\A$ by the vector $a_t$. For $n\in N$, $\nu(n)$ is a translation if and only if $n\in N$ and we have $\nu(N)=W^v\ltimes Y$.

The masure  $\I=\I(\G,\cK,\omega) $ is defined to be the set $G\times \A/\sim$, for some equivalence relation $\sim$ (see \cite[Definition 3.15]{gaussent2008kac}). Then $G$ acts on $\I$ by $g.[h,x]=[gh,x]$ for $g,h\in G$ and $x\in \A$, where $[h,x]$ denotes the class of $(h,x)$ for $\sim$. The map $x\mapsto [1,x]$ is an embedding of $\A$ in $\I$ and we identify $\A$ with its image. Then $N$ is the stabilizer of $\A$ in $G$  and it acts on $\A$ by $\nu$. If $\alpha\in \Phi$ and $a\in \cK$, then $x_\alpha(a)\in U_\alpha$ fixes the half-apartment $D_{\alpha,\omega(a)}=\{y\in \A\mid \alpha(y)+\omega(a)\geq 0\}$ and for all $y\in \A\setminus D_{\alpha,\omega(a)}$, $x_\alpha(a).y\notin \A$.

 An \textbf{apartment} is a set of the form $g.\A$, for $g\in G$. We have $\I=\bigcup_{g\in G} g.\A$.  Then $\I$ satisfies axioms (MA I), (MA II) and (MA III) of \cite[Corollary 3.7]{hebert2022new}, which are equivalent to the axioms (MA1) to (MA5) defined by Rousseau in \cite{rousseau2011masures}. These axioms describe the following properties. \begin{itemize}
 \item[(MA I)] Let $A$ be an apartment of $\I$. Then $A=g.\A$, for some $g\in G$. We can then transport every notion which is preserved by $\nu(N)=W^v\ltimes Y$ to $A$ (in particular, we can define a segment, a wall, an alcove, ... in  $A$).

 \item[(MA II)] This axiom asserts that if $A$ and $A'$ are two apartments, then $A\cap A'$ is a finite intersection of affine half-apartments (i.e of sets of the form $h.D_{\alpha,k}$, for $\alpha\in \Phi$, $k\in \Z$, if $A=h.\A$) and there exists $g\in G$ such that $A'=g.A$ and $g$ fixes $A\cap A'$.

  \item[(MA III)] This axiom asserts that for some pairs of filters on $\I$, there exists an apartment containing them. This axiom is the building theoretic translation of some decompositions of $G$ (e.g Iwasawa decomposition).
 \end{itemize}

 The following proposition extends \cite[Proposition 7.4.8]{bruhat1972groupes} to masures.

\begin{Proposition}\label{p_BT_7.4.8}
    Let $g\in G$. Then $\A\cap g^{-1}.\A$ is enclosed and there exists $n\in N$ such that $g.x=n.x$ for every $x\in \A\cap g^{-1}.\A$.
\end{Proposition}

\begin{proof}
    We assume that $\Omega:=\A\cap g^{-1}.\A$ is non-empty. Then it is enclosed by (MA II) and  there exists $h\in G$ such that $hg.\A=\A$ and $h$ fixes $\A\cap g.\A$. Then $hg$ stabilizes $\A$ and thus it belongs to $N$, by \cite[5.7 5)]{rousseau2016groupes}. We get the proposition by setting $n=hg$.
\end{proof}

\subsubsection{Masures associated with $\mathbb{G}(\bk[\qp,\qp^{-1}])$, $\mathbb{G}(\bk(\qp))$ and $\mathbb{G}\left(\bk(\!(\qp^\varepsilon)\!)\right)$}\label{sss_different_masures} 
For a given Kac-Moody group functor $\mathbb G$ and a given finite field $\mathds k$, the loop group of $\mathbb G$ over $\mathds k$ admits several incarnations in the litterature, which give rise to different masures. In this paragraph we introduce some of these choices and compare the different masures associated.

Let $\bk$\index{k@$\bk$} be a finite field and $\qp$ be an indeterminate. Let $\varepsilon\in \{\ominus,\oplus\}$. Let $\omega_\varepsilon$  be the usual valuation on $\bk(\!(\qp^\varepsilon)\!)$:  $\omega_\oplus(\sum_{k=-\infty}^{+\infty} x_k \qp^{\varepsilon k})=k_0$ if $\sum_{k=-\infty}^\infty x_k  \qp^{\varepsilon k}\in \bk(\!(\qp^{\varepsilon})\!)\setminus\{0\}$ is such that $x_{k_0}\neq 0$ and $x_k=0$ for $k\in \Z_{<k_0}$, where $\qp^\oplus=\qp$ and $\qp^{\ominus}=\qp^{-1}$. Let $\cK=\bk(\qp)$\index{k@$\cK$}. Let $\widehat{\I}_{\varepsilon}=\I(\mathbb{G},\bk(\!(\qp^\varepsilon)\!),\omega_{\varepsilon})$ and $\I_{\varepsilon}=\I(\mathbb{G},\cK,\omega_{\varepsilon})$. 

Let $\cO=k[\qp,\qp^{-1}]$, one defines $G_{twin}=\langle x_\alpha(\cO), \TT(\cO),\alpha\in \Phi \rangle\subset G$\index{g@$G_{twin}$}, which is the group $\G^{min}_{\mathcal{D}}(\cO)$ defined by Marquis in \cite[Definition 8.126]{marquis2018introduction}.

We denote by $\A_{\varepsilon}$\index{a@$\A_{\oplus},\A_{\ominus}$} the standard apartment of $\I_{\varepsilon}$. As affine spaces equipped with a hyperplane arrangement, $\A_{\oplus}$ and $\A_{\ominus}$ are copies of $(\A,\mathcal M^a)$, where $\A=Y\otimes \R$. However the actions of $T$ on $\A_{\oplus}$ and $\A_{\ominus}$ differ, since they take into account $\omega_{\oplus}$ and $\omega_{\ominus}$, and for the same reason $x_\alpha(u).\A_{\oplus}\cap \A_{\oplus}$ is in general different from $x_{\alpha}(u). \A_{\ominus}\cap \A_{\ominus}$, for $u\in \cK$ and $\alpha\in \Phi$.

Let $\varepsilon\in \{\ominus,\oplus\}$. The masure $\I^{twin}_{\varepsilon}$ of $(G_{twin},\omega_{\varepsilon})$ is $G_{twin}.\A_{\varepsilon}$ equipped with the set of apartments $\{g.\A_{\varepsilon}\mid g\in G_{twin}\}$. 
As $(\cO,\omega_\varepsilon)$ is dense in $\bk(\!(\qp^\varepsilon)\!)$,  the masures $\widehat{\I}_\varepsilon$, $\I_{\varepsilon}$ and  $\I^{twin}_\varepsilon$ are the same when we regard them as sets, by \cite[Corollary 3.7]{bardy2025twin} (in particular, $\I_\varepsilon=G_{twin}.\A_\varepsilon$). More precisely, if we regard $\A_\varepsilon$ as a subset of $\I_\varepsilon$, then by \cite[Theorem 3.6]{bardy2025twin}, if $A_\varepsilon$ is an apartment of $\I_\varepsilon$, and $P$ is a  bounded subset  of $A_\varepsilon$, there exists $g_{twin}\in G_{twin}$ such that $P\subset g_{twin}.\A_\varepsilon$, and the same holds with $\mathbb{G}(\bk(\!(\qp^\varepsilon)\!))$ instead of $\mathbb{G}(\cK)$ if we regard  $\A_\varepsilon\subset\widehat{\I}_\varepsilon$. Therefore when we are interested in local questions on $\I_\varepsilon$ (for example when we consider alcoves or tangent buildings), we can work with $\mathbb{G}(\bk(\!(\qp^\varepsilon)\!))$, $G$ or $G_{twin}$ indifferently.  However, the set of apartments of $\widehat{\I}_\varepsilon$  strictly contains the set of apartments of  $\widehat{\I}_\varepsilon$, which strictly contains the set of apartments of $\I^{twin}_{\varepsilon}$.

\subsubsection{Twin masure associated with Kac-Moody groups over Laurent polynomials}

In \cite{bardy2025twin}, the authors study the interplay between $\I_\oplus$ and $\I_\ominus$, in order to obtain an analogue of the twinnings of buildings, adapted to the frameworks of masures. We recall some of the key properties obtained there.

A \textbf{twin apartment} is a set of form $g.(\A_{\oplus}\sqcup \A_{\ominus})$, for some $g\in G_{twin}$. Let $\varepsilon\in \{\ominus,\oplus\}$. An apartment of the form $g.\A_{\varepsilon}$, for some $g\in G_{twin}$ is called \textbf{twinnable}. Then the map from the set of twinnable apartments of $\I_\oplus$ to the set of twinnable apartments of $\I_{\ominus}$ defined by $A_{\oplus}=g.\A_{\oplus}\mapsto A_{\ominus}= g.\A_{\ominus}$, for $g\in G_{twin}$ is a well-defined bijection by \cite[4.1.10]{bardy2025twin}). If $A_\oplus$ is a twinnable apartment of $\I_{\oplus}$, the corresponding apartment of $\A_{\ominus}$ is called its \textbf{twin}.

If $P$ is a bounded subset of an apartment of $\I_\oplus$, then there exists a twinnable apartment containing $P$, by \cite[Theorem 3.6]{bardy2025twin}. By \cite[Proposition 3.8 and Theorem 4.7]{bardy2025twin}, if $\varepsilon\in \{\ominus,\oplus\}$ and if $A_{\varepsilon}$ and $B_\varepsilon$ are two twinnable apartments, then there exists $g\in G_{twin}$ fixing $(A_{\ominus}\cap B_{\ominus})\cup (A_{\oplus}\cap B_\oplus)$.

A priori, the group $\mathbb{G}\left(\bk(\!(\qp)\!)\right)$ has no natural action on $\I_\ominus$. The groups $G=\mathbb{G}(\cK)$ and $G_{twin}$ both act on $\I_\oplus$ and $\I_\ominus$. However, the group $G_{twin}$ is more adapted in order to define the retraction $\rho_{C_\infty}$.

From now on, we write $\A$ instead of $\A_{\oplus}$. We add a subscript $\ominus$ to the objects related to $\I_{\ominus}$.

\paragraph{Iwahori subgroups} Let $C_{0}=\mathrm{germ}_{0}(C^v_f)\subset \A$\index{c@$C_{0}$}
(resp. $C_\infty =\mathrm{germ}_{0_{\ominus}}(-C^v_{f,\ominus})\subset \A_{\ominus}$\index{c@$C_\infty$}) be the \textbf{fundamental alcove} of $\A$ (resp. $\A_{\ominus}$).

We denote by $I$ (resp. $I_0$, $I_\infty$)\index{i@$I$, $I_0$, $I_\infty$} the fixator of $C_0$ (resp. $C_0$, $C_\infty$) in $G$ (resp. in $G_{twin}$, in $G_{twin}$). Descriptions of $I$ are given in \cite[5.7 and Définition 5.3]{rousseau2016groupes} and \cite[2.4.1 1)]{bardy2025twin}. The map $g\mapsto g.C_0$ induces a bijection between $G/I$ and the set of positive alcoves of $\I_\oplus$ based at an element of $g.0_\A$. By \ref{sss_different_masures}, this set is also the set of positive alcoves based at an element of $G_{twin}.0_\A$ and we have a natural bijection between $G/I$ and $G_{twin}/I_0$.

\subsubsection{Prerequisites on masures}\label{section : prerequisites_masures}

\paragraph{Tits preorder on $\I_\varepsilon$} Let $\varepsilon\in \{\ominus,\oplus\}$. For $x,y\in \I_\varepsilon$, we write $x\leq_\varepsilon y$ if there exists $g\in G$ such that $g.x,g.y\in \A_\varepsilon$ and $g.y-g.x\in \sT_\varepsilon$. This does not depend on the choice of $g$. This defines a $G$-invariant preorder on $\I_\varepsilon$ by \cite[Théorème 5.9]{rousseau2011masures}. When $\G$ is reductive this preorder is trivial since $\sT_\varepsilon=\A_\varepsilon$. However in the non-reductive case, $\leq_\varepsilon$ is a crucial tool  which is a partial remedy to the existence of pairs of points which are not contained in a common apartment.

\paragraph{Retractions centered at alcoves}

Let $\cE$ (resp. $\cE_{C_\infty,\oplus}$) \index{E@$\cE$, $\cE_{C_\infty,\oplus}$} be the set of points of $\I_\oplus$ such that there exists an apartment $A_x$ (resp. a twinnable apartment $A_{x,\oplus}$) of $\I_\oplus$ containing $C_0$ (resp. whose twin $A_{x,\ominus}$ contains $C_\infty$).  One defines  a \textbf{retraction $\rho_{C_0}$ on 
$\A$ centered at $C_0$} and a \textbf{retraction $\rho_{C_\infty}$ onto $\A$  centered at $C_\infty$} as follows. Let $x\in \cE$ (resp. $x\in \cE_{C_\infty,\oplus}$).
Then by (MA II) (resp. by \cite[Theorem 4.7]{bardy2025twin}), there exists $g\in G$ (resp. $g\in G_{twin}$) such that $g.A_x=\A$ (resp. $g.(A_{x,\oplus}\sqcup A_{x,\ominus})=\A\sqcup \A_{\ominus}$) and $g$ fixes $A_x\cap \A$ (resp. $(\A\cap A_{x,\oplus})\sqcup (\A_{\ominus}\cap A_{x,\ominus})$ (with the same notation as above). One then sets $\rho_{C_0}(x)=g.x$ (resp. $\rho_{C_\infty}(x)=g.x$). This is well-defined, independently of the choices made and thus  it defines an $I$-invariant retraction $\rho_{C_0}:\cE\rightarrow \A$ and an $I_\infty$-invariant retraction $\rho_{C_\infty}:\cE_{C_\infty,\oplus}\rightarrow \A$. 
When $\G$ is reductive, $\cE=\cE_{C_\infty,\oplus}=\A$, however in general,   these sets can be proper subsets of $\A$ (by \cite[Exemple 4.12.3 c)]{rousseau2016groupes} and \cite[6.5]{bardy2025twin}) and determining these sets is a difficult task. By \cite[Proposition 5.4]{rousseau2011masures} or \cite[Proposition 5.17]{hebert2020new}, $\cE\supset \I_{\oplus,\geq 0}\cup \I_{\oplus,\leq 0}$, where  $\I_{\oplus,R0}=\{x\in \I_{\oplus}\mid xR 0\}$ for $R\in \{\leq_{\oplus},\geq_{\oplus}\}$. However, very little is known on $\cE_{C_\infty,\oplus}$. It is conjectured in \cite[4.4.1]{bardy2025twin} that $\cE_{C_\infty,\oplus}\supset \I_{\oplus,\geq 0}\cup \I_{\oplus,\leq 0}$, but no case of the conjecture is known. In a  recent preprint,  Patnaik \cite{patnaik2024local} proves a  local Birkhoff decomposition of a completion of $G$ when $\mathbb{G}$ is affine. This decomposition seems to be related to this conjecture in the affine case, but we do not know to what extent, since the groups and subgroups considered are slightly different. 

Note that there exist elements of $\mathbb{T}(\bk(\qp))$ which fix $\A_{\ominus}$ but translate $\A_{\oplus}$ non-trivially. Therefore $\rho_{C_\infty}$ is not $\mathrm{Fix}_{G}(C_\infty)$-invariant and this is the main reason why we consider $G_{twin}$.

\paragraph{Tangent buildings}
For any $x \in \I_\oplus$ and $\varepsilon \in \{+,-\}$, let $\cT^\varepsilon_x(\I_\oplus)$\index{t@$\cT^\varepsilon_x(\I_\oplus)$} denote the set of alcoves of sign $\varepsilon$ based at $x$. Explicitly, an element of $\cT^\varepsilon_x(\I_\oplus)$ is the image $g.a$ of an alcove $a\in \cT^\varepsilon_y(\A)$ for some $y\in \A$ and $g\in G$ such that $g.y=x$.

We extend the $W^v$-distance $d^\varepsilon$\index{d@$d^\varepsilon,d^+,d^-$}  (resp. the codistance $d^\ast$\index{d@$d^\ast$}) defined in \eqref{e_def_d_epsilon} (resp. \eqref{e_def_d_ast}) to a $W^v$-valued distance $d^\varepsilon:\cT_x^\varepsilon(\I_\oplus)\times \cT_x^\varepsilon(\I_{\oplus})\rightarrow W^v$ (resp. a codistance $d^\ast:\cT^+_x(\I_\oplus)\times \cT^-_x(\I_\oplus) \cup \cT^-_x(\I_\oplus)\times\cT^+_x(\I_\oplus) \rightarrow W^v$). Let   $a_1,a_2\in \cT^\varepsilon_x(\I_\oplus)$ (resp. $(a_1,a_2)\in \cT^+_x(\I_\oplus)\times \cT^-_x(\I_\oplus)$). By \cite[Proposition 5.17]{hebert2020new} there exists an apartment containing both of them, so by (MA I) a $g\in G$ such that $g.a_1,g.a_2$ lie in $\cT^\varepsilon_y(\A)$, for some $g\in G$ such that $g.x=y\in \A$. One sets $d^\varepsilon(a_1,a_2)=d^\varepsilon(g.a_1,g.a_2)$ (resp. $d^\ast(a_1,a_2)=d^\ast(g.a_1,g.a_2)=d^\ast(a_2,a_1)^{-1}$).  By (MA II) these functions are well-defined (do not depend on the choices we made).

We say that $a_-,a_+$ are \textbf{opposite} if $d^\ast(a_-,a_+)=1_{W^v}$. Then for any $x$ and $\varepsilon\in \{+,-\}$, $(\cT^\varepsilon_x(\I_\oplus),d^\varepsilon)$ defines a building of type $(W^v,S)$ (where $S$ is the set of simple reflections of $W^v$: $S=\{s_i\mid i\in I\}$) \index{S@$S$}, with system of apartments induced by the one of $\I_\oplus$. The codistance $d^\ast$ defines a twinning, we say that $(\cT^\pm_x(\I_\oplus),d^\pm,d^\ast)$ is the \textbf{local tangent building} at $x$, it is thus a twin building of type $(W^v,S)$ and of model twin apartment $\cT^\pm_x(\A)$ (see \cite[\S 5.1, \S 5.8]{abramenko2008buldings} for the axioms of buildings and twin buildings, and \cite[Propositions 4.7, 4.8]{gaussent2008kac} for proofs). Any element $g\in G$ induces an isomorphism of twin buildings: $\cT^\pm_x(\I_\oplus)\cong \cT^\pm_{g.x}(\I_\oplus)$.

\begin{Remark}
    Note that, in this paper, $\cT^{\pm}_x(\I_\oplus)$ is a set of alcoves and not a set of segment germs (unlike for instance in \cite[\S 4.5]{gaussent2008kac}), because we consider the tangent buildings as abstract buildings, of which the set of segment germ is a geometric realization. In \cite[\S 4.5]{gaussent2008kac}, there are two distinct tangent building structures, the one exposed and used here corresponds to the unrestricted structure.
\end{Remark}
\paragraph{Thickness of the tangent buildings}
Let $(\cT,d)$ be an abstract building of type $(W,S)$, and let $s\in S$.  A \textbf{panel} of type $s$ is a subset of the form \index{p@$p(C,s)$} $p(C,s):=\{C'\in \cT\mid d(C,C')\in \{1,s\}\}$ for some alcove $C$. We say that $C'$ dominates $p(C,s)$ if $C'\in p(C,s)$; each alcove dominates exactly one panel of each type. Note that, for $C,C'$ two distinct alcoves  $p(C,s)=p(C',s)$ if and only if $d(C,C')=s$ and a panel is entirely defined by two of its dominating alcoves. The \textbf{thickness} of $p(C,s)$ is $|p(C,s)|-1$, a panel is called \textbf{thin} if its thickness is equal to $1$, and \textbf{thick} otherwise. Note that an isomorphism of buildings sends a given panel to another panel of same type and same thickness.

In general, the local tangent buildings in a masure admit  thin panels. More precisely, suppose that $x\in \A$ and $a=a(x,\varepsilon w)\in \T^\varepsilon_x(\A)$ is an alcove based at $x$ (with $\varepsilon \in \{+,-\}, w\in W^v$). Then, for $s_i\in S$, the panel $p(a,s_i)$ is identified with $a(x,\varepsilon w)\cap a(x,\varepsilon ws_i)$: an alcove $\tilde a \in T_x(\I_\oplus)$ dominates $p(a,s_i)$ if and only if it contains the filter $a(x,\varepsilon w)\cap a(x,\varepsilon ws_i)$. Let $\beta=w(\alpha_i)$ the root associated to $(w,s_i)$, then the \textbf{support} of the panel $p(a(x,\varepsilon w),s_i)$ is the (affine or ghost) wall $M_{\beta[-\langle x,\beta\rangle]}=\{y\in \A \mid \langle y,\beta\rangle =\langle x,\beta\rangle\}$, it is the only wall containing the filter-theoretic intersection $a(x,\varepsilon w)\cap a(x,\varepsilon ws_i)$. Note that the support does not depend of the sign $\varepsilon$. The panel is thick if and only if its support is an affine wall (an element of $\mathcal M^a$), if and only if $\langle x,\beta\rangle \in \mathbb Z$. Moreover, for split Kac-Moody groups, the thickness of thick panels is constant equal to $|\mathds k|$.

This extends to any $x\in \I_\oplus$ and any panel $p$ of $\cT_x^\pm(\I_\oplus)$, as there always exists an element $g$ of $G$ which sends $x$ to an element of $\A$, and $p$ to a panel lying in $\cT_{gx}(\A)$.

We deduce that, for $x\in \A$, $\cT_x^\pm(\I_\oplus)$ is thick (in the sense that all its panels are thick) if and only if $\langle x,\alpha_i\rangle \in \mathbb Z$ for all simple root $\alpha_i$, in particular this is verified if $x\in Y$. Since elements of $G$ induce isomorphism of tangent buildings, this extends to all $x\in G.Y=G.0_\A\subset \I_\oplus$ (which are called the "special vertices" of $\I_\oplus$). Moreover for any special point $x$, the tangent twin building $\cT^\pm_x(\I_\oplus)$ is $|\mathds k|$-homogeneous: all its panels have the same thickness $|\mathds k|$. In general, $\cT^\pm_x(\I_\oplus)$ is $(|\mathds k|,1)$-semi-homogeneous and the set of thick panels depends on the $G$-orbit of $x$. For $x\in \I_\oplus$ we denote by $\cP_{x}$ \index{P@$\cP_{x}$} the set of thick panels of the tangent building $\cT^\pm_x(\I_\oplus)$ in $\A$.

\begin{Remark}
    If we forget about thin panels, we obtain the restricted building structure (see \cite[\S 4.5]{gaussent2008kac}) $^r\cT^\pm_x(\I_\oplus)$ which is a quotient of $\cT^\pm_x(\I_\oplus)$ and a $|\mathds k|$-homogeneous twin building. However its type depends on $x$ and may be a Coxeter group of infinite rank, we work with the unrestricted structure in order to avoid this difficulty.
\end{Remark}

In order to work in local tangent buildings, we introduce galleries and building-theoretic retractions, with the formalism of abstract twin buildings.
\begin{Definition}[Galleries and retractions in a building]\label{Def galleries}

  Let $W$ be a Coxeter group with set of simple reflections $S$, and let $(\cT^\pm,d^\pm,d^\ast)$ be a twin building of type $(W,S)$ (in the sense of \cite[5.8]{abramenko2008buldings}).
  \begin{itemize}  
        \item Let $w\in W$ and let $\mathbf i_w=(s_1,\dots,s_n)\in S^n$ be a reduced $S$-expression of $w$. A sequence of chambers $(C_0,\dots,C_n)\in (\cT^+)^n$ forms a gallery of type $\mathbf i_w$ if $C_i$ lies in the panel $p(C_{i-1},s_i)\}$ for all $i\in \llbracket1,n\rrbracket$ (that is to say; $d(C_{i-1},C_i)\in \{1_W,s_i\}$. We say that it is a \textbf{minimal gallery} if there is no shorter gallery from $C_0$ to $C_n$, in particular this implies that $d(C_{i-1},C_i)\neq 1_W$ for all $i$. Note that, in this case, we have $d(C_0,C_n)=s_1\dots s_n=w$. 
        If $i$ is such that $C_{i-1}=C_i$ we say that the gallery is folded along the panel $p(C_i,s_i)$. 

    \item Let $(A^+,A^-)$ be a pair of twin apartments. Let $\mathbf g=(C_0,..,C_n)$ be a gallery in $A^+$ of type $\mathbf i_w=(s_1,..,s_n)$ and let $C_-$ be a chamber lying in $A^-$. We say that $\mathbf g$ is \textbf{$C_-$-centrifugally folded} if, for any $i\in \llbracket1,n\rrbracket$ such that $C_{i-1}=C_i$, we have that $d^\ast(C_-,C_i)<d^\ast(C_-,C_i)s_i$. Geometrically, this means that the wall supporting $p(C_i,s_i)$ (along which the gallery is folded) separates $C_-$ from $C_{i-1}$. More generally, we say that $\mathbf g$ crosses $M$, or folds along $M$ away from $C_-\in A^\pm$ (resp. towards $C_-$) if and only if, for every panel $p$ supported by $M$ and $i \in \llbracket1,n\rrbracket$ such that $C_{i-1},C_i\in p$, $C_{i-1}$ and $C_-$ lie on the same side of $M$ (resp. on opposite sides of $M$).

    \item Let $\mathbf g=(C_0,\dots,C_n)$ be a gallery in $A^+$ and $\rho_{A^\pm}$ be a retraction $\cT^\pm\rightarrow A^\pm$. Then a \textbf{$\rho_{A^\pm}$-lift} of $\mathbf g$ is a gallery $(\tilde C_0,\dots,\tilde C_n)$ in $\cT^+$ such that $\rho_{A^\pm}(\tilde C_i)=C_i$ for all $i\in \llbracket0,n\rrbracket$. A \textbf{$\rho_{A^\pm}$-minimal lift} is a $\rho_{A^\pm,x}$-lift which is also a minimal gallery. 
    
    \item If $C^-\in A^-$, we denote by $\rho_{A^\pm,C^-}$ the retraction centered at $C^-$ and with image $A^\pm$. It sends any alcove $C\in \cT^+$ (resp. $C\in \cT^-)$ to the unique alcove $C'\in A^+$ (resp. $C'\in A^-$) such that $d^\ast(C',C^-)=d^\ast(C,C^-)$  (resp. $d^-(C',C^-)=d^-(C,C^-)$).
    
\end{itemize}
\end{Definition}

\paragraph{Projections in the masure}
Like other notions, we can extend the definition of projections from the fundamental apartment to more general points and segment germs in the masure as follows.

Suppose that $\tilde x \leq_\oplus \tilde y\in \I_\oplus$ (resp. $\tilde x \geq_\oplus \tilde y\in \I_\oplus$) and let $a$ be an alcove based at $\tilde y$. Then by \cite[Proposition 5.17 (ii)]{hebert2020new} there is an apartment $A_\oplus$ containing $a$ and $\tilde x$. Therefore there is $g\in G$ such that $\A=g.A_\oplus$ and $g.\tilde y-g.\tilde x \in \sT$ (resp. $g.\tilde y-g.\tilde x \in -\sT$). We then define $\pr_{\tilde x}(a)$ as $g^{-1}.\pr_{g.\tilde x}(a)\in \cT^+_{\tilde x}(\I_\oplus)$ (resp. $g^{-1}.\pr_{g.\tilde x}(a)\in \cT^-_{\tilde x}(\I_\oplus)$). If $\fs_\varepsilon(t)$ is a segment germ based at $\bx$, then (by \cite[Proposition 5.17 (ii)]{hebert2020new} again) we can suppose that $A_\oplus$ also contains $\fs_\varepsilon(t)$, and we can therefore define $\pr_{\fs_\varepsilon(t)}(\by)$ similarly. Note that, for $\varepsilon\in \{+,-\}$, any apartment containing $a$ and $\fs_\varepsilon(t)$ (resp $a$ and $x$) contains $\pr_{\fs_\varepsilon(t)}(a)$ (resp. $\pr_x(a)$) which ensures, using the second axiom (MAII), that these projections are well-defined. 
In particular, $C^{++}_{\tilde x}$ is well-defined for any $\tilde x \leq_\oplus 0_\A$, as $C^{++}_{\tilde x}=\pr_{\tilde x}(C_0)$. Since it is well-defined, $C^{++}_{\tilde x}$ is the only alcove based at $\tilde x$ which retracts to $C^{++}_x$ by $\rho_{C_0}$.

We also extend the definition of $C^\infty_{\tilde x}$ to any $\tilde x\in \rho_{C_\infty}^{-1}(-\sT)$: recall that, for $x \in -\sT$, $C^\infty_x$ is the unique alcove of $\cT^-_x(\A)$ opposite to $C^{++}_x$. Now if $\tilde x\in \rho_{C_\infty}^{-1}(x)$, there is $g\in I_\infty$ such that $\tilde x=g.x$, and set $C^\infty_{\tilde x}:=g.C^\infty_x$, it does not depend on the choice of $g$ (see \cite[\S 5.1]{bardy2025twin}). This alcove should be thought of as the projection of $C_\infty$ at $x$ (except that $C_\infty$ does not live in $\A$ but in the opposite twin apartment $\A_\ominus$). Note that, for $x\in \rho_{C_\infty}^{-1}(-\sT)\cap \rho_{C_0}^{-1}(-\sT)$, the alcoves $C^\infty_x$ and $C^{++}_x$ need not be opposite in $\cT^\pm_x(\I_\oplus)$, unless $x\in -\sT$.

In the following proposition, we rephrase results of \cite{bardy2025twin}.

\begin{Proposition}\label{p_def_retraction_C_infty}
    \begin{enumerate}
        \item  Let $x\in \I_\oplus$ be $C_\infty$-friendly (i.e such that there exists a twin apartment $(A_\oplus,A_{\ominus})$ such that $A_\oplus\ni x$ and $C_\infty\subset A_\ominus$). Then if $(B_\oplus,B_\ominus)$ is a twin apartment  containing $x$ and $C_\infty$, $B_\oplus$  contains $C_x^\infty$.

        \item 
Let $\varepsilon\in \{\ominus,\oplus\}$. Let $A=(A_\oplus,A_{\ominus})$, $B=(B_\oplus,B_{\ominus})$ be two twin apartments such that $A_\varepsilon\cap B_\varepsilon$ has non-empty interior.  Let $g\in G_{twin}$ be such that $g.A=B$ and such that $g$ fixes a non-empty open subset of $A_\varepsilon\cap B_\varepsilon$. Then $g$ fixes $A\cap B:=(A_\oplus\cap B_\oplus)\sqcup (A_\ominus\sqcup B_\ominus)$. 

 \end{enumerate}
\end{Proposition}

\begin{proof}
1) This is a consequence of \cite[5.1 2)]{bardy2025twin}.

2) Using isomorphism apartments, we may assume $A=\A_{twin}:=(\A,\A_\ominus)$. By \cite[Theorem 4.7]{bardy2025twin}, there exists $h\in G_{twin}$ such that $h.\A_{twin}=B$ and such that  $h$ fixes $\A_{twin}\cap B$. Then $g^{-1}h\A_{twin}=\A_{twin}$ and $g^{-1}h$ induces an affine automorphism $\phi$ of $\A_{\eta}$ for both $\eta\in \{\ominus,\oplus\}$. As $\phi$ fixes a non-empty open subset of $\A_\varepsilon$, it fixes $\A_\varepsilon$. Using \cite[4.1.7]{bardy2025twin}, we deduce $g^{-1}h\in \mathbb{T}(\mathds{k})$ and thus $g^{-1}h$ fixes $\A_{twin}$. Lemma follows.
\end{proof}

Using projections, the masure-theoretic retractions $\rho_{C_0}$, $\rho_{C_\infty}$can be computed at the level of tangent buildings, as we develop now.

\paragraph{Local behaviour of masure retractions}

Let $x\in -\sT$, then $C^\infty_x$ is an alcove in $\cT^-_x(\A)$. In the tangent building $\cT^\pm_x(\I_\oplus)$ we have the building-theoretic retraction centered at $C^\infty_x$: $\rho_{C^\infty_x}: \cT^\pm_x(\I_\oplus)\rightarrow\cT^\pm_x(\A)$. By definition, for any $a\in\cT^-_x(\I_\oplus)$ (resp. $a\in\cT^-_x(\I_\oplus)$), $\rho_{C^\infty_x}(a)$ is the unique alcove in $\cT^\pm_x(\A)$ such that $d^-(C^\infty_x,a)=d^-(C^\infty_x,\rho_{C^\infty_x}(a))$ (resp. $d^\ast(C^\infty_x,a)=d^\ast(C^\infty_x,\rho_{C^\infty_x}(a))$).
Let $\tilde{x}\in \rho_{C_\infty}^{-1}(\{x\})$ and let $g\in I_\infty$ be any element such that $g.\tilde x = x$. Then, $\rho_{C_\infty}$ defines a map $\cT^\pm_{\tilde x}(\I_\oplus)\rightarrow\cT^\pm_x(\A)$ and by \cite[Lemma 5.1.4]{bardy2025twin} we have the following factorization:

\begin{equation}\label{eq: factorizationinfinite}\begin{tikzcd}
 & \cT^\pm_x(\I_\oplus) \arrow[dr,"\rho_{C^\infty_x}"] \\
\cT^\pm_{\tilde x}(\I_\oplus) \arrow[ur,"g"] \arrow[rr,"\rho_{C_\infty}"] && \cT^\pm_x(\A)
\end{tikzcd}\end{equation}

Similarly for $\tilde x \in \rho_{C_0}^{-1}(x)$ and any $g\in I_0$ such that $g.\tilde x = x$, we have the factorization:
\begin{equation}\label{eq: factorizationzero}\begin{tikzcd}
 & \cT^\pm_x(\I_\oplus) \arrow[dr,"\rho_{C^{++}_x}"] \\
\cT^\pm_{\tilde x}(\I_\oplus) \arrow[ur,"g"] \arrow[rr,"\rho_{C_0}"] && \cT^\pm_x(\A)
\end{tikzcd}
\end{equation}

\begin{Proposition}\label{p_Cinfty_chains}
Let $x\in \A$ and $C_\A$ be an alcove of $\A$ based at $x$. Let   $\rho:\cT^{\pm}_x(\I_\oplus)\rightarrow \cT^{\pm}_x(\A)$ be the retraction onto $\cT^{\pm}_x(\A)$, centered at $C_\A$. Let $C,C'$ be two alcoves based at $x$ and of the same sign $\varepsilon$. Let $\tilde{C}$ be the alcove of $\A$ such that $d^{\varepsilon}(\rho(C),\tilde{C})=d^{\varepsilon}(C,C')$ Then there exists a $(C_\A,W^v_x)$-chain from    $\rho(C')$ to $\tilde{C}$. 
\end{Proposition}

\begin{proof}
    We prove it by induction on $n:=\ell(d^{\varepsilon}(C,C'))$. If $n=0$, this is clear. Assume that $n>0$. Let $w=d^{\varepsilon}(C,C')$. Write $w=s_{i_1}\ldots s_{i_n}$, with $i_1,\ldots,i_n\in I$. Let $C^{(1)},\ldots,C^{(n+1)}$ be the gallery from $C$ to $C'$ such that $d^{\varepsilon}(C^{(j)},C^{(j+1)})=s_j$ for all $j\in \llbracket 1,n\rrbracket$. We have $d^{\varepsilon}(C^{(2)},C')=w_2:=s_{i_1}w$ and $\ell(w_2)=n-1$. Let $\tilde{C}_2$ be the alcove of $\A$ such that $d^\varepsilon(\rho(C^{(2)}),\tilde{C}_2)=w_2$. 
    
    Assume the existence of  a $(C_\A,W^v_x)$-chain from $\rho(C')$ to  $\tilde{C}_2$. If there exists an apartment $A$ containing $C$, $C^{(2)}$ and $C_\A$, let $\phi:A\rightarrow \A$ be the apartment isomorphism fixing $A\cap \A$.
    Then $d^{\varepsilon}(\rho(C),\rho(C^{(2)}))=d^{\varepsilon}(\phi(C),\phi(C^{(2)}))=d^{\varepsilon}(C,C^{(2)})=s_{i_1}$. Therefore $d^\varepsilon(\rho(C),\tilde{C})=s_{i_1}w_2=w$, thus $\tilde{C}=\tilde{C}_2$, which proves the existence of a chain, by assumption.
    
    Assume now that  there exists no apartment containing $C, C^{(2)}$ and $C_\A$. Let $A$ be an apartment containing $C^{(2)}$ and $C_\A$, which exists by \cite[Proposition 5.17]{hebert2020new}. Let $M$ be the wall delimiting $C$ and $C^{(2)}$. Let $D_1$ and $D_2$  be the half-apartments of $A$ delimited by $M$. We can assume that $D_1$ contains $C_\A$. 
    
    Then by \cite[Proposition 2.9]{rousseau2011masures}, there exists an apartment $B$ containing $D_1$ and $C$. If $D_1$ contained $C^{(2)}$, then $B$ would contain $C,C^{(2)}$ and $C_\A$. Therefore $D_1$ does not contain $C^{(2)}$. Let $\psi:B\rightarrow A$ be the apartment isomorphism fixing $A\cap B$. In particular since $\psi$ fixes $C_\A\subset A\cap B$, we have $\rho|_B=\rho|_\A\circ \psi$. Then $\psi(C)$ is the alcove of $B$ dominating the closed panel $C\cap C^{(2)}$ and not contained in $D_1$. Therefore $\psi(C)=C^{(2)}$ and thus $ \rho(C)=\rho(C^{(2)})$. Moreover  if $C''$ is the alcove of $\A$ such that $d^{\varepsilon}(\rho(C),C'')=s_{i_1}$ and $M''$ is the wall of $\A$ separating $\rho(C)$ and $C''$, then $C''$ and $C_\A$ are on the same side of $M''$. Let $\gamma[m]\in \Phi^a$ be such that $s_{\gamma[m]}$ is the reflection of $W^v\ltimes Y$ fixing $M''$. Then we have $\tilde{C}=s_{\gamma[m]}(\tilde{C}_2)$. The minimal gallery of type $(s_{i_1},\ldots,s_{i_n})$ (see Definition~\ref{Def galleries}) from $\rho(C)$ to $\tilde{C}$ has $C''$ as a second term and  $M''$ separates $C''$ and $\rho(C)$. Consequently  $C''$ and $\tilde{C}$ are on the same side of $M$ and $\tilde{C}_2$ is on the opposite side.  Let $(\beta_1[n_1],\ldots,\beta_r[n_r])$ be a chain data from $\rho(C')$ to $\tilde{C}_2$. Then $(\beta_1[n_1],\ldots,\beta_r[n_r],\gamma[m])$ is a chain data from $\rho(C')$ to $\tilde{C}$. 
\end{proof}
\subsection{Lifts of $C_\infty$-Hecke open paths in a twin masure}\label{subsection: lifts in masure}
In this section we explain the main results of the paper through Theorem~\ref{Prop : lifts are finite} and Corollary~\ref{Corollary: RPols count doubleorbit}. These results will be detailed and proved in the subsequent sections. 
 
\paragraph{Segments in a masure} 
As it has already been mentioned, the notion of segment in $\A$ naturally extends to the notion of segment in the masure $\I_\oplus$: $\fs$ is a segment in $\I_\oplus$ if and only if it can be written $\fs=g_\fs.\fs_{0}$ for some element $g_\fs\in G$ and some segment $\fs_0$ in $\A$.  Similarly an open segment in $\I_\oplus$ is any triplet $\underline\fs=(\fs,D_\fs,\underline\fs(1))$ of the form $g_{\underline\fs}.\underline\fs_0=(g_{\underline\fs} \fs_0,g_{\underline\fs} D_{\fs_0},g_{\underline\fs} \underline\fs_0(1))$ for some open segment $\underline\fs_0=(\fs_0,D_{\fs_0},\underline\fs_0(1))$ in $\A$.
 
\begin{Definition}For any $\lambda \in Y^+$, we say that a segment is \textbf{a segment of type $\lambda$} if it is a segment in $\I_\oplus$ such that $\rho_{C_0}\circ \fs = \fs_\lambda$. Note that this slightly differs from the definition given in \cite{bardy2025twin}. Similarly, an \textbf{open segment of type $\bx\in W^+_{sph}$} in $\I_\oplus$ is any open segment $\underline\fs$ such that $\rho_{C_0}(\underline\fs)=\underline\fs_\bx$. 
\end{Definition}

\begin{Lemma}\label{l_lift_open_segment} Let $\bx\in W^+_{sph}$, then the following properties hold:
\begin{enumerate}
    \item For any alcove $C$ in $\rho_{C_0}^{-1}(C_\bx)$, there is a unique open segment $\underline\fs$ starting at $0_\A$ and with end alcove $C$. Moreover $\underline\fs$ is an open segment of type $\bx$.
    \item An open segment $\underline\fs$ is of type $\bx$ if and only if there is $g\in I=\operatorname{Fix}_G(C_0)$ such that $\underline\fs=g.\underline\fs_\bx$
    
\end{enumerate}
\end{Lemma}
\begin{proof}

    By definition of $\rho_{C_0}$, there exists an apartment $A$ of $\I_{\oplus}$ containing $C_0$ and $C$. Then $A$ contains $0_\A$ and $C$. By (MAII) there exists $g\in G$ fixing $A\cap \A$ and such that $A=g.\A$. Then $\underline{\fs}:=g\underline\fs_\bx$ is an open segment starting at $0_\A$ and with end alcove $C$. As $C_0\subset A\cap \A$, $g^{-1}$ fixes $C_0$ and therefore $\rho_{C_0}(\underline\fs)=\rho_{C_0}(g^{-1}\underline\fs)=\underline\fs_\bx$, hence $\underline\fs$ is of type $\bx$. 
    
    Let now $\underline{\tilde{\fs}}=(\tilde{\fs},\tilde{D},C)$ be an other segment starting at $0_\A$ and   with end alcove $C$. Let $\tilde{A}$ be an apartment containing $\underline{\tilde{\fs}}$. Then as $\tilde{A}$ contains $0_\A$ and the base point $a$ of $C$, we have $\tilde{\fs}=\fs$ by (MA II). We have $\tilde{D}_1^-=\pr_{\tilde{\fs}_-(1)}(C)=\pr_{\fs_-(1)}(C)=D_1^-$ and then $D_t^\varepsilon=\tilde{D}_t^\varepsilon$ for $(t,\varepsilon)\in [0,1]^{\pm}$, by the compatibility condition. Therefore $\underline{\tilde{\fs}}=\underline{\fs}$, which concludes the proof of point 1.
    
    Moreover any open segment of type $\bx$ necessarily starts at $0_\A$ (because $\rho_{C_0}^{-1}(0_\A)=\{0_\A\}$), and ends at an alcove in $\rho_{C_0}^{-1}(\bx)$. Hence by uniqueness it is, as constructed above, of the form $\underline\fs=g\underline\fs_\bx$ for some $g\in \operatorname{Fix}_G(C_0)$, which proves point 2.
    
\end{proof}

 A segment $\fs$ is said to be $C_\infty$-friendly if, for any $t\in[0,1]$, $\fs(t)\cup C_\infty$ is contained in a twin-apartment. In this case,  $\rho_{C_\infty}$ is well-defined on $\fs$. Muthiah has proved in \cite{muthiah2019double} that the image $\gamma$ of a $C_\infty$-friendly segment by $\rho_{C_\infty}$ is a $C_\infty$-Hecke path (see~\ref{ss_Cinfty_hecke_paths} for the definition). Reciprocally, he proves that if $\gamma$ is a $C_\infty$-Hecke path, then  the number of lifts of $\gamma$ by $\rho_{C_\infty}$ is finite and non-zero. More than that, he proves the existence of a polynomial $R_\gamma$ such that for every finite field $\bk'$, the number of lifts of $\gamma$ in $\I(\G,\bk'(\qp),\omega_\oplus)$ is $R_\gamma(|\bk'|)$. Muthiah's proof was detailed   in \cite{bardy2025twin}.   We aim to produce the same results for open segments and $C_\infty$-Hecke open paths. We say that an open segment is $C_\infty$-friendly if its non-open component is. Note that, by \cite[Proposition 4.2]{bardy2025twin}, the retraction $\rho_{C_\infty}$ is then well defined on its decoration and its end alcove aswell. The next sections are dedicated to the proof of the following result.
\begin{Theorem}\label{Prop : lifts are finite}
The image of a $C_\infty$-friendly open segment of type $\bx \in W^+_{sph}$ is a $C_\infty$-Hecke open path of type $\bx$. Moreover there exists a polynomial $R_{\bx,\underline{\fc}}(X)\in \mathbb Z[X]$, which only depends on $(\A,\bx,\underline{\fc})$, such that the number of open segments of type $\bx$ lifting $\underline{\fc}$ is given by $R_{\bx,\underline{\fc}}(|\mathds k|)$. 
\end{Theorem}

This result is reformulated and proved in the next Section: In Proposition~\ref{Proposition: retraction is C infty} we prove that a $C_\infty$-friendly open segment of type $\bx$ retracts to a $C_\infty$-Hecke open path of the same type, and in Theorem~\ref{Theorem : number of lifts spherical} we give an explicit formula for the polynomials $R_{\underline{\fc},\by}(X)\in \mathbb Z[X]$.

Recall that the principal motivation is to produce an analogous of Formula~\eqref{eq: RPolreductive_intro} for general Kac-Moody groups. In the language of masures, we want to count $\rho_{C_0}^{-1}(C_\bx)\cap \rho_{C_\infty}^{-1}(C_\by)$. Conjecturally, any element of this set is the ending alcove of a $C_\infty$-friendly segment of type $\bx$, and from Theorem~\ref{Prop : lifts are finite} we deduce the following.

\begin{Corollary}\label{Corollary: RPols count doubleorbit}
Suppose that Conjecture \cite[4.4.1]{bardy2025twin} holds. Then, for $\bx,\by\in W^+$ such that $\bx$ is spherical, the cardinality $|(I_\infty \by I_0\cap I_0 \bx I_0 )/I_0|$ is finite and given by the evaluation of a polynomial $R_{\bx,\by}(X)\in \mathbb Z[X]$, which only depends on the combinatorics of $W^+$, at $|\mathds k|$.
\end{Corollary}

\begin{proof}
    Since $I_0$ is the stabiliser of $C_0$ and $I_\infty$ the stabiliser of $C_\infty$, alcoves in $\rho_{C_0}^{-1}(C_{\bx})\cap \rho_{C_\infty}^{-1}(C_{\by})$ are in bijection with $(I_\infty \by I_0\cap I_0 \bx I_0 )/I_0$ (note that $C_{\bx}=\bx.C_0$ and $C_{\by}=\by.C_0$). Moreover, for any alcove $C\in\rho_{C_0}^{-1}(C_{\bx})\cap \rho_{C_\infty}^{-1}(C_{\by})$, the open segment starting at $0_\A$ and with end alcove $C$ is sent to $\underline \fs_{\bx}$ by $\rho_{C_0}$. By Conjecture \cite[4.4.1]{bardy2025twin}, it is $C_\infty$-friendly and its image by $\rho_{C_\infty}$ is a $C_\infty$-Hecke open path with end alcove $C_{\by}$. Therefore $|(I_\infty \by I_0\cap I_0 \bx I_0 )/I_0|$ is equal to the number of $\rho_{C_\infty}$-lifts of $C_\infty$-Hecke open paths of type $\bx$ with ending alcove $C_{\by}$. Hence, by Theorem~\ref{Prop : lifts are finite} and Corollary~\ref{Corollary: finiteness open hecke paths}, it is given by the evaluation of the polynomial $\sum_{\underline{\fc}}R_{\bx,\underline{\fc}}$ at $|\mathds k|$, where the summand runs over all $C_\infty$-Hecke open paths of type $\bx$ with end alcove $C_{\by}$.
\end{proof}

\subsection{Retractions of open segments}\label{ss_retraction_open_segments} 
In this section, we prove the first part of Theorem~\ref{Prop : lifts are finite}: the image of a $C_\infty$-friendly segment of type $\bx$ by the retraction $\rho_{C_\infty}$ is a $C_\infty$-Hecke open path (see Proposition~\ref{Proposition: retraction is C infty}).

\begin{Lemma}\label{l_alcoves_retraction_opposite}
Let $C, C_-, C_+$ be three alcoves of $\I_\oplus$ based at the same point $x$, with $C_+$ positive and $C_-$ negative. Let $A_\oplus$ be an apartment containing $C$. Let $\rho:\cT_x^\pm(\I_{\oplus})\twoheadrightarrow \cT_x^{\pm}(A_\oplus)$ be the retraction centred at $C$.  Assume that $\rho(C_-)$ and $\rho(C_+)$ are opposite. Then $C_-$ and $C_+$ are opposite.
\end{Lemma}

\begin{proof}
    Let $B_\oplus$ be an apartment containing $C_+$ and $C$, and $\phi:B_\oplus\rightarrow A_\oplus$ be the apartment isomorphism fixing $A_\oplus\cap B_\oplus$. By symmetry, we can assume that $C$ is positive. Then   we have $d^+(C,C_+)=d^+(\phi(C),\phi(C_+))=d^+(C,\rho(C_+))$ and similarly,   $d^*(C,C_-)=d^*(C,\rho(C_-))$. Since $\rho(C_-)$ and $\rho(C_+)$ are opposite and since $\rho(C_-), \rho(C_+)$ and $C$ are contained in a common apartment, we have $d^+(C,\rho(C_+))=d^*(C,\rho(C_-))$. Therefore $d^+(C,C_+)=d^*(C,C_-)$ and thus $C_-$ and $C_+$ are opposite by \cite[Corollary 5.141]{abramenko2008buldings}.
\end{proof}

\begin{Lemma}\label{l_segment_germ_retraction}
Let $A_\oplus$ be an apartment of $\I_\oplus$ and $\tau:(0^-,0^+)\rightarrow A_\oplus$ be a  segment germ. We assume that $\tau$  is increasing for $\leq_{\oplus}$.
Let $C$ be an alcove based at $\tau(0)$, $C_+$ be an alcove containing $\tau_+(0)$ and $A^C_\oplus$ be an apartment of $\I_\oplus$ containing $C$. Let $\rho:\cT_{\tau(0)}^{\pm}(\I_\oplus)\twoheadrightarrow \cT_{\tau(0)}^{\pm}(A_\oplus^C)$ be the retraction centred at $C$.  Assume that $\rho\circ \tau$ is a segment germ. Then there exists an apartment containing $\tau((0^-,0^+))$, $C$ and $C_+$.
\end{Lemma}

\begin{proof}
The alcove $\rho(C_+)$ contains $\rho\circ\tau_+(0)$. Let $C'_-$ be the alcove of $A_\oplus^C$ opposite to $\rho(C_+)$. Let $B_\oplus$ be an apartment containing $\tau_-(0)$ and $C$. Let $\phi:B_\oplus\rightarrow A_\oplus^C$ be the apartment isomorphism fixing $A_\oplus^C\cap B_\oplus$. Then $\phi$ is the restriction of $\rho$ to $B_\oplus$. Let $C_-=\phi^{-1}(C'_-)$. Then $\rho(C_-)$ and $\rho(C_+)$ are opposite and thus by Lemma~\ref{l_alcoves_retraction_opposite}, $C_-$ and $C_+$ are opposite. Moreover, $\rho(C_-)=\phi(C_-)=C'_-\supset \rho(\tau_-(0))=\phi(\tau_-(0))$ and thus $C_-\supset \tau_-(0)$. 

By symmetry, we can assume that $C$ is positive.  We have $d^*(C,C_-)=d^*(C,\phi(C_-))=d^*(C,C'_-)$ and similarly,  $d^+(C,C_+)=d^+(C,\rho(C_+))$. As $C,\rho(C_-)$ and $\rho(C_+)$ are contained in a common apartment and as $\rho(C_-)$ and $\rho(C_+)$ are opposite, we have $d^*(C,C_-)=d^*(C,\rho(C_-))=d^+(C,\rho(C_+))=d^+(C,C_+)$. By \cite[Proposition 5.179]{abramenko2008buldings}, $C$ belongs to the unique apartment of the tangent twin building $\cT^{\pm}_{\tau(0)}(\I_\oplus)$ containing $C_-$ and $C_+$ and this apartment also contains $\tau$, since $C_-\supset \tau_-(0)$ and $C_+\supset \tau_+(0)$. 
\end{proof}

\begin{Lemma}\label{l_compactness_argument}
 Let $\bx\in W^+_{\sph}$ and   $\underline{\fs}$ be an open segment of type $\bx$. We assume that $\underline{\fs}$ is $C_\infty$-friendly.  Write $\underline{\fs}=(\fs,D,C)$.   Then there exist $n\in \Z_{\geq 0}$ and $t_1,\ldots,t_n\in (0,1)$ such that $t_0:=0<t_1<\ldots< t_n<t_{n+1}:= 1$ and such that  for every $i\in \llbracket 0,n\rrbracket$, there exists an apartment $A_i$ containing $C_\infty$, $\fs([t_i,t_{i+1}])$ and $D_t^\varepsilon$, for $(t,\varepsilon)\in [t_i,t_{i+1}]^{\pm}$. In particular, if $i\in \llbracket 0,n\rrbracket$,  then $\rho_{C_\infty}\circ \fs|_{[t_i,t_{i+1}]}$ is a segment in $\A$ and $\rho_{C_\infty}(\underline{\fs})=(\rho_{C_\infty}(\fs),\rho_{C_\infty}(D),\rho_{C_\infty}(C))$ is an open path. \end{Lemma}

\begin{proof}

Let $(t,\varepsilon)\in [0,1]^{\pm}$.  By \cite[Proposition 4.2]{bardy2025twin}, there is a twin apartment $B^\varepsilon_t$  containing $D^\varepsilon_{t}\cup C_\infty$ and thus $\rho_{C_\infty}(D^\varepsilon_{t})$ is well-defined. Since $D^+_{t}\cup D^-_{t}\subset A^+_t\cup A^-_t$ there is an open interval $U_t$ of $[0,1]$,  containing $t$ and  such that $D^\varepsilon_{t^*}\subset A^+_t\cup A^-_t$ for all $t^*\in U_t$. By compactness of $[0,1]$, there exist $k\in \N$ and $u_1,\ldots,u_k\in [0,1]$ such that $[0,1]=\bigcup_{i=1}^k U_{u_i}$.  

Therefore, up to reducing $k$ and to renumbering  there exist $s_1,\ldots,s_k\in [0,1]$ such that $0=s_0<s_1<\ldots < s_k=1$ and such that for every $i\in \llbracket 0,k-1\rrbracket$, for every $t\in [s_i,s_{i+1})$, $D_t^+\subset B_{u_i}^+\cup B_{u_i}^-$. Let $\varepsilon\in \{-,+\}$ be such that $B_{u_i}^\varepsilon\supset D_{s_i}^+$. If for all $t\in (s_i,s_{i+1})$, we have $D_t^+\subset B^\varepsilon_{u_i}$, one sets $m_i=s_i$. Otherwise, one sets $m_i=\inf\{t'\in (s_i,s_{i+1})\mid D_{t'}^+\not\subset B_{u_i}^\varepsilon\}$. If $D_{m_i}^+\subset B_{u_i}^\varepsilon$, then $D_t^+\subset B_{u_i}^\varepsilon$ for $t>m_i$ close enough to $m_i$: a contradiction. 
Therefore $D_{m_i}^+\not\subset B_{u_i}^\varepsilon$. Let $t\in (m_i,s_{i+1})$ and assume that $D_t^+\subset B_{u_i}^\varepsilon$. Then $\fs(t)\in B_{u_i}^\varepsilon$ and thus $\fs([s_i,t])\subset B_{u_i}^\varepsilon$ by (MA II). Therefore 
$D_{m_i}^+=\pr_{\fs_+(m_i)}(D_t^+)\subset B_{u_i}^\varepsilon$: a contradiction. 
Consequently $D_{t}^+\not\subset B_{u_i}^\varepsilon$ and thus $D_t^+\subset B_{u_i}^{-\varepsilon}$ for all $t\in [m_i,s_{i+1})$ . By induction we deduce the existence of $\{t_0,t_1,\ldots,t_n,t_{n+1}\}$ among  $\{s_j,m_j\mid j\in \llbracket 0,k-1\rrbracket\}$ and of apartments $A_0,\ldots,A_n$ among the $\{B_{u_j}^+,B_{u_j}^-\mid  j\in \llbracket 0,k-1\rrbracket\}$ such that $0=t_0<\ldots <t_{n+1}=1$ and such that $A_i$ contains $C_\infty$, $\fs([t_i,t_{i+1}])$ and $D_{t_i}^+$ for $i\in \llbracket 0,n\rrbracket$. But then by \eqref{e_decoration}, $A_i$ contains $D_t^\varepsilon$ for all $(t,\varepsilon)\in [t_i,t_{i+1}]^{\pm}$. 

Since   $\rho_{C_\infty}$ restricts to an isomorphism of twin apartments $A_i\rightarrow \mathbb A \sqcup \A_\ominus$, we deduce that $\rho_{C_\infty}\circ\fs|_{[t_i,t_{i+1}]}$ is a segment for all $i\in \llbracket 1 ,n\rrbracket$ and that $\rho_{C\infty}(\underline{\fs})$ is  an open path.
\end{proof}

\begin{Lemma}\label{l_local_folding_retraction_segment}
  Let $\bx\in W^+_{\sph}$ and $\underline{\fs}=(\fs,D,C)$, be an open segment of type $\bx$.  Assume that $\underline{\fs}$ is $C_\infty$-friendly  and set $\fc=\rho_{C_\infty}\circ \fs$.    Let   $t_*\in (0,1)$. We assume that $\fc'_-(t_*)\neq \fc'_+(t_*)$. Then there exists a $(C^\infty_{\fc(t_*)},W^v_{\fc(t_*)})$-chain from $\rho_{C_\infty}(D_{t_*}^+)$ to $\lim_{t\rightarrow t_*^-}\rho_{C_\infty}(D_t^+)$.
\end{Lemma}

\begin{proof}
Let $n\in \Z_{\geq 0}$, $t_1,\ldots, t_n$ and $A_1,\ldots,A_n$ be as in Lemma~\ref{l_compactness_argument}.  Let $i\in \llbracket 1,n+1\rrbracket$ be such that $t_*\in (t_i,t_{i+1}]$ and $A=A_i$. As $\fc([t_i,t_{i+1}])$ is a segment, we have $t_{i+1}=t_*$.

    Let $C^-$ be the alcove of $A_{\oplus}$ based at $\fs(t_*)$ and whose direction is opposite to the direction of $D_t^+$, for any $t\in [t_i,t_*)$. Let $\tilde{A}$ be a twin apartment such that $\tilde{A}_\oplus$ contains $C^-$ and $D_{t_*}^+$, which exists by \cite[Proposition 5.17]{hebert2020new}. Since $\overline{C^-}$ contains $\fs\left((t_*^-,t_*]\right)$, $\tilde{A}_{\oplus}$
    contains $D_{t}^+$ for $t\in (t_{i}^-,t_*^-)$, by \eqref{e_decoration}. Let $\tilde{t}\in (t_{i},t_*)$ be such that $D_{\tilde{t}}^+\subset \tilde{A}_{\oplus}$. Identify $\tilde{A}_\oplus$ and $\A$.  Let $C^v$ be the direction of $D_{t_*}^+$. Then we have $D_{t_*}^+=germ_{\fs(t_*)}(\fs(t_*)+C^v)$, $D_{\tilde{t}}^+=germ_{\fs(\tilde{t})}(\fs(\tilde{t})+C^v)$. By assumption, the directions of $D_{\tilde{t}}^+$ and $C^-$ are opposite and thus $C^-=germ_{\fs(t_*)}(\fs(t_*)-C^v)$: $C^-$ and $D_{t_*}^+$ are opposite.
    
    Let $\mu\in Y\cap C^v$. Let $\tau:(0^-,0^+)\rightarrow \tilde{A}_{\oplus}$ be defined by $\tau(t)=\fs(t_*)+t\mu$ for $t\in \R$ close enough to $0$. Then $\tau$ is a twinnable segment germ and by \cite[5.10 Proposition]{bardy2025twin}, $\gamma:=\rho_{C_\infty}\circ \tau$ is a $C_\infty$-Hecke path-germ. More precisely, let $p=\rho_{C_\infty}(\fs(t_*))$. 
    Then there exist $m\in \Z_{\geq 0}$, $\xi_0=\gamma'_+(0),\xi_1,\ldots,\xi_m=\gamma'_-(0)$, $\beta_1,\ldots,\beta_m\in \Phi$ such that for all $j\in \llbracket 1,m\rrbracket$, $s_{\beta_j}(\xi_{j-1})=\xi_j$, $\beta_j(\xi_{j-1})<0$, $\beta_j(p)\in \Z$ and $\beta_j(C^\infty_p)>\beta_j(p)$. For $j\in \llbracket 0,m\rrbracket$, denote by $C^{(j)}$ the alcove of $\A$ containing $p+(0,0^+)\xi_j$ and set $k_j=-\langle p,\beta_j\rangle$.  Then we have $C^{(j)}=s_{\beta_{j-1}[k_{j-1}]}(C^{(j-1)})$ for $j\in \llbracket 1,m\rrbracket$ and $C^{(0)}=\rho_{C_\infty}(D_{t_{i+1}}^+)$.
    Denote by $C_j^v$  the direction of the sector generated by $C^{(j)}$.  As the direction of $\overline{C^v_j}$ contains $\xi_j$, we have $\beta_j[k_j](C^{(j-1)})<0$. Therefore $C^{(0)}$, $C^{(1)}$, $\ldots$, $C^{(m)}$ is a $(C^\infty_p,W^v_p)$-chain from $C^{(0)}=\rho_{C_\infty}(D_{t_*}^+)$ to $C^{(m)}$.

Let us prove that $C^{(m)}=\lim_{t\rightarrow t_*^-} \rho_{C_\infty}(D_t^+)$. We now identify $A=A_i$ and $\A$, using an element $i_\infty$ of $I_\infty$ such that $i_\infty.A_{\oplus}=\A$. Let $\mu=\tau'_-(0)\in Y$. Let $\tilde{t}\in (t_i,t_*)$. We have $D_{\tilde{t}}^+\subset A_{\oplus}$. Let $\tau^{\tilde{t}}$ denote the segment germ $\tau^{\tilde{t}}=\fs(\tilde{t})+[0,0^+)\mu$. Then $(\tau^{\tilde{t}})'_+(0)=\mu=\tau'_-(0)$. We have $\rho_{C_\infty}(a)=i_\infty.a$ for every $a\in A_{\oplus}$. As $A_{\oplus}$ contains $\tau^{\tilde{t}}([0,0^+))$ and $\tau((0^-,0])$ we deduce that if $\gamma^{\tilde{t}}=\rho_{C_\infty}\circ \tau^{\tilde{t}}$, then we have $(\gamma^{\tilde{t}})'_+(0)=\gamma'_-(0)$. As the alcove containing $\rho_{C_\infty}\left(\tau^{\tilde{t}}\left(\left(0,0^+\right)\right)\right)=\fc(\tilde{t})+(0,0^+)(\gamma^{\tilde{t}})'_+(0)$ is $\rho_{C_\infty}(D_{\tilde{t}}^+)$ and as this is true for every $\tilde{t}\in [t_i,t_*)$, we deduce that $C^{(m)}=\lim_{t\rightarrow t_*^-} \rho_{C_\infty}(D_t^+)$, which concludes the proof of the lemma.
\end{proof}

\begin{Proposition}\label{Proposition: retraction is C infty}
    Let $\bx \in W^+_{sph}$ and let $\underline\fs$ be a $C_\infty$-Hecke friendly open segment of type $\bx$. Then $\rho_{C_\infty}(\underline\fs)$ is a $C_\infty$-Hecke open path of type $\bx$. In particular, it is admissible.
\end{Proposition}

\begin{proof}
Let $n\in \Z_{\geq 0}$, $t_0,t_1,\ldots, t_n,t_{n+1}$ and $A_1,\ldots,A_n$ be as in Lemma~\ref{l_compactness_argument}.  

Write $\underline{\fs}=(\fs,D,\underline{\fs}(1))$ and set $\underline{\fc}=(\fc,D_{\fc},\underline{\fc}(1))=\rho_{C_\infty}(\underline{\fs}):=(\rho_{C_\infty}(\fs),\rho_{C_\infty}(D),\rho_{C_\infty}(\underline{\fc}(1)))$. 

For each $A_i$ there is $g_i\in G_{twin}$ sending $A_i$ to $(\A,\A_\ominus)$ and fixing the intersection; in particular $g_i\in I_\infty$. By the admissibility condition for $D$, we have that $A_i$ contains $D^\varepsilon_{t}=\pr_{\fs_\varepsilon(t)}(D^+_{t_i})$ and $\rho_{C_\infty}(D^\varepsilon_{t})=g_iD^\varepsilon_{t}$  for all $t\in [t_i,t_{i+1}]^{\pm}$.  Therefore $D_{\fc}$ satisfies  the admissibility condition \eqref{e_admissibility}  on $[t_i,t_{i+1}]$. By Lemma~\ref{Lemma: admissibility condition} it remains to check that for all $i\in \llbracket 1,n\rrbracket$, $$\fc'_+(t_i)\neq \fc'_-(t_i) \text{ or } d^\ast(D^-_{\fc,t_i},D^+_{\fc,t_i})=w_{\lambda,0}.$$

Let $i\in \llbracket 1,n\rrbracket$. Assume that $\fc_+'(t_i)=\fc_-'(t_i)$, then $\fc((t_i^-,t_i^+))$ is a segment germ based at $\fc(t_i)$. Note that, since $\fs$ is a segment and  since $g_{i-1}(\fs(t_i))=\fc(t_i)$, the filter $g_{i-1} \fs((t_i^-,t_i^+))$ is also a segment germ based at $\fc(t_i)$.
Since $\rho_{C_\infty}((\fs,D_\fs))=(\fc,D_\fc)$, the local factorization of $\rho_{C_\infty}$ implies: \begin{equation}
         \rho_{C^\infty_{\fc(t_i)}}(g_{i-1}\fs((t_i^-,t_i^+)))=\fc((t_i^-,t_i^+)), \; D^\pm_{\fc,t_i}=\rho_{C^\infty_{\fc(t_i)}}(g_{i-1}D^\pm_{t_i}).
     \end{equation}

     By Lemma~\ref{l_segment_germ_retraction} since $g_{i-1}D^+_{t_i}$ contains $g_{i-1}\fs_+(t_i)$, there is an apartment $A_\oplus$ containing  the segment germ $g_{i-1}\fs((t_i^-,t_i^+))$ and the alcoves $C^\infty_{\fc(t_i)}, g_{i-1}D^+_{t_i}$. Since $g_{i-1}D^-_{t_i}=\pr_{g_{i-1}\fs_-(t_i)}(g_{i-1}D^+_{t_i})$, this apartment also contains $g_{i-1}D^-_{t_i}$. Since $C^\infty_{\fc(t_i)}\subset A_\oplus \cap \A_\oplus$, the retraction $\rho_{C^\infty_{\fc(t_i)}}$ restricts to an isomorphism $\cT^{\pm}_{\fc(t_i)}(A_\oplus)\rightarrow \cT^{\pm}_{\fc(t_i)}(\A)$. Therefore, we have:
\begin{equation}\label{eq: dist_equality_admissibility}
    d^\ast(D^-_{t_i},D^+_{t_i})=d^\ast(g_{i-1}D^-_{t_i},g_{i-1}D^+_{t_i})=d^\ast(\rho_{C^\infty_{\fc(t_i)}}(g_{i-1}D^-_{t_i}),\rho_{C^\infty_{\fc(t_i)}}(g_{i-1}D^+_{t_i}))=d^\ast(D^-_{\fc,t_i},D^+_{\fc,t_i}).
\end{equation} Since $\underline\fs$ is an open segment and $\fs$ is a segment of type $-\lambda$, we have $d^\ast(D^-_t,D^+_t)=w_{\lambda,0}$ for all $t$, and in particular for $t=t_i$. We deduce from Equation~\eqref{eq: dist_equality_admissibility} that $d^\ast(D^-_{\fc,t_i},D^+_{\fc,t_i})=w_{\lambda,0}$. Therefore by Lemma~\ref{Lemma: admissibility condition}, $\underline{\fc}$ is admissible.

 Let $E=\{t\in (0,1)\mid \fc'_+(t)\neq \fc'_-(t)\}$ and $n'=|E|$ (we have $n'\leq n$). Write $E=\{t'_1,\ldots,t'_{n'}\}$, with $t_1'<\ldots<t_{n'}'$. Set $t_0'=0$ and $t_{n'+1}'=1$. We proved that for all $i\in \llbracket 1,n'+1\rrbracket$, the restriction of  $(\fc,D)$ to $(t_{i-1}',t_i')$ is a decorated path satisfying  \eqref{e_admissibility}.

Let now $i\in \llbracket 1,n'\rrbracket$. Then by Lemma~\ref{l_local_folding_retraction_segment}, there exists a $(C^\infty_{\fc(t'_i)},W^v_{\fc(t_i')})$-chain from $D_{\fc,t_i'}^+$ to $\lim_{t\rightarrow t_i'^-} D_{\fc,t}^+$. It remains to prove that $\underline{\fc}$ satisfies the second condition of Lemma~\ref{l_characterization_decorated_Hecke_paths}. By assumption, $D_1^-=\pr_{\fs_-(1)}(\underline{\fs}(1))$. Let $\underline{\tilde{\fc}}(1)$ be the alcove of $\cT^+_{\fc(1)}(\A)$ such that \[d^+(\rho_{C_\infty}(D_1^-),\underline{\tilde{\fc}}(1))=d^+(D_{\fc,1}^-,\underline{\tilde{\fc}}(1))=d^+(D^-_{\bx,1},C_\bx).\] Then as $d^+(D^-_{\bx,1},C_\bx)=d^+(D_1^-,\underline{\fs}(1))$, there exists a $(C^\infty_{\fc(1)},W^v_{\fc(1)})$-chain from $\rho_{C_\infty}(\underline{\fs}(1))=\underline{\fc}(1)$ to $\underline{\tilde{\fc}}(1)$, according to Proposition~\ref{p_Cinfty_chains}, applied with $C_\A=C^\infty_{\fc(1)}$ (using the local factorization of $\rho_{C_\infty})$. Therefore by Lemma~\ref{l_characterization_decorated_Hecke_paths}, $\underline{\fc}$ is a $C_\infty$-Hecke open path of type $\bx$. 

\end{proof}

\subsection{Counting  the lifts of a $C_\infty$-Hecke open path}\label{ss_counting_lifts}
We have proved that any open segment of type $\bx\in W^+_{sph}$ retracts to a $C_\infty$-Hecke open path of type $\bx$. We now aim to compute, for a fixed open path $\underline{\fc}$, the number of open segments of type $\bx$ retracting to $\underline{\fc}$. In order to do this, we give in Subsection~\ref{subsubsection: local description Cinfty hecke} a description of this set of open segments in term of local tangent buildings (see Proposition~\ref{Prop: local description bijection}). In Subsection~\ref{subsubsection: local computation} we introduce building-theoretic methods in order to obtain numerical results from this local description. This leads to Theorem~\ref{Theorem : number of lifts spherical}, in which we explicit the polynomials mentioned in Theorem~\ref{Prop : lifts are finite}.

 \subsubsection{Local description of the lifts of a $C_\infty$-Hecke open path}\label{subsubsection: local description Cinfty hecke}
For any $\bx\in W^+_{sph}$ and any $C_\infty$-Hecke open path $\underline{\fc}$ of type $\bx$, we denote $\rho_{C_0}^{-1}(\underline\fs_\bx)\cap \rho_{C_\infty}^{-1}(\underline{\fc})$ the set of open segments of type $\bx$ which retract to $\underline{\fc}$. In this section, we give an explicit local description of this set.
\begin{Definition}\label{Notations: localsets} Let $\underline{\fc}=(\fc,D,\underline{\fc}(1))$ be an admissible $C_\infty$-Hecke open path of type $\bx=\qp^\lambda w\in W^+_{sph}$. Recall that $w_{\lambda,0}$ denotes the maximal element of the spherical parabolic subgroup $W_{\lambda^{++}}$, and let $w^\bx$ be such that $D^+_{\bx,0}=a(0,-w^\bx)$. We introduce the following notation:
\begin{itemize}
    \item \index{T@$T(\underline{\fc})$} The ordered set $T(\underline{\fc})$ of times $t\in[0,1]$ such that, either $t\in\{0,1\}$, either there is an affine wall separating $C^\infty_{\fc(t)}$ and $\fc_+(t)$. By \cite[Lemma 5.4.2)]{bardy2025twin}, this is a finite set.
    \item The set of folding times: $T^f(\fc):=\{t\in [0,1)\mid \fc'_+(t)\neq \fc'_-(t)\}$ (by convention, $\fc'_-(0)=-\lambda^{++}$ and $D^-_0=C_0$). This set is contained in $T(\underline{\fc})$ by \cite[Lemma 5.4.4)]{bardy2025twin}.
    \item For $t\in [0,1]$, $w^\infty_{\underline{\fc},t}:=\begin{cases}
       d^-(C^\infty_{\fc(t)},D^+_t)&\ \text{if} \ t<1 \\
       d^\ast(C^\infty_{\fc(1)},\underline{\fc}(1)) &\ \text{if} \ t=1
    \end{cases}$\index{w@$w^\infty_{\underline{\fc},t}$,$w^-_{\bx,t}$} and $w^-_{\bx,t}:=\begin{cases*}
        d^\ast(C_0,D^+_{\bx,0}) &\  \text{if} \ t=0 \\
        d^+(D^-_{\bx,1},C_\bx) &\  \text{if} \ t=1 \\
        w_{\lambda,0} &\ \text{otherwise}
    \end{cases*}$. 
    \item For $t\in [0,1)$: $$\cC(\underline{\fc},\bx,t):=\{C\in \cT^-_{\fc(t)}(\I_\oplus)\mid d^-(C^\infty_{\fc(t)},C)=w^\infty_{\underline{\fc},t}, \; d^\ast(D^-_t,C)=w^-_{\bx,t}\},$$
    and $$\cC(\underline{\fc},\bx,1):=\{C\in \cT^+_{\fc(1)}(\I_\oplus)\mid d^+(D^-_1,C)=w^-_{\bx,1}, \; d^\ast(C^\infty_{\fc(1)},C)=w^\infty_{\underline{\fc},1}\}.$$
\end{itemize}

\end{Definition}

\begin{Lemma}\label{Lemma: local element lift}
    Let $\underline{\fc}$ be an admissible $C_\infty$-Hecke open path of type $\bx\in W^+_{sph}$. For each $t\in T(\underline{\fc})\setminus \{1\}$ and every $a_t\in \cC(\underline{\fc},\bx,t)$, there exists $i_{a_t}\in I_\infty$ such that $$i_{a_t}.C^\infty_{\fc(t)}=C^\infty_{\fc(t)} \ \text{and} \ i_{a_t}.D^+_t=a_t.$$
    Similarly, for every $a_1\in \cC(\underline{\fc},\bx,1)$, there exists $i_{a_1}\in I_\infty$ such that $i_{a_1}.C^\infty_{\fc(1)}=C^\infty_{\fc(1)} \ \text{and} \ i_{a_1}.\underline{\fc}(1)=a_1.$

\end{Lemma}
\begin{proof}
  Let $t\in T(\underline{\fc})$. By \cite[Proposition 4.2]{bardy2025twin}, for any $a_t\in \cT^-_{\fc(t)}(\I_\oplus)$ there is a pair of twin apartments $(A_\oplus,A_\ominus)$ such that $C_\infty\subset A_\ominus$ and $a_t\subset A_\oplus$. By Proposition~\ref{p_def_retraction_C_infty},  $C^\infty_{\fc(t)}\subset A_\oplus\cap \A$ and there is an element $g\in G_{twin}$ fixing $(A_\oplus \cap \A) \sqcup (A_\ominus \cap \A_\ominus)$ such that $g.(\A,\A_\ominus)=(A_\oplus,A_\ominus)$. In particular, such $g$ fixes $C_\infty$, so belongs to $I_\infty$, and also fixes $C^\infty_{\fc(t)}$. 
  Moreover, if $t<1$ (resp. $t=1$) and $a_t \in \cC(\underline{\fc},\bx,t)$  then $d^-(C^\infty_{\fc(t)},a_t)=d^-(C^\infty_{\fc(t)},D^+_t)$ (resp. $d^\ast(C^\infty_{\fc(1)},a_1)=d^\ast(C^\infty_{\fc(1)},\underline{\fc}(1))$) so, since $D^+_t\subset\A$ (resp. $\underline{\fc}(1) \subset \A$) and $a_t \subset A_\oplus$, we have $g.D^+_t=a_t$ (resp. $g.\underline{\fc}(1)=a_1$), and we can set $i_{a_t}=g$. 
\end{proof}

\begin{Lemma}\label{Lemma: global lift}
   Let $\underline{\fc}$ be an admissible $C_\infty$-Hecke path of type $\bx \in W^+_{sph}$, and, for each time $t\in T(\underline{\fc})$ and alcove $a_t\in \cC(\underline{\fc},\bx,t)$, fix $i_{a_t}$ as in Lemma~\ref{Lemma: local element lift}. Set by convention that, for $t\notin T(\underline{\fc})$, $i_{a_t}=1_G$. 

    For any $a\in \prod\limits_{t\in T(\underline{\fc})} \cC(\underline{\fc},\bx,t)$, let us define $g_a.\underline{\fc}:=(g_a.\fc,g_a.D,g_a.\underline{\fc}(1))$:
    \begin{enumerate}

        \item the path $g_a.\fc$ is defined by $t\mapsto g_a(t).\fc(t)$, with $g_a(t)=\prod\limits_{t'\in T(\underline{\fc})\cap [0,t)}i_{a_{t'}}$,
        \item the decoration $g_a.D$ is given by $(g_a.D)^-_t=g_a(t)D^-_t$ and $(g_a.D)^+_t=g_a(t)i_{a_t}D^+_t$,
        \item the end alcove is given by $g_a.\underline{\fc}(1)=g_a(1)i_{a_1}\underline{\fc}(1)$.
    \end{enumerate}

    Then, $g_a.\underline{\fc}$ is an open segment such that $\rho_{C_0}(g_a.\underline{\fc})=\underline\fs_\bx$ and $\rho_{C_\infty}(g_a.\underline{\fc})=\underline{\fc}$.
\end{Lemma}

\begin{proof}

 By construction, since $g_a(t)\in I_\infty$ for all $t\in [0,1]$, it is clear that $\rho_{C_\infty}(g_a.\underline{\fc})=\underline{\fc}$. 
    
Let us first prove that $g_a.\fc$ is a segment on $[0,1]$, and that the decoration $g_a.\underline{\fc}$ satisfies the admissibility condition. Let $t_1,t_2$ be two consecutive times in $T(\underline{\fc})$. Then $g_a$ is constant on $(t_1,t_2]$, $g_a(t_1)=g_a(t_2)i_{a_{t_1}}^{-1}$ and $i_{a_{t_1}}$ fixes $\fc(t_1)$, since it fixes  $C^\infty_{\fc{t_1}}$. Moreover $\fc|_{[t_1,t_2]}$ is a segment and thus $(g_a.\fc)|_{[t_1,t_2]}=g_a(t_2).\fc|_{[t_1,t_2]}$  is a segment. By the admissibility condition for $\underline{\fc}$, $g_a.\underline{\fc}$ satisfies the admissibility condition on $[t_1,t_2]$.
    
    Therefore in order to prove admissibility of $g_a.\underline{\fc}$ on $[0,1]$, by Lemma~\ref{Lemma: admissibility condition} it remains to check that condition \eqref{eq: admissibility condition} holds at every time $t\in T(\underline{\fc})\setminus\{0,1\}$. However, by construction, for any such time $t$, $(g_a.D)^-_t=g_a(t).D^-_t$ and $(g_a.D)^+_t=g_a(t).i_{a_t}D^+_t$, hence:
    $$ d^\ast((g_a.D)^-_t,(g_a.D)^+_t)=d^\ast(D^-_t,i_{a_t}D^+_t).$$ Moreover, by definition of $i_{a_t}$, $i_{a_t}D^+_t=a_t\in \cC(\underline{\fc},\bx,t)$, and in particular $d^\ast(D^-_t,i_{a_t}D^+_t)=w^-_{\bx,t}=w_{\lambda,0}$ (because $t\notin \{0,1\}$). Hence condition \eqref{eq: admissibility condition} is verified, and $g_a.\underline{\fc}$ is an admissible open path. This also proves that the segment germs $(g_a.\fc)_-(t)$ and $(g_a.\fc)_+(t)$ are opposite  (because $d^\ast((g_a.D)^-_t,(g_a.D)^+_t)\in W_{\lambda^{++}}$) for all $t\in [0,1]$, and thus by \cite[Lemma 4.9]{gaussent2014spherical} it proves that $(g_a.\fc)$ is a segment, starting at $0_\A$.
    
    Let us now prove that $\rho_{C_0}(g_a.\underline\fc)=\underline \fs_\bx$. By \cite[Proposition 5.17 (ii)]{hebert2020new} there exists an apartment $A_\oplus$ containing $C_0$ and $g_a.\underline{\fc}(1)$. Note that the distance $d^+((g_a.D)^-_1,g_a.\underline{\fc}(1))=w^-_{\bx,1}=d^-(D^-_{\bx,1},C_\bx)$ is minimal in its $W_{\lambda^{++}}$-coset $w^-_{\bx,1}W_{\lambda^{++}}$ (because $D^-_{\bx,1}=\pr_{\fs_{\lambda,-}(1)}(C_\bx)$), therefore $(g_a.D)^-_1=\pr_{(g_a.\fc)_-(1)}(g_a.\underline{\fc}(1))$. Thus by convexity, $A_\oplus$ also contains  $(g_a.D)^-_1$. Since $g_a.\fc$ is a segment and since, by admissibility, $(g_a.D)^\varepsilon_t=\pr_{(g_a.\fc)_\varepsilon(t)}((g_a.D)^-_1)$, it also contains $(g_a.D)^\varepsilon_t$ for all $(t,\varepsilon)\in[0,1]^\pm$. By (MA II), there exists $h\in G_{twin}$ such that $h.A_{\oplus}= \A$  fixing $\A_\oplus\cap A_\oplus$, in particular fixing $C_0$. Therefore $hg_a.\underline \fc=\rho_{C_0}(g_a\underline\fc)$ is an open segment starting at $0_\A$ (in particular it is admissible).
    Since $\fc'_+(t)\in -W^v\lambda$ for all $t$, the derivative $(hg_a.\fc)'_+(t)$ lies in $-W^v\lambda$ for all $t$. Moreover, since $(g_a.D)^+_0\in \cC(\underline{\fc},\bx,0)$, $d^\ast(C_0,(g_a.D)^+_0)=w^-_{\bx,0}=d^\ast(C_0,D^+_{\bx,0})$, so $\rho_{C_0}((g_a.D)^+_0)=D^+_{\bx,0}$. This implies that $(hg_a.\fc)_+(0)$ is the segment germ $\fs_{\bx,+}(0)$, so for any $t\in [0,1)$, $(hg_a.\fc)'_+(t)=(hg_a.\fc)'_+(0)=-\lambda$, thus $hg_a\fc$ is the segment $\fs_\lambda$. Since $hg_a.D=\rho_{C_0}(g_a.D)$ satisfies the admissibility condition, the equality $\rho_{C_0}((g_a.D)^+_0)=D^+_{\bx,0}$ extends to any $(t,\varepsilon)\in[0,1]^\pm$. Finally $d^+((g_a.D)^-_1,g_a.\underline{\fc}(1))=d^+(h(g_a.D)^-_1,hg_a.\underline{\fc}(1))=d^+(D^-_{\bx,1},C_\bx)$, hence $\rho_{C_0}(g_a.\underline{\fc}(1))=hg_a.\underline{\fc}(1)=C_\bx$, so $\rho_{C_0}(g_a.\underline{\fc})=hg_a.\underline{\fc}=\underline{\fs}_\bx$. We thus have proved that $g_a.\underline{\fc}$ is an open segment of type $\bx$.
\end{proof}

\begin{Lemma}\label{l_continuity}
Let $f_1,f_2:[0,1]\rightarrow \A$ be two continuous maps and $g_1,g_2\in G$ be such that $g_1.f_1(r)=g_2.f_2(r)$, for every $r\in [0,1)$. Then $g_1.f_1(1)=g_2.f_2(1)$. 
\end{Lemma}

\begin{proof}
 By Proposition~\ref{p_BT_7.4.8}, there exists $n\in N$   such that for every $x\in g_2^{-1}g_1.\A\cap \A$, we have $g_1^{-1}g_2.x=n.x$. The element $n$ induces an affine automorphism of $\A$ and thus its restriction to $\A$ is continuous. Therefore $f_1(1)=n.f_2(1)$ and hence $g_1.f_1(1)=g_2.f_2(1)$.  
\end{proof}

\begin{Proposition}\label{Prop: local description bijection}
     Let $\underline{\fc}$ be an admissible $C_\infty$-Hecke open path of type $\bx\in W^+_{sph}$. Then the map $a\mapsto g_a.\underline{\fc}$ defined in Lemma~\ref{Lemma: global lift} defines a bijection

    \begin{equation}\label{eq: decoratedlifts}
        \prod\limits_{t\in T(\underline{\fc})}\cC(\underline{\fc},\bx,t)\cong\rho_{C_0}^{-1}(\underline\fs_\bx)\cap \rho_{C_\infty}^{-1}(\underline{\fc}).
    \end{equation}
\end{Proposition}

\begin{proof}
 
  By Lemma~\ref{Lemma: global lift}, the map $a\mapsto g_a.\underline{\fc}$ is a well defined map associating, to a sequence of alcoves $a\in \prod\limits_{t\in T(\underline{\fc})}\cC(\underline{\fc},\bx,t)$, an open segment in $\rho_{C_0}^{-1}(\underline\fs_\bx)\cap \rho_{C_\infty}^{-1}(\underline{\fc})$. It remains to prove that it is injective and surjective.

    \begin{enumerate}

    \item \underline{Injectivity:} Suppose that $a,a'$ are two distinct elements of $\prod\limits_{t\in T(\underline{\fc})} \cC(\underline{\fc},\bx,t)$. Let $t$ be the first time such that $a_t\neq a'_t$. If $t<1$, then $g_a(t)=g_{a'}(t)$ and $i_{a_t}D^+_t = a_t \neq  a'_t=i_{a'_t}D^+_t$, so $(g_a.D)^+_t\neq (g_{a'}.D)^+_t$. If $t=1$, then $g_a.\underline{\fc}(1)=a_1\neq a'_1=g_{a'}.\underline{\fc}(1)$. 
    Either way, $g_a.\underline{\fc} \neq g_{a'}.\underline{\fc}$, hence the map is injective (and this also shows injectivity of $(a_s)_{s\in [0,t)\cap T(\underline{\fc})}\mapsto g_a.\underline\fc|_{[0,t)}$ for all $t\in(0,1]$).
    
    \item \underline{Surjectivity:} Let $\underline\fs=(\fs,D_\fs,\underline\fs(1)) \in \rho_{C_0}^{-1}(\underline\fs_\bx)\cap \rho_{C_\infty}^{-1}(\underline{\fc})$. Since $C_0$ is based at $0_\A$, $\rho_{C_0}^{-1}(0_\A)=\{0_\A\}$ and thus $\fs(0)=\fs_\lambda(0)=0_\A$. Let $t$ be the upper bound of the set $\mathcal U$ of times $s$ such that $\underline\fs|_{[0,s]}=g_a.\underline\fc|_{[0,s]}$ for some $a\in \prod\limits_{u\in T(\underline{\fc})\cap [0,s)}\cC(\underline{\fc},\bx,u)$. We aim to show that $t$ is a maximum and that $t=1$.
    
    By assumption, $D^+_{\fs,0} \in \cC(\underline{\fc},\bx,0)$ so for $r>0$ small enough (in particular such that $T(\fc)\cap[0,r)=\{0\}$), $\underline{\fs}|_{[0,r)}=g_{(D^+_{\fs,0})}.\underline{\fc}|_{[0,r)}$, which implies that $t>0$. Let $s\in[0,t)$ be such that $s>\max\{u\in T(\underline{\fc})\mid u<t\}$, and let $a=(a_u)_{u\in T(\underline{\fc})\cap [0,s)}$ be such that $\fs|_{[0,s]}=g_a.\fc_{[0,s]}$. For any $r\in (s,t)$ we know that $\underline\fs|_{[0,r]}=g_a.\underline\fc|_{[0,r]}$. Indeed, if $s\leq r<t$ then $r\in\mathcal U$ so there exists $a'\in \prod\limits_{u\in T(\underline{\fc})\cap [0,r)}\cC(\underline{\fc},\bx,u)$ such that $\underline\fs|_{[0,r]}=g_{a'}.\underline\fc|_{[0,r]}$. By injectivity of the map $(a_u)_{u\in [0,s)}\mapsto g_a.\underline\fc|_{[0,s)}$, $a_u=a'_u$ for all $u<s$ and thus, since $T(\fc)\cap [s,t)=\emptyset$, this is also true for all $u\leq r$, therefore $\underline\fs|_{[0,r]}=g_a.\underline\fc|_{[0,r]}$.
    By Lemma~\ref{l_continuity},
    $g_a.\fc(t)=\fs(t)$. Moreover since $D^+_{\fs,s}=g_a(s)D^+_s$ and $g_a(s)=g_a(t)$ we have: $$D^-_{\fs,t}=\pr_{\fs_-(t)}(D^+_{\fs,s})=\pr_{g_a(t)\fc_-(t)}(g_a(t)D^+_s)=g_a(t)\pr_{\fc_-(t)}(D^+_s)=g_a(t)D^-_t.$$ Therefore $D^-_{\fs,t}=g_a(t)D^-_t$  and $t=\max \cU$. 
    
    Assume by contradiction that $t<1$. Note that $D^+_{\fs,t}=\pr_{\fs_+(t)}(D^-_t)$ is the unique alcove containing $\fs_+(t)$ at codistance $w^-_{\bx,t}$ from $D^-_{\fs,t}$, and since $\rho_{C_\infty}(D^+_{\fs,t})=D^+_t$ we have $d^-(C^\infty_{\fs(t)},D^+_{\fs,t})=w^\infty_{\underline{\fc},t}$. Applying $g_a(t)^{-1}$, which preserves distance, codistance and sends $D^-_{\fs,t}$ to $D^-_t$, $C^\infty_{\fs(t)}$ to $C^\infty_{\fc(t)}$, we deduce that $g_a(t)^{-1}D^+_{\fs,t}$ is an element of $\cC(\underline{\fc},\bx,t)$. 
        
    If $t\in T(\underline{\fc})$, we extend $a=(a_u)_{u\in T(\underline{\fc})\cap [0,t)}$ setting $a_t=g_a(t)^{-1}D^+_{\fs,t}$ and then by construction $(g_a.\fc)_+(t)=\fs_+(t)$. If $t\notin T(\underline{\fc})$, then since there is no wall between $C^\infty_{\fc(t)}$ and $\fc_+(t)$, the fact that $C^\infty_{\fs(t)}=g_a(t).C^\infty_{\fc(t)}$ already implies that $g_a(t).\fc_+(t)=\fs_+(t)$ (because $\fc_+(t)$ is in the enclosure of $C^\infty_{\fc(t)}$, see \cite[Definition 2.1.2.4)]{bardy2025twin}). Either way there is $g_a$ such that $\fs_+(t)=(g_a.\fc)_+(t)$ and in particular there exists $r>t$ such that $\fs|_{[0,r]}=g_a.\fc|_{[0,r]}$, contradicting the definition of $t$. Thus $t=1$, and there is a sequence $a\in \prod\limits_{u\in T(\underline{\fc})\setminus \{1\}}\cC(\underline{\fc},\bx,u)$ such that $(\fs,D_\fs)=(g_a.\fc,g_a.D)$. To prove surjectivity, it remains to check that, for such a sequence, the end alcove $g_a(1)^{-1}\underline\fs(1)$ belongs to $\cC(\underline{\fc},\bx,1)$. The element $g_a(1)$ defines an isomorphism of local tangent buildings $\cT^\pm_{\fc(1)}(\I_\oplus)\cong \cT^\pm_{\fs(1)}(\I_\oplus)$, which sends the alcoves $C^\infty_{\fc(t)}$ and $D^-_1$ to $C^\infty_{\fs(1)}$ and $D^-_{\fs,1}$ respectively. Therefore $a_1:=g_a(1)^{-1}\underline\fs(1)$ satisfies $d^\ast(C^\infty_{\fc(1)},a_1)=d^\ast(C^\infty_{\fs(1)},\underline\fs(1))$ and $d^+(D^-_1,a_1)=d^+(D^-_{\fs,1},\underline\fs(1))$. Since $d^\ast(C^\infty_{\fs(1)},\underline\fs(1))=d^\ast(C^\infty_{\fc(1)},\underline{\fc}(1))$ (because $\rho_{C_\infty}(\underline\fs(1))=\underline{\fc}(1)$) and $d^+(D^-_{\fs,1},\underline\fs(1))=d^+(D^-_{\bx,1},C_\bx)$ (because $\rho_{C_0}(\underline\fs)=\underline\fs_\bx$), we deduce that $a_1$ belongs to $\cC(\underline{\fc},\bx,1)$, and we conclude that the map $C\mapsto g_a.\underline{\fc}$ is surjective. 
\end{enumerate}
\end{proof}

 In order to deduce numerical expressions from Formula \eqref{eq: decoratedlifts}, it therefore remains to compute, for a given $C_\infty$-Hecke open path $\underline{\fc}$ the cardinals of the sets $\cC(\underline{\fc},\bx,t)$ for all $t\in T(\underline{\fc})$. This is doable at the level of local tangent buildings using the techniques below.

\subsubsection{Computation at the level of local tangent buildings}\label{subsubsection: local computation}
Let us start by recalling results which hold for general twin-buildings, they will be used in order to compute the cardinality of the sets $\cC(\underline{\fc},\bx,t)$ appearing in the preceding subsection. 

\begin{Lemma}\label{Lemma: twinbuildings1}Let $(\cT^\pm,d^\pm,d^\ast)$ be a twin building and let $\A^\pm$ be a twin apartment of $\cT^\pm$.
Let $C_- \in \A^-$ and $C_+ \in \A^+$ be two chambers of opposite signs and let $w,v$ be two elements of the Weyl group $W$ of $\cT^\pm$. Let $\mathbf i_w$ be a reduced expression of $w$ and let $C_v$ be the unique chamber of $\A^+$ at codistance $v$ from $C_-$. Let $\Gamma_{C_-}(C_+,C_v,\mathbf i_w,\A^\pm)$ be the set of $C_-$-centrifugally folded galleries of type $\mathbf i_w$ from $C_+$ to $C_v$ in $\A^\pm$ . For any $\mathbf g \in \Gamma_{C_-}(C_+,C_v,\mathbf i_w,\A^\pm)$, let $\mathcal C^m_{C_-}(\mathbf g,C_+,\cT^+)$ denote the set of minimal galleries, in $\cT^+$, which retract to $\mathbf g$ by $\rho_{\A^\pm,C_-}$.

Then, the map which sends a gallery to its last chamber defines a bijection:

\begin{align*}
    \bigsqcup\limits_{\mathbf g \in  \Gamma_{C_-}(C_+,C_v,\mathbf i_w,\A^\pm)}\mathcal C^m_{C_-}(\mathbf g,C_+,\cT^+)\cong \{C\in \cT^+ \mid d^+(C_+,C)=w,\; d^\ast(C_-,C)=v\}.
\end{align*}

\end{Lemma}
\begin{proof}
    Suppose that $C\in \cT^+$ verifies $d^+(C_+,C)=w$, then there exists a unique minimal gallery $\tilde{\mathbf g}=(\tilde C_0,\dots,\tilde C_n)$ of type $\mathbf i_w=(s_1,\dots,s_n)$ from $C_+$ to $C$ in $\cT^+$. Let $\mathbf g=(C_0,\ldots,C_n)$ be the image of $\tilde{\mathbf g}$ by $\rho_{\A^\pm,C_-}$, we need to verify that it is centrifugally folded. Suppose that $C_i=C_{i-1}$. By definition, $\rho_{\A^\pm,C_-}$ preserves the codistance to $C_-$, so $d^\ast(C_-,\tilde C_i)=d^\ast(C_-,\tilde C_{i-1})=d^\ast(C_-,C_i)$. Moreover  and $d(\tilde C_{i-1},\tilde C_i)=s_i$. Therefore, by the converse of axiom $(TW2)$ of \cite[Definition 5.133]{abramenko2008buldings}, the fact that $d^\ast(C_-,\tilde{C}_i)=d^\ast(C_-,\tilde{C}_{i-1})$ implies that $d^\ast(C_-,\tilde{C}_{i-1})s_i>d^\ast(C_-,\tilde{C}_{i-1})$ and since $d^\ast(C_-,C_{i-1})=d^\ast(C_-,C_i)=d^\ast(C_-,\tilde C_{i-1})$ this implies that $\mathbf g$ is $C_-$-centrifugally folded at $i$. Hence we have an injection: $$\{C\in \cT^+ \mid d^+(C_+,C)=w,\; d^\ast(C_-,C)=v\}\hookrightarrow \bigsqcup\limits_{\mathbf g \in  \Gamma_{C_-}(C_+,C_v,\mathbf i_w,\A^\pm)}\mathcal C^m_{C_-}(\mathbf g,C_+,\cT^+).$$  
    Conversely, suppose that $\tilde g$ lies in $\mathcal C^m_{C_-}(\mathbf g,C_+,\cT^+)$ for some $\mathbf g \in  \Gamma_{C_-}(C_+,C_v,\mathbf i_w,\A^\pm)$, then since $\tilde g$ is minimal of type $\mathbf i_w$ and starts at $C_+$, its last chamber $C$ is at distance $w$ of $C_+$. Moreover, since $\rho_{\A^\pm,C_-}(C)=C_v$, we have $d^\ast(C_-,C)=d^\ast(C_-,C_v)=v$. Hence we obtain an inverse map of the injection above, and thus a bijection: 
    $$\bigsqcup\limits_{\mathbf g \in  \Gamma_{C_-}(C_+,C_v,\mathbf i_w,\A^\pm)}\mathcal C^m_{C_-}(\mathbf g,C_+,\cT^+) \hookrightarrow \{C\in \cT^+ \mid d^+(C_+,C)=w,\; d^\ast(C_-,C)=v\}.$$
    
\end{proof}
The next lemma is proved in more generality in \cite[Corollary 4.5]{gaussent2014spherical}, it is also detailed in \cite[\S 2.3]{bardy2021structure}.
\begin{Lemma}\label{Lemma: twinbuildings2} Let $\A^\pm$ be a twin apartment of type $(W,S)$, let $C_- \in \A^-$ and $C_+ \in \A^+$ be two chambers of opposite signs. Let $\mathcal P$ be a set of panels in $\A^+$ (which should be thought of as thick panels). Let $\mathbf g=(C_1,\dots,C_n)$ be a $C_-$-centrifugally folded gallery of type  $\mathbf i =(\mathbf i_1,\dots,\mathbf i_n)$ in $\A^+$ from $C_0=C_+$ to $C_n=C_v$ and suppose that the panels along which $\mathbf g$ folds belong to $\mathcal P$.

Let $l(\mathbf g)= |\{j \in \llbracket 1,n\rrbracket \mid C_{j-1} \neq C_j,\; d^\ast(C_-,C_j)<d^\ast(C_-,C_{j-1}),\; p(C_j,s_j)\in\mathcal P \}|$\index{l@$l(\mathbf{g})$}  and $r(\mathbf g)=|\{j\in \llbracket1,n\rrbracket \mid C_{j-1}=C_j\}|$\index{r@$r(\mathbf{g})$}. Let $R^{C_-}_{\mathbf g,\cP}(X)$ denote the polynomial with integral coefficients defined by: \index{r@$R^{C_-}_{\mathbf g,\mathcal P}(X)$, $R^{C_-,C_+}_{v,w,\mathcal P}(X)$} $$R^{C_-}_{\mathbf g,\mathcal P}(X)=X^{l(\mathbf g)}(X-1)^{r(\mathbf g)}.$$

Then for any $(q,1)$-semi-homogeneous twin building $\cT^\pm$ with twin apartment $\A^\pm$ and such that $\mathcal P$ is the set of thick panels of $\A^+$ in $\cT^+$, we have:

    $$|\mathcal C^m_{C_-}(\mathbf g,C_+,\cT^+)|=R^{C_-}_{\mathbf g,\mathcal P}(q).$$

 In particular, it is non-zero for any gallery $\mathbf g$ which is $C_-$-centrifugally folded along panels of $\mathcal P$. If $\mathbf g$ folds along a panel which is not thick (not in $\mathcal P$), then it does not lift to any minimal gallery, so $|\mathcal C^m_{C_-}(\mathbf g,C_+,\cT^+)|=0$ and we set $R^{C_-}_{\mathbf g,\mathcal P} =0$. 

\end{Lemma}
This motivates the following definition, in order to explicit Formula \eqref{eq: decoratedlifts}.
\begin{Definition} \label{Def: localRpolynomial}
    Let $\A^\pm$ be a twin apartment and let $\mathcal P$  be a set of panels in $\A^+$. Let $C_-\in \A^-$ and $C_+\in \A^+$, let $v,w\in W$. As in Lemma \ref{Lemma: twinbuildings1}, let $\mathbf i_w$ be a reduced expression of $w$, and let $C_v$ be the unique chamber of $\A^+$ at codistance $v$ from $C_-$. Then we define the polynomial $R^{C_-,C_+}_{v,w,\mathcal P}(X)\in \mathbb Z[X]$ by:
    \begin{equation}
        R^{C_-,C_+}_{v,w,\mathcal P}(X)=\sum\limits_{\mathbf g \in  \Gamma_{C_-}(C_+,C_v,\mathbf i_w,\A^\pm)} R^{C_-}_{\mathbf g,\mathcal P}(X).
    \end{equation}
    By Lemmas \ref{Lemma: twinbuildings1} and \ref{Lemma: twinbuildings2}, it is the unique polynomial such that, if $\cT^\pm$ is a $(q,1)$-semi-homogeneous twin building with twin apartment $\A^\pm$ and $\mathcal P$ is its set of $q$-thick panels in $\A^+$  we have:
    \begin{equation*}
        |\{C\in \cT^+ \mid d^+(C_+,C)=w,\; d^\ast(C_-,C)=v\}|= R^{C_-,C_+}_{v,w,\mathcal P}(q).
    \end{equation*}
In particular, this polynomial does not depend on the choice of $\mathbf i_w$.

\end{Definition}
\begin{Remark}\label{r_usual_KL_polynomials}
    These polynomials should be thought of as twisted versions, with unequal parameters, of the R-polynomials associated to a twin building, or to a Hecke algebra. Indeed when $\mathcal P$ is the set of all panels of $\A^+$, then $R^{C_-,C_+}_{v,w,\mathcal P}$ only depends on $v,w,d^\ast(C_-,C_+)\in W$, and if $d^\ast(C_-,C_+)=1$ then it is equal to the classical polynomial $R_{v,w}(X)$ defined by Kazhdan and Lusztig \cite{kazhdan1979representations}.

\end{Remark}

\subsubsection{Polynomial expression for the lifts of $C_\infty$-Hecke open paths}\label{subsubsection: Polynomial expression for the lifts}
Combining the results of Subsections~\ref{subsubsection: local description Cinfty hecke} and~\ref{subsubsection: local computation}, we obtain the following result.

\begin{Theorem}\label{Theorem : number of lifts spherical} 
    Let $\mathcal D$ be any Kac-Moody root datum, and let $\A$ be the associated affine apartment. Let $\underline{\fc}$ be a $C_\infty$-Hecke open path of type $\bx\in W^+_{sph}$ in $\A$. Then, there exists a polynomial \index{r@$R_{\bx,\underline{\fc}}(X)$, $R^{loc,t}_{\bx,\underline{\fc}}(X)$} $R_{\bx,\underline{\fc}}(X)\in \mathbb Z[X]$ such that, for any twin masure $(\I_\oplus,\I_\ominus)$ with standard apartment $\A$ and constant finite thickness $q$, we have:
   
    \begin{equation}\label{eq : R_counts_retractions}
        |\rho_{C_0}^{-1}(\underline\fs_{\bx})\cap \rho_{C_\infty}^{-1}(\underline{\fc})|=R_{\bx,\underline{\fc}}(q)
    \end{equation}

    Moreover, this polynomial is explicitly given by: 
    \begin{equation}\label{eq:openpaths}
        R_{\bx,\underline{\fc}}(X)=\prod\limits_{t \in T(\underline{\fc})} R^{loc,t}_{\bx,\underline{\fc}}(X)
    \end{equation}
    where \begin{equation}R^{loc,t}_{\bx,\underline{\fc}}(X)=\begin{cases} R^{D^-_t,C^\infty_{\fc(t)}}_{w^-_{\bx,t},w^\infty_{\underline{\fc},t},\mathcal P_{\fc(t)}}(X) \text{ if } t<1 \\
        R^{C^\infty_{\fc(1)},D^-_1}_{w^\infty_{\underline \fc,1},w^-_{\bx,1},\mathcal P_{\fc(1)}}(X) \text{ if } t=1
    \end{cases}\end{equation} with the notation of Definition \ref{Def: localRpolynomial}, and $\mathcal P_{\fc(t)}$ denotes the set of panels of $\cT^\pm_{\fc(t)}(\A)$ supported on affine walls.
    
\end{Theorem}
\begin{proof}
    This is a direct application of Formula~\eqref{eq: decoratedlifts} and Lemmas~\ref{Lemma: twinbuildings1},~\ref{Lemma: twinbuildings2}.
\end{proof}
\begin{Remark}
Note that for any Kac-Moody root datum and associated apartment $\A$, there exists a twin masure of standard apartment $\A$ and of constant finite thickness $q$ for any prime power $q$. Therefore the polynomials $R_{\underline{\fc},\qp^\lambda w}$ are characterized by Formula \eqref{eq : R_counts_retractions}. 
\end{Remark}

\subsection{Definition and computation of $R$-Kazhdan-Lusztig polynomials}\label{subsection: Definition Kazhdan-Lusztig Polynomials}
We use the computations above to define affine $R$-Kazhdan-Lusztig polynomials for Kac-Moody groups. This follows the work of D. Muthiah in \cite{muthiah2019double}, where he defines spherical $R$-Kazhdan-Lusztig polynomials for affine simply laced Kac-Moody groups. Spherical and affine Kazhdan-Lusztig polynomials play a very similar role, but the former for spherical Hecke algebras, and the latter for Iwahori-Hecke algebras.
\subsubsection{Definition of affine $R$-Kazhdan-Lusztig polynomials for Kac-Moody groups}

\begin{Lemma}\label{l_bijection_paths_alcoves}
    Let $\bx,\by\in W^+$, with $\bx$ spherical. Then the map $\psi:\underline{s}\mapsto \underline{s}(1)$ is a bijection between $E_{\bx,\by}:=\bigsqcup_{\underline{\fc}\in \cC^\infty_{\bx}(\by)}\rho_{C_\infty}^{-1}(\underline{\fc})\cap \rho_{C_0}^{-1}(\underline{\fs}_\bx)$ and the set $E'_{\bx,\by}$ of alcoves  $C$ of $\I_{\oplus}$ such that: \begin{enumerate}
        \item $C$ and $0_\A$ belong to an apartment of $\I_{\oplus}$,
        
        \item every element of $[0_\A,\ve(C)]$ is $C_\infty$-friendly, where $\ve(C)$ is the vertex of $C$,

        \item $\rho_{C_0}(C)=C_\bx$ and $\rho_{C_\infty}(C)=C_\by$. 
        
    \end{enumerate}  Its reciprocal is the map $\psi':E'_{\bx,\by}\rightarrow E_{\bx,\by}$ which associates to every $C$ the open segment starting at $0_\A$ and with end alcove $C$ (which is well-defined by Lemma~\ref{l_lift_open_segment}).
\end{Lemma}

\begin{proof}
  The fact that $\psi(E_{\bx,\by})\subset E'_{\bx,\by}$ is clear. Conversely, let $C\in E'_{\bx,\by}$ and let $\underline{\fs}$ be the open segment starting at $0_\A$ and with end alcove $C$. Then $\underline{\fs}$ is of type $\bx$ by Lemma~\ref{l_lift_open_segment}. In other words, $\rho_{C_0}(\underline{\fs})=\underline{\fs}_\bx$. Moreover, by Proposition~\ref{Proposition: retraction is C infty}, $\rho_{C_\infty}(\underline{\fs})\in \cC^\infty_\bx(\by)$, so $\psi'(E'_{\bx,\by})\subset E_{\bx,\by}$. The fact that $\psi\circ \psi'=\Id$ is clear and the fact that $\psi'\circ \psi=\Id$ follows from the uniqueness in Lemma~\ref{l_lift_open_segment} 1).
\end{proof}

\begin{Definition}[Affine $R$-Kazhdan-Lusztig polynomials]

    Let $\bx,\by$ be elements of $W^+$, suppose that $\bx$ is spherical. Let \index{c@$\cC^\infty_\bx(\by)$} $\cC^\infty_\bx(\by)$ be the finite set of $C_\infty$-Hecke open paths of type $\bx$ starting at $0_\A$ and with end alcove $C_{\by}$. 
    Then define the $R$-Kazhdan-Lusztig polynomial \index{r@$R_{\bx,\by}$}$R_{\bx,\by}$ by:
    \begin{equation}\label{eq: DefRPolynomials}
        R_{\bx,\by}(X)=\sum\limits_{\underline{\fc}\in \cC^\infty_\bx(\by)}R_{\bx,\underline{\fc}}(X)=\sum\limits_{\underline{\fc} \in \cC^\infty_\bx(\by)}\prod\limits_{t \in T(\underline{\fc})} R^{loc,t}_{\bx,\underline{\fc}}(X).
    \end{equation}
    It is the polynomial such that: \begin{equation}R_{\bx,\by}(q)=\sum\limits_{\underline{\fc} \in \cC^\infty_\bx(\by)}|\rho_{C_\infty}^{-1}(\underline{\fc})\cap \rho_{C_0}^{-1}(\underline\fs_{\bx})|.\end{equation}

\end{Definition}
Upon Conjecture \cite[4.4.1]{bardy2025twin} and according to Lemma~\ref{l_bijection_paths_alcoves}, we have:
$R_{\bx,\by}(q)=|I_0 \bx I_0 \cap I_\infty \by I_0/I_0|$. This motivates their study, we expect them to be the $R$-polynomials for a, yet to be defined, Kazhdan-Lusztig theory for Kac-Moody groups over discretely valued fields. This was schemed by Muthiah in \cite{muthiah2019bruhat}.

\subsubsection{Comparison with Spherical $R$-polynomials}
Let $K\supset I$\index{k@$K$} denote the stabiliser of $0_\A$ in $G$. In \cite{muthiah2019double}, Muthiah defines spherical $R$-polynomials, which conjecturally compute the cardinal of the sets $(K \qp^\lambda K \cap I_\infty \qp^\mu K)/K$ for $\lambda \in Y^{++}_{sph}$ and $\mu \in Y^+$ (see \cite[\S 5]{muthiah2019double}). In order to do this, he computes certain sets of segments in the masure $\I_\oplus$, and his computation were made more precise by Bardy-Panse, the first author and Rousseau in \cite{bardy2025twin}. Let us briefly explain how to relate the affine $R$-polynomials defined here with their results.

Since $K$ is the stabilizer of the special vertex $0_\A$, the set $(K \qp^\lambda K \cap I_\infty \qp^\mu K)/K$ bijects with $\rho_{0_\A}^{-1}(-\lambda)\cap \rho_{C_\infty}^{-1}(-\mu)$, where $\rho_{0_\A}$ is the masure retraction centered at $0_\A$. More precisely $\rho_{0_\A}(x)$ is the unique element of $K.x\cap \pm \overline{C^v_f}$ for all $x\in \{x\in \I_{\oplus}\mid x\leq_{\oplus}0_\A\text{ or }x\geq_{\oplus} 0_\A\}$. Since the fundamental alcove is based at $0_\A$, we have $\rho_{0_\A}=\rho_{0_\A} \circ \rho_{C_0}$ and, since $\rho_{0_\A}$ restricts to $x\mapsto x^{++}$ on $\pm \sT$, we have $\rho_{0_\A}(x)=(\rho_{C_0}(x))^{++}$ for every $x\leq_\oplus 0_\A\in \I_\oplus$. Therefore $K \qp^\lambda K=\bigsqcup\limits_{w\in W^\lambda} I \qp^{w\lambda} K$ for any $\lambda \in Y^{++}$ (recall that $W^\lambda$ is the set of minimal coset representatives of $W^v/W_\lambda$). In term of masure retractions, we have:
$$\rho_{0_\A}^{-1}(-\lambda)\cap \rho_{C_\infty}^{-1}(-\mu)=\bigsqcup\limits_{w\in W^\lambda}\rho_{C_0}^{-1}(-w\lambda)\cap \rho_{C_\infty}^{-1}(-\mu).$$

Since the alcove $C_{\qp^{w\lambda}w}$ is the projection $C^{++}_{-w\lambda}=\pr_{-w\lambda}(C_0)$ of the fundamental alcove at $-w\lambda$, the map $x\mapsto \pr_x(C_0)$ defines a bijection $\rho_{C_0}^{-1}(-w\lambda)\cong \rho_{C_0}^{-1}(C^{++}_{-w\lambda})$. 

For any $C_\infty$-Hecke non-open path of type $\lambda$ ending at $-\mu$ (defined as $I_\infty$-Hecke paths in \cite[\S 5.3.1]{muthiah2019double}), and for any $w\in W^\lambda$ let $\mathcal H_{\fc,w}$ denote  the set of $C_\infty$-Hecke open paths of type $\qp^{w\lambda}w$ with underlying non-open path $\fc$.

Then there is a bijection between the set of $C_\infty$-friendly segments of type $-w\lambda$ retracting to $\fc$ and the set of $C_\infty$-Hecke open segments of type $\qp^{w\lambda}w$ retracting to some (necessarily unique) $\underline{\fc}\in \mathcal H_{\fc,w}$:
\begin{equation}\label{eq: bijection non open open}
    \rho_{C_0}^{-1}(\fs_{w\lambda})\cap \rho_{C_\infty}^{-1}(\fc) \cong \bigsqcup\limits_{\underline{\fc}\in \mathcal H_{w,\fc}}(\rho_{C_0}^{-1}(\underline\fs_{\qp^{w\lambda}w})\cap \rho_{C_\infty}^{-1}(\underline\fc)).
\end{equation}
In \cite[Theorem 5.1]{bardy2025twin}, the authors compute $\rho_{0_\A}^{-1}(\fs_{\lambda})\cap \rho_{C_\infty}^{-1}(\fc)=\bigsqcup\limits_{w\in W^{\lambda}}\rho_{C_0}^{-1}(\fs_{w\lambda})\cap \rho_{C_\infty}^{-1}(\fc)$.

We recover a polynomial expression for the cardinal of this set (and in particular its finiteness) from the definition of the polynomials $R_{\qp^{w\lambda}w,\underline{\fc}}$, Formula~\eqref{eq: bijection non open open} and the first two points of the following Lemma.
\begin{Lemma}\label{Lemma: finiteness nonopen} The following sets are finite:
\begin{enumerate}
    \item for a fixed pair $\fc,w$, the set $\mathcal H_{\fc,w}$, 
    \item for a fixed $C_\infty$-Hecke path $\fc$, the set  $\{w\in W^\lambda \mid \mathcal H_{\fc,w}\neq \emptyset\}$,
    \item for a fixed pair $\lambda\in Y^{+}_{sph}$, $\mu \in Y^+$, the set $\{(w,v) \in W^\lambda \times W^v \mid R_{\qp^{w\lambda}w,\qp^\mu v}\neq 0\}$
    
\end{enumerate}
\end{Lemma}
\begin{proof}

In this proof, we use the length $\ell^a:W^+\rightarrow \R$ introduced by Muthiah in \cite[5]{muthiah2018iwahori}. It is also defined in \cite[(1.19)]{philippe2023grading}.
    
    1.,2. Let $\fc$ be a  $C^\infty$ non-open path of type $\lambda$ and ending at $-\mu$. Let $w\in W^\lambda$ and assume that $\cH_{\fc,w}$ is non-empty. Then $\underline{\fc}(1)$ is based at $\mu$. Combining Proposition~\ref{Proposition: open paths give chains} and \cite[Corollary 2.12]{philippe2023grading}, we get: \[ \qp^\mu v^\mu\leq \underline{\fc}(1)\leq \qp^{w\lambda}w,\]
where $v^\mu$ is the minimal element such that $\mu = v^\mu \mu^{++}$. Using the finiteness of the intervals for the Bruhat order ( \cite[Corollary 3.4]{hebert2024quantum}), we deduce 1. Moreover by \cite[Theorem 1.1]{muthiah2019bruhat}, $\ell^a$ is compatible with $\leq$ and thus  we have $\ell^a(\qp^\mu v^\mu)\leq \ell^a(\qp^{w\lambda}w)$. Using \cite[(2.8)]{philippe2023grading} this is equivalent to:  \[\ell(w)\leq \ell^a(\qp^\lambda)+\ell^a(v^\mu)-\ell^a(\qp^\mu).\] Since $W^v$ admits finitely many simple reflections, the set $\{x\in W^\lambda \mid \ell(x)\leq \ell^a(\qp^\lambda)+\ell^a(v^\mu)-\ell^a(\qp^\mu)\}$ is finite, from which we deduce 2.

    \item The polynomial $R_{\qp^{w\lambda}w,\qp^\mu v}$ is non zero if and only if there is a $C_\infty$-Hecke open path of type $\qp^{w\lambda}w$ ending at $C_{\qp^\mu v}$, which is not possible if $\qp^\mu v \not\leq \qp^{w\lambda}w$. Since $\qp^\mu W^v$ is bounded below by $\qp^\mu v^\mu$, and $\qp^{w\lambda}w$ is bounded above by $\qp^\lambda$, by finiteness of the interval between these two elements for the affine Bruhat order again, we deduce this third finiteness result.

\end{proof}
From the third point of this lemma, we also obtain an alternative description of the spherical $R$-polynomial $R^K_{\mu,\lambda}$ introduced by Muthiah (see \cite[Definition 5.55]{muthiah2019double}). It is defined as $$R^K_{\mu,\lambda}(q)=\sum\limits_{\fc \in \cC^\infty_\lambda(\mu)} |\rho_{0_\A}^{-1}(\fs_{\lambda})\cap \rho_{C_\infty}^{-1}(\fc)|,$$ for any twin masures $(\I_\oplus,\I_\ominus)$ of constant thickness $q$ (where $\cC^\infty_\lambda(\mu)$ denotes the set of $C_\infty$-Hecke (non-open) paths of type $\lambda$ ending at $-\mu$). Recollecting the previous arguments we deduce:

\begin{equation}\label{e_RK}
    R^K_{\mu,\lambda}(X)=\sum\limits_{w\in W^\lambda}\sum\limits_{v\in W^v} R_{\qp^{w\lambda}w,\qp^\mu v}(X).
\end{equation}
\color{black}
\subsubsection{Elementary properties of $R$-Kazhdan-Lusztig polynomials}

Our $R$-polynomials have properties quite similar to the defining properties of the classical $R$-polynomials for Coxeter groups. In the next Proposition, we state analogs of points $i)$ and $ii)$ of \cite[Theorem 5.1.1]{bjorner2005combinatorics}. In Proposition~\ref{Proposition Rpolynomialscoversregular} we compute $R$-polynomials for covers, it is the analog of a particular case of \cite[Theorem 5.1.1.iii)]{bjorner2005combinatorics}.

\begin{Proposition}\label{Prop: RPolynomials Properties}

    Let $\bx,\by \in W^+$ and suppose that $\bx$ is spherical. Then:
    \begin{enumerate}
        
        \item If $R_{\bx,\by}\neq 0$, then $\bx \geq \by$.
        \item We have $R_{\bx,\by}=1 \iff \bx=\by$.
    \end{enumerate}
\end{Proposition}
\begin{proof} If $R_{\bx,\by}$ is non-zero, then there is a $C_\infty$-Hecke open path of type $\bx$ with end alcove $C_\by$ which implies, by Proposition~\ref{Proposition: open paths give chains} 2., that $\by\leq\bx$.
   
    Let $\bx=\qp^\lambda w\in W^+_{sph}$. The unique $C_\infty$-Hecke open path of type $\bx$ ending at $C_\bx$ is the open segment $\underline\fs_\bx$, and $T(\underline\fs_\bx)=\{0,1\}$. Since $C^\infty_{0_\A}$ and $C_0$ are opposite in $\cT^\pm_{0_\A}(\I_\oplus)$, we have $\cC(\underline{\fs_\bx},\bx,0)=\{D^+_{\bx,0}\}$, hence $R^{loc,0}_{\bx,\underline{\fs}_\bx}=1$. It remains to check that $R^{loc,1}_{\bx,\underline{\fs}_\bx}=1$. Let $\mathbf i$ be the type of a reduced expression for $w^-_{\bx,1}=d^+(D^-_{\bx,1},C_\bx)$. There is a unique gallery $\mathbf g$ from $D^-_{\bx,1}$ to $C_\bx=\underline{\fs}_\bx(1)$ of type $\mathbf i$, and it is a minimal gallery. Therefore $R^{loc,1}_{\bx,\underline{\fs}_\bx}(X)=X^{l(\mathbf g)}$ where $l(\mathbf g)$ is the number of affine walls crossed by $\mathbf g$ away from $C^\infty_{-\lambda}$ or, equivalently, towards $C^{++}_{-\lambda}$, its opposite in $\cT^\pm(\A)$. 
    
    Let us check that $l(\mathbf g)=0$. On the one hand since $\mathbf g$ is minimal, the walls it crosses are crossed away from its initial alcove $D^-_{\bx,1}$ and since $D^-_{\bx,1}=\pr_{\fs_{\lambda,-}(1)}(C_\bx)$, it does not cross any wall containing the segment germ $\fs_{\lambda,-}(1)$. Since the non-open component of $\underline{\fs}_\bx$ is the segment $\fs_\lambda$ from $0_\A$ to $-\lambda$, we have $C^{++}_{-\lambda}=\pr_{\fs_{\lambda,-}(1)}(C_0)$ which dominates the segment germ $\fs_{\lambda,-}(1)$. Therefore a half-apartment containing $D^-_{\bx,1}$ also contains $C^{++}_{-\lambda}$, unless the affine wall delimiting it contains $\fs_{\lambda,-}(1)$. Since none of the walls crossed by $\mathbf g$ contain $\fs_{\lambda,-}(1)$ and since they are all crossed away from $D^-_{\bx,1}$, they are also crossed away from $C^{++}_{-\lambda}$, so $l(\mathbf g)=0$ and $R^{loc,1}_{\bx,\underline{\fs}_\bx}=1$.

       Conversely, suppose that $R_{\bx,\by}=1$. Then there is a unique $C_\infty$-Hecke open path of type $\by$ with end alcove $C_\bx$. Since any folding induces a factor $X-1$, it can not be folded. Therefore this Hecke open path is just the segment of type $\by$, and $\by=\bx$.

    \end{proof}

We will now compute $R$-polynomials associated to covers, the result is the same as for classical $R$-Kazhdan-Lusztig polynomials (see~\ref{Proposition Rpolynomialscoversregular}). The computation relies on several technical lemma, and on the classification of covers given in \cite[Proposition 3.20]{philippe2023grading}.
    \begin{Lemma}\label{Lemma: separating wall condition}
        Let $\fc_+(t)=\fc(t)+[0,0^+)\fc'_+(t)$ be a segment germ such that $\fc(t),\fc'_+(t)\in -\cT$ and let $\gamma \in \Phi$. \begin{enumerate}
            \item If $\langle \fc(t),\gamma\rangle\neq 0$, then the (affine or ghost) wall of direction $\gamma$ going through $\fc(t)$ separates $C^\infty_{\fc(t)}$ from $\fc_+(t)$ if and only if $\langle \fc(t),\gamma\rangle \langle \fc'_+(t),\gamma\rangle <0$.
            \item If $\langle \fc(t),\gamma\rangle =0$, then the affine wall of direction $\gamma$ going through $\fc(t)$ separates $C^\infty_{\fc(t)}$ from $\fc_+(t)$ if and only if $\sgn(\gamma)\langle \fc'_+(t),\gamma\rangle >0$.
            
        \end{enumerate}
    
    \end{Lemma}
\begin{proof}
\begin{enumerate}
    \item Let $n=-\langle \fc(t),\gamma\rangle$, so that $\fc(t)\in M_{\gamma[n]}$, and suppose $n\neq 0$. If $\langle \fc'_+(t),\gamma\rangle =0$ then the wall $M_{\gamma[n]}$ contains the segment germ $\fc_+(t)$, hence this segment germ is automatically on both sides of $M_{\gamma[n]}$, so there is nothing to prove. Let us suppose that it is not the case. By definition, $C^\infty_{\fc(t)}$ dominates the segment germ $\fc(t)+[0,0^+)\fc(t)$. Therefore the side of $M_{\gamma[n]}$ containing $C^\infty_{\fc(t)}$ contains this segment germ, and since $\langle\fc(t),\gamma\rangle \neq 0$ it is the only side of $M_{\gamma[n]}$ containing this segment germ. Therefore $C^\infty_{\fc(t)}$ and the segment germ $\fc_+(t)$ are separated by $M_{\gamma[n]}$ if and only if the respective directions $\fc(t)$ and $\fc'_+(t)$ of $\fc(t)+[0,0^+)\fc(t)$ and $\fc_+(t)$ are separated by the vectorial wall $M_\gamma$, if and only if $\langle \fc(t),\gamma\rangle$ and $\langle\fc'_+(t),\gamma\rangle $ have opposite sign, hence we get the first point.

    \item Suppose now that $\langle \fc(t),\gamma\rangle =0$. Up to replacing $\gamma$ by $-\gamma$, we suppose that $\gamma$ is positive. By definition $C^\infty_{\fc(t)}$ and the fundamental alcove $C_0$ are on opposite sides of $M_{\gamma[0]}$ (because this wall goes through $\fc(t)$), so $\fc_+(t)$ and $C^\infty_{\fc(t)}$ are on opposite sides if and only if $\fc_+(t)$ is on the same side of $M_{\gamma[0]}$ as $C_0$).  Since $\gamma$ is positive, $\langle C_0,\gamma\rangle >0$, so $C_0$ and $\fc_+(t)$ are on the same side of $M_{\gamma[0]}$ if and only if $\langle \fc'_+(t),\gamma\rangle >0$, which concludes point 2.
    
\end{enumerate}\end{proof}

    \begin{Lemma}\label{Lemma: length additivity inversion sets}
    Let $v_1,v_2\in W^v$ be such that $\ell(v_1v_2)=\ell(v_1)+\ell(v_2)$. Then $$\Inv(v_2^{-1})\cap \Inv(v_1)=\emptyset.$$ 
        
    \end{Lemma}
\begin{proof}
    Suppose that there exists $\gamma\in \Inv(v_2^{-1})\cap\Inv(v_1)$. Then $v_1s_\gamma <v_1$ and $v_2^{-1}s_\gamma<v_2^{-1}$ (so $s_\gamma v_2< v_2$).
     Therefore, by the triangular inequality, $$\ell(v_1v_2)=\ell( (v_1 s_\gamma)(s_\gamma v_2))\leq \ell(v_1s_\gamma)+\ell(s_\gamma v_2)<\ell(v_1)+\ell(v_2).$$
\end{proof}

\begin{Lemma}\label{Lemma: Rpol distinct coweight case}
    Let $\bx,\by\in W^+$ be such that $\by$ is covered by $\bx$. We assume that $\by=\qp^{v\lambda}w$ and $\bx=s_{v(\beta)[n]}\by=\qp^{v(s_\beta \lambda+n\beta^\vee)}s_{v(\beta)}w$, with $\lambda\in Y^{++}$, $v\in W^\lambda$, $w\in W^v$,  $\beta[n]\in \Phi_+^a$ and $n\in \{-1,\langle \lambda,\beta\rangle +1\}$.  Let $t_0=\frac{|n|}{\langle \lambda,\beta\rangle+2}$. Let $\underline{\fc}=\phi_{t_0,v(\beta)[n]}(\underline{\fs}_\bx)$. Then $T(\underline{\fc})=\{0,t_0,1\}$ and $R^{loc,t_0}_{\bx,\underline{\fc}}=X-1$. 
\end{Lemma}
\begin{proof}
  With the notation of the Lemma, since $\bx$ covers $\by$ with $\pr^{Y^{++}}(\bx)\neq \pr^{Y^{++}}(\by)$, by \cite[Proposition 3.21]{philippe2023grading}, $\beta$ needs to be a quantum root (see \cite[Definition 2.4]{hebert2024quantum}): \begin{equation}\label{eq: quantum_root}\forall \gamma \in \Inv(s_\beta)\setminus\{\beta\}, \, \langle \beta^\vee,\gamma\rangle =1.\end{equation}
        Moreover, $\lambda+\beta^\vee$ is almost-dominant, that is to say: 
    \begin{align}
    \langle \lambda+\beta^\vee,\tau\rangle \geq -1 \, \text{for all} \, \tau \in \Phi_+. \label{eq: almostdominance}
    \end{align}

        It is clear that $\{0,t_0,1\}\subset T(\underline{\fc})$ since there is a folding at $t_0$. Moreover since $\underline{\fc}$ is obtained from $\underline{\fs}_\bx$ by a folding at time $t_0$, both open paths agree on $[0,t_0)$ and thus we have that $T(\underline{\fc})\cap [0,t_0)=T(\underline{\fs}_\bx)\cap[0,t_0)=\{0\}$. Let us now prove that, on $[t_0,1)$, there is no other affine wall separating $C^\infty_{\fc(t)}$ and $\fc_+(t)$.

For any $\gamma \in \Phi_+$ and $t\in [t_0,1)$, let $M_{\gamma,t}$ denote the (affine or ghost) wall of direction $v(\gamma)$ going through $\fc(t)$, we say that it is separating if it separates $C^\infty_{\fc(t)}$ from $\fc_+(t)$. We want to prove that, if $M_{\gamma,t}$ is separating, then it is a ghost wall, unless $t=t_0$ and $\gamma=\beta$. We use Lemma~\ref{Lemma: separating wall condition} in order to verify this. Let $g_\gamma$ denote the affine function $t\mapsto \langle \fc(t),v(\gamma)\rangle$ defined on $[t_0,1]$. Then $M_{\gamma,t}$ is an affine (non-ghost) wall if and only if $g_\gamma(t)\in \Z$, and by Lemma~\ref{Lemma: separating wall condition} $M_{\gamma,t}$ separates $C^\infty_{\fc(t)}$ from $\fc_+(t)$ if and only if either $g_\gamma(t)\langle \fc'_+(t),v(\gamma)\rangle<0$ or $g_\gamma(t)=0$ and $\sgn(v(\gamma))\langle \fc'_+(t),v(\gamma)\rangle >0$. We will separate cases according to the value of $n$ and to the sign of $\langle \fc'_+(t),v(\gamma)\rangle$. This value does not depend on $t\in [t_0,1)$ as $\fc$ is affine on $[t_0,1]$.

Note that $g_\gamma(1)=\langle -\lambda,\gamma\rangle \leq 0$ and\begin{equation}\label{e_g_gamma_1} g_\gamma(t_0)=\begin{cases*}
        -t_0\langle \lambda+\beta^\vee,\gamma\rangle &\text{ if $ n=\langle \lambda,\beta\rangle+1$} \\
        -t_0 \langle \lambda+\beta^\vee,s_\beta(\gamma)\rangle &\text{ if $ n=-1$.}
    \end{cases*}
\end{equation}

In particular, since $\lambda+\beta^\vee$ is almost dominant and $\gamma \in \Phi_+$, $g_\gamma(t_0)\leq t_0<1$ unless $n=-1$ and $\gamma \in \Inv(s_\beta)$. Moreover if $n=-1$, clearly $g_\beta(t_0)=1$ and if $\gamma \in \Inv(s_\beta)\setminus\{\beta\}$, since $\beta$ is a quantum root , $\langle \beta^\vee,\gamma\rangle =1$ and $s_\beta(\gamma)=\gamma-\beta$, thus $g_\gamma(t_0)=\frac{1}{\langle \lambda+\beta^\vee,\beta\rangle}\langle \lambda+\beta^\vee,\beta-\gamma\rangle <1$ in this case aswell. Therefore we have
\begin{equation}
    g_\gamma(t_0)\leq 1,
\end{equation}with equality if and only if $n=-1$ and $\gamma=\beta$.
In particular outside of this case, since $g_\gamma$ is affine on $[t_0,1]$, $g_\gamma([t_0,1])\subset (-\infty,1)$: there is no time for which $g_\gamma(t)$ is integral and positive.
\begin{enumerate}
    \item Suppose first that $\langle \fc'_+(t),v(\gamma)\rangle<0$, and $\gamma\neq\beta$. Then since $g_\gamma(t)$ cannot be integral and positive, by Lemma~\ref{Lemma: separating wall condition} the wall $M_{\gamma,t}$ is separating and non-ghost if and only if $g_\gamma(t)=0$ and $v(\gamma)\in \Phi_-$ (that is to say $\gamma \in \Inv(v)$). Note that the existence of such a time $t$ implies that $g_\gamma(t_0)> 0$
    \begin{enumerate}
        \item If $n=\langle \lambda,\beta\rangle +1$ then $g_\gamma(t_0)=-t_0\langle \lambda+\beta^\vee,\gamma\rangle $ is positive if and only if $\langle \lambda+\beta^\vee,\gamma\rangle < 0$, if and only if $\gamma \in \Inv(u^{-1})$ where $u\in W^v$ is the minimal element such that $u^{-1}(\lambda+\beta^\vee)\in Y^{++}$.  However, by \cite[Lemma 3.7]{philippe2023grading} we have $\ell(vu)=\ell(v)+\ell(u)$, so by Lemma~\ref{Lemma: length additivity inversion sets}, $\Inv(v) \cap \Inv(u^{-1})=\emptyset$ and $M_{\gamma,t}$ can not be separating.

        \item If $n=-1$ then $g_\gamma(t_0)=-t_0\langle s_\beta(\lambda+\beta^\vee),\gamma\rangle$ is positive if and only if $\gamma \in \Inv(u^{-1}s_\beta)$ (with $u$ defined as above). By \cite[Lemma 3.7]{philippe2023grading} and Lemma~\ref{Lemma: length additivity inversion sets}, if $\gamma \in \Inv(u^{-1}s_\beta)$ then $\gamma \not \in \Inv(v)$, so $v(\gamma)$ is a positive root. Therefore by Lemma~\ref{Lemma: separating wall condition} the wall $M_{\gamma,t}$ is not separating.
    \end{enumerate}

    \item Suppose that $\langle \fc'_+(t),v(\gamma)\rangle>0$. Let us then show that, for any such $\gamma$ and $t\in [t_0,1)$, $g_\gamma(t)$ is not integral, unless $t=t_0$ and $\gamma=\beta$. Since $g_\gamma(1)$ is integral, it suffices to prove that 
    \begin{equation}
        |g_\gamma(t_0)-g_\gamma(1)|\leq 1,
    \end{equation}with equality if and only if $\gamma=\beta$.
    We have $|g_\gamma(t_0)-g_\gamma(1)|=|(1-t_0) \langle \fc'_+(t_0),v(\gamma)\rangle|$, and $t_0=\frac{|n|}{\langle \lambda+\beta^\vee\beta\rangle }$
    \begin{enumerate}
        \item If $n=-1$ then $\fc'_+(t_0)=-v(\lambda+\beta^\vee)$ so $\langle \fc'_+(t),v(\gamma)\rangle=-\langle \lambda+\beta^\vee,\gamma\rangle$ is positive if and only if $\langle \lambda+\beta^\vee,\gamma\rangle<0$. Since $\lambda+\beta^\vee$ is almost dominant and $\gamma\in \Phi_+$, this is equivalent to $\langle \lambda+\beta^\vee,\gamma\rangle =-1$, in which case $$|(1-t_0) \langle \fc'_+(t_0),v(\gamma)\rangle|=1-t_0<1.$$

        \item If $n=\langle \lambda,\beta\rangle +1$ then $\fc'_+(t_0)=-vs_\beta(\lambda+\beta^\vee)$ so if $\langle \fc'_+(t),v(\gamma)\rangle=-\langle \lambda+\beta^\vee,s_\beta(\gamma)\rangle>0$ then it is equal to $1$ unless $\gamma\in \Inv(s_\beta)$. If it is equal to $1$, then $|g_\gamma(t_0)-g_\gamma(1)|=1-t_0<1$. If $\gamma \in \Inv(s_\beta)\setminus \{\beta\}$, and since $(1-t_0)=\frac{1}{\langle \lambda+\beta^\vee,\beta\rangle}$, we have $|g_\gamma(t_0)-g_\gamma(1)|=\frac{|\langle \lambda+\beta^\vee,s_\beta(\gamma)\rangle|}{\langle \lambda+\beta^\vee,\beta\rangle}$. This is equal to $1$ if $\gamma=\beta$. Else, since $\beta$ is quantum, we have $|\langle \lambda+\beta^\vee,s_\beta(\gamma)\rangle| =\langle \lambda+\beta^\vee,\beta\rangle -\langle \lambda+\beta^\vee,\gamma\rangle <\langle \lambda+\beta^\vee,\beta\rangle$ so $|g_\gamma(t_0)-g_\gamma(1)|<1$.
    \end{enumerate}
\end{enumerate}
We have therefore proven that there is no non-ghost separating wall on $[t_0,1)$, except $M_{\beta,t_0}$. Therefore $T(\underline{\fc})=\{0,t_0,1\}$. It remains to prove that $R^{loc,t_0}_{\bx,\underline\fc}=X-1$. 

Let $\mathbf g$ be a minimal gallery, in $\cT^-_{\fc(t_0)}(\A)$, from $C^\infty_{\fc(t_0)}$ to the decorating alcove $D^+_{t_0}$ of $\underline{\fc}$ at time $t_0$, and let $\mathbf i$ be its type. In order to compute $R^{loc,t_0}_{\bx,\underline{\fc}}$ we will prove that there is a unique $D^-_{t_0}$-centrifugally folded gallery $\tilde{\mathbf g}$ of type $\mathbf i$ from $C^\infty_{\fc(t_0)}$ to $D^+_{\bx,1}$ (which is the unique alcove at codistance $w^-_{\bx,t_0}$ from $D^-_{\bx,t_0}=D^-_{t_0}$), and that it satisfies $r(\tilde{\mathbf g})=1$, $l(\tilde{\mathbf g})=0$. Since $\mathbf g$ is minimal, it is of the form $\mathbf g_0 \mathbf g_1$ where $\mathbf g_0$ goes from $C^\infty_{\fc(t_0)}$ to $\pr_{\fc_+(t_0)}(C^\infty_{\fc(t_0)})$ and where every alcove of $\mathbf g_1$ dominates $\fc_+(t_0)$. Since $M_{v(\beta),t_0}$ is the only affine wall separating $C^\infty_{\fc(t_0)}$ from $\fc_+(t_0)$, it is the unique affine wall along which $\mathbf g_0$ can be folded, let $\tilde{\mathbf g}_0$ denote the associated folded gallery. Then, since $d^+(C^\infty_{\fc(t_0)},D^+_{\bx,t_0})\in W_{(\lambda+\beta^\vee)^{++}}$, any gallery of type $\mathbf i$ from $C^\infty_{\fc(t_0)}$ to $D^+_{\bx,t_0}$ is of the form $\tilde{\mathbf g}_0 \tilde{\mathbf g}_1$ with $\tilde{\mathbf g}_1$ of type $\mathbf i_1$. Moreover, if we suppose that $\tilde{\mathbf g}$ is $D^-_{\bx,t_0}$-centrifugally folded, since $D^-_{\bx,t_0}$ and $D^+_{\bx,t_0}$ are on the same side of every wall dominating $\fs_{\pr^{Y^+}(\bx),+}(t_0)$, the gallery $ \tilde{\mathbf g}_1$ can not be folded (else the last folding would be done towards $D^+_{\bx,t_0}$, hence towards $D^-_{\bx,t_0}$). We deduce that there is a unique possible choice of such $\tilde{\mathbf g}$, and since $\tilde{\mathbf g}_1$ is minimal it crosses every wall towards its last alcove, hence also towards $D^-_{\bx,t_0}$. Therefore it satisfies $r(\tilde{\mathbf g})=1$, $l(\tilde{\mathbf g})=0$. Thus, $R^{loc,t_0}_{\bx,\underline{\fc}}=X-1$, which concludes the proof of the lemma.
\end{proof}

We now compute the $R$-polynomials associated to covers. We will use the classification of covers which was obtained by Welch in the affine ADE case (see~\cite[Theorem 2]{welch2022classification}), and extended to the general Kac-Moody case by the second author (see \cite[Proposition 3.20]{philippe2023grading}).

    \begin{Proposition}\label{Proposition Rpolynomialscoversregular}
        Let $\bx,\by \in W^+$ and suppose that $\bx\in W^+_{sph}$. Then
        \begin{equation}
            \by \lhd \bx \implies R_{\bx,\by}=X-1.
        \end{equation}
        
    \end{Proposition}
  \begin{proof}
     
We use the classification of covers for the affine Bruhat order given in \cite[Lemma 3.1]{hebert2024quantum} (note that, in this proposition $\bx$ covers $\by$, whereas it is the opposite in \cite{hebert2024quantum}).
    We separate the computation in 5 cases, according to the form of the cover:  \begin{equation}
        \by=\qp^{v\lambda}w \text{ and }\bx=s_{v(\beta)[n]}\by=\qp^{v(s_\beta \lambda+n\beta^\vee)}s_{v(\beta)}w,
    \end{equation}  with $\lambda\in Y^{++}$, $v\in W^\lambda$, $w\in W^v$ and $\beta[n]\in \Phi_+^a$. 
    
    The integer $n$ may take four values, namely $\{0,\langle \lambda,\beta\rangle, -1,\langle\lambda,\beta\rangle +1\}$. Moreover, an element $\underline{\fc}\in \cC^\infty_\bx(\by)$  is obtained from the open segment $\underline\fs_\bx$ by folding along the affine root $\beta[n]$. Indeed, if $\underline{\fc}\in \cC^\infty_{\bx}(\by)$, then $\underline{\fc}$ is obtained from $\underline{\fs}_\bx$ via some folding data  $(t_i,\beta_i[n_i])_{i\in \llbracket 0,r\rrbracket}$. By Proposition~\ref{Proposition: open paths give chains}, we have $\by=s_{\beta_r[n_r]}\ldots s_{\beta_0[n_0]}\bx$. By definition of a cover, we thus have $r=0$ and $\beta_0[n_0]=v(\beta)[n]$. Therefore it is unique, unless $\langle \lambda,\beta\rangle=n=0$ (in which case we can fold either at time $0$ or $1$, but we will see that only one choice contributes to $R_{\bx,\by}$).
    \begin{enumerate}
        
        \item Suppose that $n=\langle\lambda,\beta\rangle\neq0$, so that $\bx=\qp^{v\lambda}s_{v(\beta)}w$ and $\ell(v^{-1}s_\beta w)=\ell(v^{-1}w)+1$. In this case the only difference between $\underline{\fc}$ and $\underline{\fs}_\bx$ is the end alcove. Therefore $R^{loc,t}_{\bx,\underline{\fc}}=R^{loc,t}_{\bx,\underline{\fs}_\bx}=1$ for all $t<1$. We need to check that $R^{loc,1}_{\bx,\underline{\fc}}=X-1$. Let $\mathbf i$ be the type of a minimal gallery from $C^{++}_{-v\lambda}$ to $C_\bx$. Since $C^{++}_{-v\lambda}$ dominates the segment germ $\fs_{v\lambda,-}(1)$ and by definition of $D^-_{\bx,1}$, we can write $\mathbf i=\mathbf i_1 \mathbf i_2$ where $\mathbf i_1$ is the type of a minimal gallery (let us denote it $\mathbf g_1$) from $C^{++}_{-v\lambda}$ to $D^-_{\bx,1}$ and $\mathbf i_2$ is the type of a minimal gallery from $D^-_{\bx,1}$ to $C_\bx$. Therefore if $\mathbf g_2$ is a gallery of type $\mathbf i_2$ from $D^-_{\bx,1}$ to $C_\by$  then concatenating it with $\mathbf g_1$ we obtain a gallery from $C^{++}_{-v\lambda}$ to $C_\by$ of type $\mathbf i$. 
        Since (by Lemma~\ref{l_explicit_C++}) \[\ell(v^{-1}w)=\ell(d^+(C^{++}_{-v\lambda},C_\by))=\ell(v^{-1}s_\beta w)-1=\ell(d^+(C^{++}_{-v\lambda},C_\bx))-1,\] this gallery necessarily folds exactly once, and does not cross any wall towards $C^{++}_{-v\lambda}$, so $\ell(\mathbf g_2)=0$. Moreover the folding is $C^\infty_{-v\lambda}$-centrifuged as it folds towards $C^{++}_{-v\lambda}$ , so $r(\mathbf g_2)=1$. Hence $R_{\bx,\by}=R^{loc,1}_{\bx,\underline{\fc}}=X-1$.

        \item Suppose that $n=0\neq \langle \lambda,\beta\rangle$, then the unique element of $\cC^\infty_\bx(\by)$ is obtained from folding $\underline{\fs}_\bx$ at time $0$ along $M_{v(\beta)[0]}$. Therefore it is just the open segment $\underline\fs_\by$, and $T(\underline{\fs}_\by)=\{0,1\}$. By \cite[Lemma 3.1]{hebert2024quantum}, $\bx=\qp^{vs_\beta(\lambda)}s_{v(\beta)}w$ and $\ell(vs_\beta)=\ell(v)-1$.  
        \begin{itemize}
            \item To compute $R^{loc,1}_D$, let us fix a reduced expression $\mathbf i$ of $w^-_{\bx,1}=d^-(D^-_{\bx,1},C_\bx)$. Since there is no folding at time $1$, we have $d^-(D^-_{\by,1},C_\by)=w^-_{\by,1}=w^-_{\bx,1}$,so there is a unique gallery $\mathbf g$ of type $\mathbf i$ from $D^-_{\by,1}$ to $C_\by$, it is minimal, and does not cross any wall containing $\fc_-(1)$ since $D^-_{\by,1}=\pr_{\fc_-(1)}(C_\by)$. Moreover $C^{++}_{-v\lambda}$ and $D^-_{\by,1}$ are separated only by walls containing $\fc_-(1)$ since they both dominate this segment germ, so any wall crossed by $\mathbf g$ is crossed away from $C^{++}_{-v\lambda}$. Hence $l(\mathbf g)=0$ and $R^{loc,1}_{\bx,\underline{\fc}}=1$.

             \item At $t=0$, we have that $D^-_0=C_0$ is opposite (in $\cT^\pm_{0_\A}(\A)$) to $C^\infty_{0_\A}$. Therefore by Remark~\ref{r_usual_KL_polynomials}, the polynomial $R^{loc,0}_{\bx,\underline{\fc}}$ is a classical $R$-polynomial, and to prove that it is equal to $X-1$, it suffices to check that $w^\infty_{\underline{\fc},0}$ is a cover of $w^-_{\bx,0}$. We thus need to express these two elements.  Since $C^{++}_{-v\lambda}$ (which is the opposite of $C^\infty_{-v\lambda}$) and $D^-_{\by,1}$ both dominate the segment germ $\fs_{v\lambda,-}(1)$ the codistance $d^\ast(C^\infty_{-v\lambda},D^-_{\by,1})=:\tilde w$ lies in $W_\lambda$. Explicitly, we therefore have  $C^\infty_{-v\lambda}=a(-v\lambda,-v)$, $D^-_{\by,1}=a(-v\lambda,+v\tilde w)$ and $D^-_{\bx,1}=s_{v(\beta)[0]}D^-_{\by,1}=a(-vs_\beta\lambda,+vs_\beta \tilde w)$. Since there is no folding after time $t=0$, by the admissibility condition we deduce that $D^+_{\by,0}=a(0_\A,-v\tilde w w_{\lambda,0})$ and $D^+_{\bx,0}=a(0_\A,-vs_\beta \tilde ww_{\lambda,0})$, where $w_{\lambda,0}$ is the maximal element of the standard parabolic subgroup $W_\lambda$. Therefore: \begin{equation}
             w^-_{\bx,0}=vs_\beta \tilde w w_{\lambda,0} \ \text{and} \ w^\infty_{\underline{\fc},0}=v\tilde w w_{\lambda,0}.\end{equation}
             
             In particular, $w^\infty_{\underline{\fc},0}=s_{v(\beta)}w^-_{\bx,0}$, so the two elements are comparable in $W^v$. Let us write $vs_\beta= u^J u_J$ with $u^J\in W^\lambda$. By \cite[Proposition 3.20.2c]{philippe2023grading}, $a(-v\lambda,vu_J^{-1})$ is on a minimal gallery from $C^{++}_{-v\lambda}=a(-v\lambda,v)$ to $C_\by$ so, since $a(-v\lambda,v\tilde w)=D^-_{\by,1}$ is the projection of $C_\by$ on $\fs_{v\lambda,-}(1)$ and $u_J\in W_\lambda$, the alcove $a(-v\lambda,vu_J^{-1})$ is on a minimal gallery from $C^{++}_{-v\lambda}$ to $D^-_{\by,1}$. In term of distance, this implies that $\ell(\tilde w)=\ell(u_J)+\ell(u_J\tilde w)$. Recall that for any element $r\in W^\lambda$, $t\in W_\lambda$, we have $\ell(rtw_{\lambda,0})=\ell(r)+\ell(tw_{\lambda,0})= \ell(r)+\ell(w_{\lambda,0})-\ell(t)$. We deduce:
             \begin{align*}
                 \ell(w^\infty_{\underline{\fc},0})=\ell(v)+\ell(w_{\lambda,0})-\ell(\tilde w), \\
                 \ell(w^-_{\bx,0})=\ell(u^J)+\ell(w_{\lambda,0})-\ell(u_J\tilde w)=\ell(u^J)+\ell(w_{\lambda,0})-\ell(\tilde w)+\ell(u_J).
             \end{align*} Since $\ell(u^J)+\ell(u_J)=\ell(vs_\beta)=\ell(v)-1$, we deduce that $\ell(w^\infty_{\underline{\fc},0})=\ell(w^-_{\bx,0})+1$, so $w^\infty_{\underline{\fc},0}$ covers $w^-_{\bx,0}$ and,  by Remark~\ref{r_usual_KL_polynomials}, $R^{loc,0}_{\bx,\underline{\fc}}=X-1$. We conclude that $R_{\bx,\by}=X-1$ in this case.

        \end{itemize} 
         \item Suppose that $n=0=\langle \lambda,\beta\rangle$, then there are two open paths of type $\bx$ ending at $C_\by$, obtained folding along $M_{v(\beta)[0]}$ at time either $0$ or $1$. Let $\underline{\fc}$ denote the open path obtained by folding at time $0$, and $\tilde{\underline{\fc}}$ the one obtained by folding at time $1$. Note that $D^-_{\tilde{\underline{\fc}},1}=D^-_{\bx,1}$. However, $M_{\beta[0]}$ contains the segment germ $\fs_{v\lambda,-}(1)$ and $D^-_{\bx,1}$ is the projection of $C_\bx$ on this segment germ, so $d^+(D^-_{\bx,1},C_\by)=d^+(s_{\beta[0]}D^-_{\bx,1},C_\bx)>d^+(D^-_{\bx,1},C_\bx)$. Therefore there is no gallery with type a reduced expression of $w^-_{\bx,1}$ from $D^-_{\bx,1}$ to $C_\by$. Hence $R^{loc,1}_{\bx,\tilde{\underline{\fc}}}=0$ and only the first open path $\underline{\fc}$ may have a non-zero contribution to $R_{\bx,\by}$.
         
         Note that $\underline{\fc}=\underline{\fs}_\by$. The computation of $R_{\bx,\underline{\fc}}$ is similar to the previous case: $T(\underline{\fc})=\{0,1\}$ and we compute that $R^{loc,1}_{\bx,\underline{\fc}}=1$. At time $t=0$, the computation of $w^-_{\bx,0}$ and $w^\infty_{\underline{\fc},0}$ is slightly different: we are in a specific case of \cite[Proposition 3.20.1]{philippe2023grading}, so we have $\ell(s_\beta v^{-1}w)=\ell(v^{-1}w)+1$ (where, recall, $\by=\qp^{v\lambda}w$ and $\bx=s_{v(\beta)[0]}\by$). The decorating alcove $D^-_{\by,1}$ dominates $\fs_{v\lambda,-}(1)$ hence is of the form $a(-v\lambda,+v\tilde w)$ for some $\tilde{w}\in W_\lambda$, and since it is the projection of $C_\by=a(-v\lambda,w)$ on this segment germ, any minimal gallery from an alcove dominating it to $C_\by$ goes through $D^-_{\by,1}$. Therefore \begin{align}\ell(v^{-1}w)=\ell(d^+(C^{++}_{-v\lambda},C_\by))&=\ell(d^+(C^{++}_{-v\lambda},D^-_{\by,1}))+\ell(d^+(D^-_{\by,1},C_\by)) \\
         &=\ell(\tilde w)+\ell(d^+(D^-_{\by,1},C_\by)).\end{align}

    Similarly for $\underline{\fs}_\bx$, since $D^-_{\bx,1}=s_{v(\beta)[0]}D^-_{\by,1}=a(-v\lambda,vs_\beta\tilde w)$ and by $W^+$ invariance of $d^+$, we have
    \begin{align}\ell(s_\beta v^{-1}w)=\ell(d^+(C^{++}_{-v\lambda},C_\bx))&=\ell(d^+(C^{++}_{-v\lambda},D^-_{\bx,1}))+\ell(d^+(D^-_{\bx,1},C_\bx)) \\
         &=\ell(s_\beta\tilde w)+\ell(d^+(D^-_{\by,1},C_\by)).\end{align}
         Hence from $\ell(s_\beta v^{-1}w)=\ell(v^{-1}\beta)+1$ we deduce that $\ell(s_\beta\tilde w)=\ell(\tilde w)+1$, or equivalently that $\ell(s_\beta \tilde w w_{\lambda,0})=\ell(w_{\lambda,0})-\ell(s_\beta \tilde w)=\ell(w_{\lambda,0})-\ell(\tilde{w})-1=\ell(\tilde w w_{\lambda,0})-1$.
         Moreover $D^+_{\by,0}=a(0_\A,-v\tilde ww_{\lambda,0})$ and $D^+_{\bx,0}=a(0_\A,-vs_\beta \tilde w w_{\lambda,0})$, so $w^\infty_{\underline{\fc},0}=v\tilde w w_{\lambda,0}$ covers $w^-_{\bx,0}=vs_\beta \tilde w w_{\lambda,0}$ and we conclude that that $R_{\bx,\by}=R^{loc,0}_{\bx,\underline{\fc}}=X-1$ exactly as in case 2. 
         
    \item If $n\in \{-1,\langle \lambda,\beta\rangle +1\}$, then by Lemma~\ref{Lemma: Rpol distinct coweight case} the unique $C_\infty$-Hecke $\underline{\fc}$ open path of type $\bx$ with end alcove $C_\by$ verifies $T(\underline{\fc})=\{0,t_0,1\}$ with $t_0=\frac{|n|}{\langle \lambda+\beta^\vee,\beta\rangle}$ and $R^{loc,t_0}_{\bx,\fc}=X-1$. Moreover since $\underline{\fc}$ and $\underline{\fs}_\bx$ agree on $[0,t_0)$ and $R^{loc,0}_{\bx,\underline{\fs}_\bx}=1$, we have $R^{loc,0}_{\bx,\underline{\fc}}=1$. It remains to verify that $R^{loc,1}_{\bx,\underline{\fc}}=1$. Since there is no folding at time $1$, $R^{loc,1}_{\bx,\underline{\fc}}=X^l$  where $l$ is the number of walls crossed away from $C^\infty_{-v\lambda}=a(-v\lambda,-v)$ (or equivalently, towards $C^{++}_{-v\lambda}=a(-v\lambda,+v)$) in a given minimal gallery from $s_{v(\beta)[n]}D^-_{\bx,1}$ to $C_\by$. Let $\mu=\pr^{Y^+}(\bx)\in \{v(\lambda+\beta^\vee),vs_\beta(\lambda+\beta^\vee)\}$ and let $\tilde v$ denote the minimal element such that $\mu=\tilde v (\mu^{++})$, so $\tilde v\in W^{\mu^{++}}$. Then $C^{++}_{-\mu}=a(-\mu,+\tilde v)$ contains the segment germ $\fs_{\mu,-}(1)$, so a minimal gallery from this alcove to $C_\bx$ goes through $D^-_{\bx,1}=\pr_{\fs_{\mu,-}(1)}(C_\bx)$. Therefore, a minimal gallery from $s_{v(\beta)[n]}C^{++}_{-\mu}$ to $C_\by=s_{v(\beta)[n]}C_\bx$ goes through $D^-_{\underline{\fc},1}=s_{v(\beta)[n]}D^-_{\bx,1}$. Moreover, by \cite[Proposition 3.20.3.d]{philippe2023grading}, a minimal gallery from $C^{++}_{-v\lambda}$ to $C_\by$ goes through $s_{v(\beta)[n]}C^{++}_{-\mu}$, hence also through $s_{v(\beta)[n]}D^-_{\bx,1}$. We deduce that a minimal gallery from $s_{v(\beta)[n]}D^-_{\bx,1}$ to $C_\by$ only crosses walls away from $C^{++}_{-v\lambda}$, hence $l=0$ and $R^{loc,1}_{\bx,\underline{\fc}}=1$.

    Therefore, in these cases aswell $R_{\bx,\by}=X-1$, which concludes the proof.
         \end{enumerate}
         \end{proof}

\printindex

 \bibliography{bibliographie.bib}

@book {abramenko2008buldings,
    AUTHOR = {Abramenko, Peter and Brown, Kenneth S.},
     TITLE = {Buildings},
    SERIES = {Graduate Texts in Mathematics},
    VOLUME = {248},
      NOTE = {Theory and applications},
 PUBLISHER = {Springer, New York},
      YEAR = {2008},
     PAGES = {xxii+747},
      ISBN = {978-0-387-78834-0},
   MRCLASS = {20E42 (20F55 20J06 51E24 51F15)},
  MRNUMBER = {2439729},
MRREVIEWER = {Ralf Koehl},
       DOI = {10.1007/978-0-387-78835-7},
       URL = {https://doi-org.acces.bibliotheque-diderot.fr/10.1007/978-0-387-78835-7},
}

@article {bardy2016iwahori,
    AUTHOR = {Bardy-Panse, Nicole and Gaussent, St\'ephane and Rousseau, Guy},
     TITLE = {Iwahori-{H}ecke algebras for {K}ac-{M}oody groups over local
              fields},
   JOURNAL = {Pacific J. Math.},
  FJOURNAL = {Pacific Journal of Mathematics},
    VOLUME = {285},
      YEAR = {2016},
    NUMBER = {1},
     PAGES = {1--61},
      ISSN = {0030-8730},
   MRCLASS = {20G44 (17B67 20C08 20E42 20G25 22E65 33D80 51E24)},
  MRNUMBER = {3554242},
MRREVIEWER = {Walter D. Freyn},
       DOI = {10.2140/pjm.2016.285.1},
       URL = {http://dx.doi.org/10.2140/pjm.2016.285.1},
}

@article{bardy2021structure,
     author = {Bardy-Panse, Nicole and Rousseau, Guy},
     title = {On structure constants of {Iwahori{\textendash}Hecke} algebras for {Kac{\textendash}Moody} groups},
     journal = {Algebraic Combinatorics},
     pages = {465--490},
     publisher = {MathOA foundation},
     volume = {4},
     number = {3},
     year = {2021},
     doi = {10.5802/alco.163},
     language = {en},
     url = {http://www.numdam.org/articles/10.5802/alco.163/}
}

@article{bardy2025twin,
 author = {Bardy-Panse, Nicole and H{\'e}bert, Auguste and Rousseau, Guy},
 title = {Twin masures associated with {Kac}-{Moody} groups over {Laurent} polynomials},
 fjournal = {Annals of Representation Theory},
 journal = {Ann. Represent. Theory},
 issn = {2704-2081},
 volume = {2},
 number = {3},
 pages = {281--353},
 year = {2025},
 language = {English},
 doi = {10.5802/art.25},
 zbMATH = {8032257},
 Zbl = {1562.20060}
}

@book {bjorner2005combinatorics,
    AUTHOR = {Bj\"{o}rner, Anders and Brenti, Francesco},
     TITLE = {Combinatorics of {C}oxeter groups},
    SERIES = {Graduate Texts in Mathematics},
    VOLUME = {231},
 PUBLISHER = {Springer, New York},
      YEAR = {2005},
     PAGES = {xiv+363},
      ISBN = {978-3540-442387; 3-540-44238-3},
   MRCLASS = {05-01 (05E15 20F55)},
  MRNUMBER = {2133266},
MRREVIEWER = {Jian-yi Shi},
}

@article{braverman2014affine,
  title={An affine {G}indikin-{K}arpelevich formula},
  author={Braverman, Alexander and Garland, Howard and Kazhdan, David and Patnaik, Manish},
  journal={Perspectives in Representation Theory (Yale University, May 12--17, 2012)(P. Etingof, M. Khovanov, and A. Savage, eds.), Contemp. Math},
  volume={610},
  pages={43--64},
  year={2014}
}

@article {braverman2016iwahori,
    AUTHOR = {Braverman, Alexander and Kazhdan, David and Patnaik, Manish
              M.},
     TITLE = {Iwahori-{H}ecke algebras for {$p$}-adic loop groups},
   JOURNAL = {Invent. Math.},
  FJOURNAL = {Inventiones Mathematicae},
    VOLUME = {204},
      YEAR = {2016},
    NUMBER = {2},
     PAGES = {347--442},
      ISSN = {0020-9910},
   MRCLASS = {20G44 (11S85 20C08 20G25)},
  MRNUMBER = {3489701},
MRREVIEWER = {Dongwen Liu},
       DOI = {10.1007/s00222-015-0612-x},
       URL = {https://doi-org.acces.bibliotheque-diderot.fr/10.1007/s00222-015-0612-x},
}

@article{bruhat1972groupes,
  title={Groupes r{\'e}ductifs sur un corps local},
  author={Bruhat, Fran{\c{c}}ois and Tits, Jacques},
  journal={Publications Math{\'e}matiques de l'IH{\'E}S},
  volume={41},
  number={1},
  pages={5--251},
  year={1972},
  publisher={Springer}
}

@article{eberhardt2023standard,
      title={Standard Extension Algebras {I}: Perverse Sheaves and {F}ukaya Calculus}, 
      author={Jens Niklas Eberhardt and Catharina Stroppel},
      year={2023},
      eprint={2310.09206},
      archivePrefix={arXiv},
      primaryClass={math.RT}
}

@inproceedings{gaussent2008kac,
  title={Kac-{M}oody groups, hovels and {L}ittelmann paths},
  author={Gaussent, St{\'e}phane and Rousseau, Guy},
  booktitle={Annales de l'institut Fourier},
  volume={58},
  number={7},
  pages={2605--2657},
  year={2008}
}

@article{gaussent2014spherical,
  title={Spherical {H}ecke algebras for {K}ac-{M}oody groups over local fields},
  author={Gaussent, St{\'e}phane and Rousseau, Guy},
  journal={Annals of Mathematics},
  volume={180},
  number={3},
  pages={1051--1087},
  year={2014},
  publisher={Princeton University and the Institute for Advanced Study}
}

@incollection {haddad1985coxeter,
    AUTHOR = {Haddad, Ziad},
     TITLE = {A {C}oxeter group approach to {S}chubert varieties},
 BOOKTITLE = {Infinite-dimensional groups with applications ({B}erkeley,
              {C}alif., 1984)},
    SERIES = {Math. Sci. Res. Inst. Publ.},
    VOLUME = {4},
     PAGES = {157--165},
 PUBLISHER = {Springer, New York},
      YEAR = {1985},
      ISBN = {0-387-96216-6},
   MRCLASS = {22E65 (14M15 17B67 20H15)},
  MRNUMBER = {823319},
MRREVIEWER = {James\ E.\ Humphreys},
       DOI = {10.1007/978-1-4612-1104-4\_6},
       URL = {https://doi.org/10.1007/978-1-4612-1104-4_6},
}

@article {hebert2020new,
    AUTHOR = {H{\'{e}}bert, Auguste},
     TITLE = {A {N}ew {A}xiomatics for {M}asures},
   JOURNAL = {Canad. J. Math.},
  FJOURNAL = {Canadian Journal of Mathematics. Journal Canadien de
              Math\'{e}matiques},
    VOLUME = {72},
      YEAR = {2020},
    NUMBER = {3},
     PAGES = {732--773},
      ISSN = {0008-414X},
   MRCLASS = {20E42 (20G44)},
  MRNUMBER = {4098599},
       DOI = {10.4153/s0008414x19000051},
       URL = {https://doi-org.acces.bibliotheque-diderot.fr/10.4153/s0008414x19000051},
}

@article{hebert2022new,
url = {https://doi.org/10.1515/advgeom-2022-0013},
title = {A new axiomatics for masures {II}},
author = {H{\'e}bert, Auguste},
pages = {513--522},
volume = {22},
number = {4},
journal = {Advances in Geometry},
doi = {doi:10.1515/advgeom-2022-0013},
year = {2022},
lastchecked = {2022-10-19}
}

@article{hebert2024quantum,
  title={Quantum roots for {K}ac-{M}oody root systems and finiteness properties of the {K}ac-{M}oody affine {B}ruhat order},
  author={H{\'e}bert, Auguste and Philippe, Paul},
  journal={arXiv preprint arXiv:2405.12559},
  year={2024}
}

@article {iwahori1965bruhat,
    AUTHOR = {Iwahori, N. and Matsumoto, H.},
     TITLE = {On some {B}ruhat decomposition and the structure of the
              {H}ecke rings of {${ p}$}-adic {C}hevalley groups},
   JOURNAL = {Inst. Hautes \'Etudes Sci. Publ. Math.},
  FJOURNAL = {Institut des Hautes \'Etudes Scientifiques. Publications
              Math\'ematiques},
    NUMBER = {25},
      YEAR = {1965},
     PAGES = {5--48},
      ISSN = {0073-8301},
   MRCLASS = {20.70 (14.50)},
  MRNUMBER = {0185016},
MRREVIEWER = {Rimhak Ree},
       URL = {http://www.numdam.org/item?id=PMIHES_1965__25__5_0},
}

@article{kazhdan1979representations,
  title={Representations of {C}oxeter groups and {H}ecke algebras},
  author={Kazhdan, David and Lusztig, George},
  journal={Inventiones mathematicae},
  volume={53},
  number={2},
  pages={165--184},
  year={1979},
  publisher={Springer}
}

@incollection {kazhdan1980schubert,
    AUTHOR = {Kazhdan, David and Lusztig, George},
     TITLE = {Schubert varieties and {P}oincar\'{e} duality},
 BOOKTITLE = {Geometry of the {L}aplace operator ({P}roc. {S}ympos. {P}ure
              {M}ath., {U}niv. {H}awaii, {H}onolulu, {H}awaii, 1979)},
    SERIES = {Proc. Sympos. Pure Math.},
    VOLUME = {XXXVI},
     PAGES = {185--203},
 PUBLISHER = {Amer. Math. Soc., Providence, RI},
      YEAR = {1980},
      ISBN = {0-8218-1439-7},
   MRCLASS = {14M15 (55N35 57R19)},
  MRNUMBER = {573434},
}

@book {kumar2002kac,
    AUTHOR = {Kumar, Shrawan},
     TITLE = {Kac-{M}oody groups, their flag varieties and representation
              theory},
    SERIES = {Progress in Mathematics},
    VOLUME = {204},
 PUBLISHER = {Birkh\"auser Boston, Inc., Boston, MA},
      YEAR = {2002},
     PAGES = {xvi+606},
      ISBN = {0-8176-4227-7},
   MRCLASS = {22E46 (14M15 17B67 22E65)},
  MRNUMBER = {1923198},
MRREVIEWER = {Guy Rousseau},
       URL = {https://doi-org.acces.bibliotheque-diderot.fr/10.1007/978-1-4612-0105-2},
}

@Book{marquis2018introduction,
 Author = {Marquis, Timoth{\'e}e},
 Title = {An introduction to {Kac}-{Moody} groups over fields},
 FSeries = {EMS Textbooks in Mathematics},
 Series = {EMS Textb. Math.},
 ISBN = {978-3-03719-187-3; 978-3-03719-687-8},
 Year = {2018},
 Publisher = {Z{\"u}rich: European Mathematical Society (EMS)},
 Language = {English},
 DOI = {10.4171/187},
 zbMATH = {6916409},
 Zbl = {1405.20003}
}

@article{milicevic2022gallery,
	Author = {Mili{\'c}evi{\'c}, Elizabeth and Naqvi, Yusra and Schwer, Petra and Thomas, Anne},
	Da = {2022/06/13},
	Doi = {10.1007/s00031-022-09726-8},
	Id = {Mili{\'c}evi{\'c}2022},
	Isbn = {1531-586X},
	Journal = {Transformation Groups},
	Title = {A Gallery Model for Affine Flag Varieties via Chimney Retractions},
	Ty = {JOUR},
	Url = {https://doi.org/10.1007/s00031-022-09726-8},
	Year = {2022}}

@article{muthiah2019double,
  title={Double-affine {K}azhdan-{L}usztig polynomials via masures},
  author={Muthiah, Dinakar},
  journal={arXiv preprint arXiv:1910.13694},
  year={2019}
}

@article {muthiah2018iwahori,
    AUTHOR = {Muthiah, Dinakar},
     TITLE = {On {I}wahori-{H}ecke algebras for {$p$}-adic loop groups:
              double coset basis and {B}ruhat order},
   JOURNAL = {Amer. J. Math.},
  FJOURNAL = {American Journal of Mathematics},
    VOLUME = {140},
      YEAR = {2018},
    NUMBER = {1},
     PAGES = {221--244},
      ISSN = {0002-9327},
   MRCLASS = {20C08 (22E40 22E67)},
  MRNUMBER = {3749194},
MRREVIEWER = {Maarten Sander Solleveld},
       DOI = {10.1353/ajm.2018.0004},
       URL = {https://doi-org.acces.bibliotheque-diderot.fr/10.1353/ajm.2018.0004},
}

@article {muthiah2019bruhat,
    AUTHOR = {Muthiah, Dinakar and Orr, Daniel},
     TITLE = {On the double-affine {B}ruhat order: the {$\varepsilon=1$}
              conjecture and classification of covers in {ADE} type},
   JOURNAL = {Algebr. Comb.},
  FJOURNAL = {Algebraic Combinatorics},
    VOLUME = {2},
      YEAR = {2019},
    NUMBER = {2},
     PAGES = {197--216},
   MRCLASS = {20C08 (22E67)},
  MRNUMBER = {3934828},
}

@article{patnaik2024local,
author = "Patnaik, Manish M.",
title = "Local {B}irkhoff decompositions for loop groups and double affine {R}-polynomials (available on Manish {P}atnaik's website)",
url = "http://www.ualberta.ca/~patnaik/NewBirkhoff.pdf",
year = "2024",

}

@article {remy2002groupes,
    AUTHOR = {R{\'e}my, Bertrand},
     TITLE = {Groupes de {K}ac-{M}oody d\'eploy\'es et presque d\'eploy\'es},
   JOURNAL = {Ast\'erisque},
  FJOURNAL = {Ast\'erisque},
    NUMBER = {277},
      YEAR = {2002},
     PAGES = {viii+348},
      ISSN = {0303-1179},
   MRCLASS = {20E42 (17B67 20G15 51E24)},
  MRNUMBER = {1909671},
MRREVIEWER = {Guy Rousseau},
}

@article{ronan1989lectures,
  title={Lectures on buildings, volume 7 of Perspectives in Mathematics},
  author={Ronan, Mark},
  journal={Academic Press Inc., Boston, MA},
  volume={11},
  pages={12},
  year={1989}
}

@article{rousseau2011masures,
  title={Masures affines},
  author={Rousseau, Guy},
  journal={Pure and Applied Mathematics Quarterly},
  volume={7},
  number={3},
  pages={859--921},
  year={2011}
}

@article {rousseau2016groupes,
    AUTHOR = {Rousseau, Guy},
     TITLE = {Groupes de {K}ac-{M}oody d\'eploy\'es sur un corps local {II}.
              {M}asures ordonn\'ees},
   JOURNAL = {Bull. Soc. Math. France},
  FJOURNAL = {Bulletin de la Soci\'et\'e Math\'ematique de France},
    VOLUME = {144},
      YEAR = {2016},
    NUMBER = {4},
     PAGES = {613--692},
      ISSN = {0037-9484},
   MRCLASS = {20G44 (20E42 20G25)},
  MRNUMBER = {3562609},
}

@article{rousseau2017almost,
  title={{A}lmost split {K}ac–{M}oody groups over ultrametric fields},
  author={Rousseau, Guy},
  journal={Groups Geometry, and Dynamics},
  volume={11},
  pages={891--975},
  DOI = {10.4171/GGD/418},
  year={2017}
}

@article{schwer2022shadows,
	Author = {Schwer, Petra},
	Da = {2022/03/01},
	Doi = {10.1365/s13291-021-00244-2},
	Id = {Schwer2022},
	Isbn = {1869-7135},
	Journal = {Jahresbericht der Deutschen Mathematiker-Vereinigung},
	Number = {1},
	Pages = {3--41},
	Title = {Shadows in the Wild - Folded Galleries and Their Applications},
	Ty = {JOUR},
	Url = {https://doi.org/10.1365/s13291-021-00244-2},
	Volume = {124},
	Year = {2022}}

@article {tits1987uniqueness,
    AUTHOR = {Tits, Jacques},
     TITLE = {Uniqueness and presentation of {K}ac-{M}oody groups over
              fields},
   JOURNAL = {J. Algebra},
  FJOURNAL = {Journal of Algebra},
    VOLUME = {105},
      YEAR = {1987},
    NUMBER = {2},
     PAGES = {542--573},
      ISSN = {0021-8693},
   MRCLASS = {17B67 (17B20 20G15 20H10)},
  MRNUMBER = {873684},
MRREVIEWER = {James F. Hurley},
       DOI = {10.1016/0021-8693(87)90214-6},
       URL = {http://dx.doi.org/10.1016/0021-8693(87)90214-6},
}

@Article{welch2022classification,
 Author = {Welch, Amanda},
 Title = {Classification of cocovers in the double affine {Bruhat} order},
 FJournal = {The Electronic Journal of Combinatorics},
 Journal = {Electron. J. Comb.},
 ISSN = {1077-8926},
 Volume = {29},
 Number = {4},
 Year = {2022},
 Language = {English},
 DOI = {10.37236/10745},
 zbMATH = {7613662}
}

@article{philippe2023grading,
author = {Philippe, Paul},
year = {2024},
month = {12},
pages = {273-320},
title = {Grading of affinized {W}eyl semigroups of {K}ac–{M}oody type},
volume = {332},
journal = {Pacific Journal of Mathematics},
doi = {10.2140/pjm.2024.332.273}
}

\end{document}